\documentclass[12pt,a4paper]{amsart}
\usepackage{amsfonts}
\usepackage{amsthm}
\usepackage{amsmath}
\usepackage{amscd}
\usepackage[latin2]{inputenc}
\usepackage{t1enc}
\usepackage[mathscr]{eucal}
\usepackage{indentfirst}
\usepackage{graphicx}
\usepackage{graphics}
\numberwithin{equation}{section}
\usepackage[margin=2.9cm]{geometry}
\usepackage{epstopdf} 
\usepackage{xcolor}
\usepackage{comment}
\usepackage{enumitem} 
 
 \usepackage{amssymb}

\newtheorem{thm}{Theorem}[section]
\newtheorem{theorem}[thm]{Theorem}
\newtheorem{lemma}[thm]{Lemma}
\newtheorem{corollary}[thm]{Corollary}
\newtheorem{proposition}[thm]{Proposition}

\theoremstyle{definition}

\newtheorem{remark}[thm]{Remark}

\newtheorem{definition}[thm]{Definition}

\newtheorem{definition-thm}[thm]{Definition-Theorem}

\newcommand{\erf}{\operatorname{erf}}
\numberwithin{equation}{section}

\newcommand{\Tr}{\operatorname{Tr}}

\newcommand{\Id}{\operatorname{Id}}

\newcommand{\R}{\mathbb{R}}
\def\beq{\begin{eqnarray}}
	\def\eeq{\end{eqnarray}}
\newcommand{\hvff}{\textrm{hvff}}
\newcommand{\ff}{\textrm{ff}}
\newcommand{\sff}{\textrm{sf}}
\newcommand{\tf}{\textrm{tf}}
\newcommand{\td}{\textrm{td}}
\newcommand{\hvlf}{\textrm{hvlf}}
\newcommand{\hvrf}{\textrm{hvrf}}

\definecolor{pink}{rgb}{1,0,1}

\definecolor{purple}{rgb}{0.4,0.2,1}

\newcommand{\erfc}{\operatorname{erfc}}

\newcommand{\pa}{\partial}
\newcommand{\eps}{\varepsilon}

\newcommand{\Sone}{{\mathbf{S}}^1}

\newcommand{\cA}{{\mathcal{A}}}

\newcommand{\cH}{{\mathcal{H}}}

\newcommand{\cV}{\mathcal{V}}

\newcommand{\cC}{\mathcal{C}}
\newcommand{\cL}{\mathcal{L}}
\newcommand{\cO}{\mathcal{O}}

\begin{document}
	
	\title[Heat kernel on polygonal domains]{The heat kernel on curvilinear polygonal domains in surfaces}


	\author[M. Nursultanov]{Medet Nursultanov}
	\address{Department of Mathematics and Statistics \\University of Helsinki\\ Finland}
	 \email{medet.nursultanov@gmail.com}
		
		\author[J. Rowlett]{Julie Rowlett}
		\address{Mathematical Sciences \\ Chalmers University  and the University of Gothenburg \\ 412 96 Gothenburg, Sweden}
			\email{julie.rowlett@chalmers.se}

			\author[D. Sher]{David Sher}
			\address{Department of Mathematical Sciences\\  DePaul University \\ 2320 N Kenmore Ave \\ Chicago, IL 60614, USA}
				\email{dsher@depaul.edu}

	\begin{abstract}We construct the heat kernel on curvilinear polygonal domains in arbitrary surfaces for Dirichlet, Neumann, and Robin boundary conditions as well as mixed problems, including those of Zaremba type.  We compute the short time asymptotic expansion of the heat trace and apply this expansion to demonstrate a collection of results showing that corners are spectral invariants. 
	\end{abstract}
	
	

	\maketitle

	\section{Introduction} \label{s:intro} 
	If two compact Riemannian manifolds $(M, g)$ and $(M', g')$ are isospectral, meaning they have the same Laplace spectrum, then they need not be isometric. However, isospectrality does imply that $M$ and $M'$ both have the same dimension, $n$.  Moreover, they also have the same $n$-dimensional volume.  Thus, both dimension and volume are spectral invariants, in the sense that they are determined by the spectrum.  This fact follows from Weyl's law \cite{weyl}, proven over one hundred years ago.  It is natural to ask: what other geometric features are spectral invariants?
	
	The next geometric spectral invariant was discovered by Pleijel \cite{pleijel} some forty years after Weyl's law.  For an $n$-dimensional Riemannian manifold with smooth boundary, the $n-1$ dimensional volume of the boundary is a spectral invariant.  About ten years later, McKean and Singer \cite{mc-s} proved that certain curvature integrals are also spectral invariants.  For smooth surfaces and smoothly bounded planar domains, McKean \& Singer \cite{mc-s} and independently M. Kac \cite{kac} proved that the Euler characteristic is a spectral invariant. By the Gauss-Bonnet Theorem, this shows that the number of holes in a planar domain is a geometric spectral invariant for a planar domain.
	
	The tactic of both McKean \& Singer and Kac was to use the existence of a short time asymptotic expansion for the heat trace, together with the calculation of the coefficients in this expansion.  Recall that the \emph{heat kernel} $H_M(t,z,z')$ on a 
	
	Riemannian manifold $(M,g)$ 
	is the fundamental solution of the heat equation on $M$:
	\begin{equation}\begin{cases} (\partial_t+\Delta_z)H_M(t,z,z')=0;\\ \lim_{t\to 0}H_M(t,z,z')=\delta_z(z').\end{cases}\end{equation}
	As long as the eigenvalues are discrete and approach $\infty$ sufficiently quickly, which is the case in all geometric settings considered here, the \emph{heat trace} is the trace of this kernel, and satisfies
	\begin{equation}\Tr H_M(t)=\int_{M}H_M(t,z,z)\, dz=\sum_{j=1}^{\infty}e^{-\lambda_j t}.\end{equation}
	Above, $\lambda_j$ are the eigenvalues of the Laplacian $\Delta$ on $M$, arranged in increasing order. As a consequence, the heat trace is a spectral invariant. Therefore the coefficients in its asymptotic expansion as $t\to 0$ are also spectral invariants.   The existence and calculation of an asymptotic expansion for the heat trace is a powerful method for producing spectral invariants. This program has been carried out extensively both for smooth manifolds and for manifolds with boundary \cite{mc-s}.
	
	Here, we are interested in the heat kernel on curvilinear polygonal domains which are subsets of smooth surfaces.  This includes curvilinear polygonal domains in the plane, as well as more exotic non-planar examples.  We are interested in the heat kernel for such domains in part because it may allow us to determine new geometric spectral invariants.  Indeed, we show in \S 6 that in general, the presence or lack of vertices is a spectral invariant for Dirichlet, Neumann, Robin, and mixed boundary conditions.  Moreover, we shall see there that a jump in boundary condition is also a spectral invariant.  
	
	Let us now introduce our geometric setting.
	
	\begin{definition}\label{def:curvpoly} We say that $\Omega$ is a \em curvilinear polygonal domain \em if it is a compact subset of a smooth Riemannian surface $(M,g)$ with piecewise smooth boundary and a vertex at each non-smooth point of $\partial\Omega$. A \em vertex \em is a point $p$ on the boundary of $\Omega$ at which the following are satisfied.
		\begin{enumerate} 
			\item The boundary in a neighborhood of $p$ is defined by a continuous curve $\gamma(t): (-a, a) \to M$ for $a > 0$ with $\gamma(0) = p$.  We require that $\gamma$ is smooth {on $(-a,0]$ and $[0,a)$}, with $||\dot \gamma(t)|| =1$ for all $t \in (-a, a)$, and  such that 
			\[\lim_{t \uparrow 0} \dot \gamma (t) = v_1, \quad \lim_{t \downarrow 0} \dot \gamma  (t) = v_2,\]
			for some vectors $v_1,v_2\in T_{p}M$, with $- v_1 \neq v_2$.
			\item The \em interior angle \em at the point $p$ is the \em interior angle \em at that corner, which is the angle between the vectors $-v_1$ and $v_2$. 
		\end{enumerate}
		Note that requiring $-v_1$ and $v_2$ to be distinct means that the interior angle will be an element of $(0,2\pi)$, which rules out inward and outward pointing cusps. An angle of $\pi$ is allowed. 
	\end{definition} 
	
	A vertex in a curvilinear polygonal domain is an example of a conical singularity where the link is a one-dimensional manifold with boundary.  Moreover, it is a ``non-exact'' conical singularity in the sense that the curve $\gamma$ defining the boundary near a vertex may have non-zero geodesic curvature on the entire interval $(-a, a)$.  For curvilinear domains in the plane, this means that there need not be a neighborhood of the vertex in which the edges are straight.  This geometric setting is therefore not contained within the literature for either (1) conical singularities whose link is a compact manifold without boundary nor for (2) planar polygons for which the edges are straight near the vertices. 
	
	There is substantial work in the literature on heat trace expansions in the settings (1) and (2).  For conical singularities with no boundary on the link, a non-exhaustive list of works concerning the heat kernel and its trace is:  \cite{seeley1}, \cite{karol}, \cite{koko}, \cite{loya1}, \cite{loya2}, \cite{gil1}, \cite{gil-loya}, \cite{mooers}.  In the case (2) of vertices which locally have straight edges, a non-exhaustive list includes \cite{skuj}, \cite{nest}, \cite{kap}, and \cite{koz}. 
	
	For polygonal domains in the plane with the Dirichlet boundary condition, Fedosov showed in the 1960s that the vertices produce an extra term in the short time asymptotic expansion of the heat trace \cite{fedosov}, \cite{fedosov1}.  This term appears in the coefficient of $t^0$.  Its most simplified form and calculation can be found in a paper of van den Berg and Srisatkunarajah \cite{vdbs}, although the expression there is originally due to unpublished work of Ray, and is mentioned in both \cite{kac} and \cite{mc-s}.
	
	Although it has been widely assumed that analogous results for the heat trace expansion hold for curvilinear polygons, a rigorous proof even in the planar Dirichlet case was not given until \cite{corners}. Similar results hold for Neumann boundary conditions, see \cite{htap}.  Although Robin conditions have been studied on manifolds with boundary \cite{gilkey,zayed}, to our knowledge there is no work in the  literature about heat trace expansions with Robin conditions in the presence of corners of arbitrary angles, even in the plane.  For certain corner angles, however, we refer to the physical approach of \cite{a-dowker}. Outside the planar case, or even in the planar case with mixed boundary conditions, less is known.  For the mixed boundary condition, also known as Zaremba boundary condition, references include \cite{avram}, \cite{seeleyz}, \cite{jlnp}, \cite{lpt}.  
	
	Our geometric microlocal methods allow us to handle the general case of compact curvilinear polygonal domains in surfaces, with any combination of Dirichlet, Neumann, and/or Robin boundary conditions on the various smooth boundary components. 
	The sign convention for our Laplacian, in local coordinates, with respect to the Riemannian metric, $g$, on a surface is 
	$$\Delta = - \frac{1}{\sqrt{\det(g)}} \sum_{i,j=1} ^2 \pa_i \sqrt{\det(g)} g^{ij} \pa_j.$$
	Our convention for the Robin boundary condition on any portion of the boundary is: 
	$$\left . \frac{\pa u}{\pa \nu} \right|_{\pa \Omega} = \left . \kappa u \right|_{\pa \Omega}.$$
	Here, the derivative on the left is the \em inward \em pointing normal derivative, and therefore, on the right, $\kappa$ is a \em non-negative \em function. Under this condition the spectrum is non-negative. We assume throughout, for simplicity, that $\kappa$ is smooth.
	
	Our main result is:  
	
	\begin{theorem}\label{thm:mainthm} Let $\Omega$ be a curvilinear polygonal domain in a smooth surface with finitely many vertices $V_1,\ldots, V_n$ of angles $\alpha_1,\ldots,\alpha_n$. Define its edges $E_1,\ldots,E_n$ by letting $E_j$ be the segment of the boundary between $V_{j-1}$ and $V_j$, with subscripts taken mod $n$. Let $\mathcal E_D$, $\mathcal E_N$, and $\mathcal E_R$ be three disjoint sets whose union is $\{1,\dots,n\}$. For each $j\in\mathcal E_D$, $\mathcal E_N$, and $\mathcal E_R$, we impose Dirichlet, Neumann, and Robin conditions with parameter $\kappa_j(x)$, respectively, along $E_j$. Assume that all functions $\kappa_j(x)$ are non-negative and smooth.
		
		Let $\mathcal V_{=}$ be the set of $j$ for which vertex $V_j$ has either zero or two Dirichlet edges adjacent to it, i.e. either both $j$ and $j+1\in\mathcal E_D$ or neither are. Conversely, let $\mathcal V_{\neq}$ be the set of $j$ for which $V_j$ has exactly one adjacent Dirichlet edge. Also let $K(z)$ and $k_g(x)$ be the Gauss curvature and geodesic/mean curvature of $\Omega$ and $\partial\Omega$ respectively.
		
		Then the heat trace $\Tr H^{\Omega}(t)$ for the Laplacian with those boundary conditions, and with the Friedrichs extension at each vertex, 
		has a complete polyhomogeneous conormal expansion in $t$ as $t\to 0$. Moreover, the first few terms of this expansion have the form
		\[\Tr H^{\Omega}(t)=a_{-1}t^{-1}+a_{-1/2}t^{-1/2}+a_0+O(t^{1/2}\log t),\]
		where:
		\begin{gather}
			a_{-1}=\frac{A(\Omega)}{4\pi};\\
			a_{-1/2}=\frac{1}{8\sqrt{\pi}}(\sum_{j\notin\mathcal E_D}\ell(E_j)-\sum_{j\in\mathcal E_D}\ell(E_j));\\
			a_0=\frac{1}{12\pi}\int_{\Omega}K(z)\, dz+\frac{1}{12\pi}\int_{\partial\Omega}k_g(x)\, dx-\frac{1}{2\pi}\sum_{j\in\mathcal E_R}\int_{E_j}\kappa_j(x)\, dx\\
			+\sum_{j\in V_{=}}\frac{\pi^2-\alpha_j^2}{24\pi\alpha_j}+\sum_{j\in V_{\neq}}\frac{-\pi^2-2\alpha_j^2}{48\pi\alpha_j}.
		\end{gather}
	\end{theorem}
	\begin{remark}
		It is well-known that if the boundary of $\Omega$ is smooth then there are no logarithmic terms in the heat expansion. We do not characterize the nature of logarithmic terms in the expansion in our more general setting.
	\end{remark}
	
	The proof of this result contains several ingredients which may be of independent interest.  The main strategy is to use geometric microlocal analysis to construct the heat kernel on a heat space, denoted by $\Omega_h ^2$, which is created by blowing up $\Omega \times \Omega \times [0, 1)$ along various p-submanifolds. On this heat space we show, in Theorem \ref{thm:5point8}, that the heat kernel has a polyhomogeneous conormal expansion at every boundary hypersurface. Indeed we construct the heat kernel by solving suitable model problems at the various boundary hypersurfaces. This gives a full description of the heat kernel on a curvilinear polygonal domain in a surface, in all asymptotic regimes.  As such this construction is useful for any application in which fine structure information about the heat kernel near $t=0$ is needed.
	
	A major advantage of this method is that a complete asymptotic description of the heat \em kernel, \em rather than just its trace, is obtained. This allows precise asymptotic analysis for expressions such as the gradient of the heat kernel and is likely of interest for future work.
	
	The paper is organized as follows. In \S 2, we develop an integral representation of the heat kernel for infinite circular sectors with Dirichlet, Neumann, and mixed boundary conditions.  We do this by first obtaining an integral representation of the Green's function for the corresponding boundary condition.  Using functional calculus, we prove that the heat kernel is obtained by taking the inverse Laplace transform of the Green's function.  By the uniqueness of the heat kernel, we thereby obtain the equivalence of this integral representation of the heat kernel and the more common series representation of the heat kernel \cite{cheeger}.  In \S 3, we construct the heat spaces and demonstrate the composition rule for operators with polyhomogeneous conormal Schwartz kernels.  To construct the heat kernel, we proceed in \S 4 to solve the model problem for the smooth parts of the boundary for the Dirichlet, Neumann, and Robin boundary conditions.  In \S 5 we solve the model problem for the vertices with the various boundary conditions and combinations thereof.  In this way, we construct the heat kernel on a curvilinear polygonal domain in a surface.  In \S 6, we use this construction together with our integral representation of the heat kernels obtained in \S 2 to compute the heat trace and prove Theorem \ref{thm:mainthm}.  We conclude in \S 6 with applications of Theorem \ref{thm:mainthm} showing contexts in which corners (vertices) are spectral invariants.  
	
	\subsection*{Acknowledgments} The authors are deeply grateful to Daniel Grieser for his insightful comments on an early draft of this manuscript, and would also like to thank F\'elix Houde, Rafe Mazzeo, Richard Melrose, and Iosif Polterovich for helpful conversations. Thanks also to the anonymous referee for a careful reading of the paper, which led to significant improvements. The first author was supported by the Ministry of Education and Science of the Republic of Kazakhstan under grant AP08856479. The first author also was supported by the Government of Kazakhstan and the World Bank under grant APP-PHD-A-18/013P financed by the project "Fostering productive innovation." The second author is supported by the Swedish Research Council Grant, 2018-03873 (GAAME).  The third author was partially supported by a grant from the College of Science and Health at DePaul University.     
	%
	\section{Analytic preliminaries} \label{greens} 
	The Laplace operator on a curvilinear polygonal domain, even with specified boundary conditions on each side, is emphatically \em not \em guaranteed to be self-adjoint.  The angles at the vertices as well as the possibility of different boundary conditions on either side of a vertex can give rise to interesting phenomena \cite{grisvard}, \cite{dauge1}, \cite{dauge2}.  We shall consider a Friedrichs type extension of the Laplace operator here.  The Laplacian is a priori a symmetric operator on smooth, compactly supported functions on our domain, $\Omega$.  
	
	Our sign convention for the Robin boundary condition is 
	$$\frac{\pa v}{\pa n} = \kappa v, \quad \textrm{ for the inward pointing normal derivative.}$$ 
	The Robin parameter is smooth on each boundary component and is non-negative.  
	We define the Laplace operator corresponding to the mixed boundary conditions in the following way, as in \cite{frank} and \cite{raimondi}.  Consider the form
	\begin{equation*}
		a(u,v)=\int_\Omega \nabla u(z) \overline{\nabla v(z)}dz+\int_{\partial \Omega_R} \kappa(z) u(z)\overline{v(z)}d\sigma(z) 
	\end{equation*}
	with domain
	\begin{equation*}
		\mathrm{D}(a)=\{u \in H^1 (\Omega) : u|_{\pa\Omega_D} = 0\}.
	\end{equation*}
	Above, $\pa\Omega_D$ and $\pa \Omega_R$ are the unions of the boundary components on which we impose the Dirichlet and Robin boundary conditions, respectively, and $\kappa(z)$ is the Robin parameter. Then $a$ is a closed, densely defined, symmetric form. Therefore, by \cite[Theorem 2.23 Ch. 6]{Kato}, it generates a self-adjoint operator, which we call the Laplace operator corresponding to the boundary conditions we mentioned above.
	
	\subsection{Green's functions}\label{Green's functions}
	The general approach to study the heat kernel via the associated Green's function and Kantorovich-Lebedev transform is well documented in the literature, dating at least back to Fedosov in the 1960s \cite{fedosov}.  This approach has continued to produce interesting results in modern work as well; see for example the doctoral thesis of U\c{c}ar \cite{ucar} who considers polygonal domains in hyperbolic surfaces.  Although the general technique is well known, the details of the calculations are often omitted.  To maintain the flow and focus of this work, we present here the results of our calculations, and for the sake of completeness include the details in Appendix \ref{app1}. 
	
	We obtain integral expressions for Green's functions for the Laplacian on an infinite circular sector with Dirichlet, Neumann, and mixed boundary conditions.  Let $\gamma$ be the interior angle of the sector; we need only assume $\gamma \in (0, 2\pi)$. 
	The Green's function solves the following equation:
	\begin{equation}\label{Green_fnc}
		\begin{cases}
			sG-\frac{\partial^2 G}{\partial r^2}-\frac{1}{r}\frac{\partial G}{\partial r}-\frac{1}{r^2}\frac{\partial^2 G}{\partial \phi^2}=\frac{1}{r}\delta(r-r_0)\delta(\phi-\phi_0),\\
			\left.\left(\alpha G+\beta\frac{\partial G}{\partial\phi}\right)\right|_{\phi=0,\gamma}=0,
		\end{cases}
	\end{equation}
	with $\alpha = 1$ and $\beta = 0$ for the Dirichlet boundary condition or $\alpha = 0$ and $\beta = 1$ for the Neumann boundary condition, and in all cases with spectral parameter $s>0$.

	For the Dirichlet boundary condition we compute in Appendix \ref{app1} that the Green's function is 
	\begin{equation} \label{DirichletGK}
		G_D(s,r,\phi,r_0,\phi_0)=\frac{1}{\pi^2}\int_{0}^{\infty}K_{i\mu}(r \sqrt s)K_{i\mu}(r_0 \sqrt s)
	\end{equation}
	
	\begin{equation*}
		\times\biggl\{ \cosh(\pi-|\phi_0-\phi|)\mu-\frac{\sinh\pi\mu}{\sinh\gamma\mu}\cosh(\phi+\phi_0-\gamma)\mu+\frac{\sinh(\pi-\gamma)\mu}{\sinh\gamma\mu}\cosh(\phi-\phi_0) \mu \biggl\}d\mu.
	\end{equation*}
	For the Neumann boundary condition, we obtain
	
	\begin{equation} \label{NeumannGK} 
		G_N(s,r,\phi,r_0,\phi_0)=\frac{1}{\pi^2}\int_{0}^{\infty}K_{i\mu}(r \sqrt s)K_{i\mu}(r_0 \sqrt s)
	\end{equation}
	
	\begin{equation*}
		\times\biggl\{ \cosh(\pi-|\phi_0-\phi|)\mu+\frac{\sinh\pi\mu}{\sinh\gamma\mu}\cosh(\phi+\phi_0-\gamma)\mu+\frac{\sinh(\pi-\gamma)\mu}{\sinh\gamma\mu}\cosh(\phi-\phi_0) \mu   \biggl\}d\mu.
	\end{equation*}
	
	For the mixed Dirichlet-Neumann boundary condition, taking the Dirichlet condition at $\phi=0$ and the Neumann condition at $\phi=\gamma$ we obtain the Green's function 
	\begin{equation}  \label{greens-dn} 
		G_{DN} (s,r,\phi,r_0,\phi_0) =\frac{1}{\pi^2}\int_{0}^{\infty}K_{i\mu}(r\sqrt{s})K_{i\mu}(r_0\sqrt{s}) 
	\end{equation} 
	\begin{equation*}
		\times\biggl\{ \cosh(\pi-|\phi_0-\phi|)\mu\\
		+\frac{\sinh(\pi\mu)}{\cosh\gamma\mu}\sinh((\phi+\phi_0-\gamma)\mu)-\frac{\cosh(\pi-\gamma)\mu}{\cosh\gamma\mu}\cosh((\phi-\phi_0)\mu \biggl\}d\mu.
	\end{equation*}

	\subsection{The Heat kernel and the Green's function} \label{s:hk_gf}
	
	Let $\Delta$ be a self-adjoint, non-negative Laplace operator whose domain is contained in $\cL^2(\Omega)$ associated with certain boundary conditions
	\begin{equation*}
		B(u)=0
		\quad\text{on}
		\quad\partial\Omega,
	\end{equation*}
	where $\Omega$ is a domain with a piecewise smooth boundary $\partial\Omega$ which is contained in a larger smooth ambient manifold. Assume that $G(x,y,s)$ is the Green's function of the operator $s+\Delta$, that is the solution of the system
	\begin{equation}
		\begin{cases}
			(s+\Delta)G(x,y,s)=\delta(x-y),\\
			B(G)=0.
		\end{cases}
	\end{equation}
	Before stating the next result, we recall the definition of the Laplace transform and its inverse.  
	\begin{definition} Let $f$ be a continuous function such that there exists a constant $c>0$ with
		$$\int_0 ^\infty |f(t)| e^{-c |t|} dt < \infty.$$
		The Laplace transform of $f$ is defined to be
		$$g(s):=\mathcal{L}(f(t))(s) :=\int_0^\infty f(t)e^{-st} dt, \quad \textrm{Re}(s) \geq c.$$
		The inverse Laplace transform is then
		$$f(x):=\mathcal{L}^{-1}(g(s))(t) :=\frac{1}{2\pi i} \lim_{k\to \infty}\int_{a -ik}^{a +ik} g(s) e^{st} ds,$$
		for $t>0$ and $a>c$.
	\end{definition}
	
	\begin{proposition}\label{KerisinvLapofG}
		With the notations above, let $H(x,y,t)$ be the heat kernel corresponding to $\Delta$. Then
		\begin{equation*}
			\mathcal{L}[H](x,y,s)=G(x,y,s),
		\end{equation*}
		where $\mathcal{L}$ is the Laplace transform.
	\end{proposition}
	
	\begin{proof}
		Let $\{e^{-t\Delta}\}_{t\geq 0}$ be the semigroup generated by $-\Delta$. We note that $-\Delta$ is a non-positive, self-adjoint operator, so that this semigroup is well defined on $\cL^{2}(\Omega)$. Moreover, the self-adjointness gives
		\begin{equation}
			\left\|(\lambda + \Delta)^{-1}\right\|\leq \frac{1}{\mathrm{dist}(\lambda, \sigma(-\Delta))}
			\qquad \lambda \in \rho(-\Delta),
		\end{equation}
		where $\sigma(-\Delta)$ and $\rho(-\Delta)$ are, respectively, the spectrum and resolvent set of $-\Delta$. Therefore, by the Hille-Yosida theorem \cite{Jost}, $\{e^{-t\Delta}\}_{t\geq 0}$ is a contracting semigroup. Hence, by Theorem 8.2.2 in \cite{Jost}, it follows
		\begin{equation}\label{semi_res}
			\mathcal{L}\circ e^{-t\Delta}(s)=(s+\Delta)^{-1},
			\qquad s>0,
		\end{equation}
		where $\mathcal{L}$ is the Laplace transform acting in $t$ variable.
		
		On the one hand, since the heat kernel $H(x,y,t)$ is $\mathcal{L}$ transformable, we may express 
		\begin{equation*}
			\mathcal{L}[e^{-t\Delta}\phi]=\mathcal{L}\int_{\Omega}H(t,x,y)\phi(y)dy=\int_{\Omega}\mathcal{L}[H]\phi(y)dy
		\end{equation*}
		for $t$, $s>0$ and
		\begin{equation*}
			(s+\Delta)^{-1}\phi=\int_{\Omega}G(x,y,s)\phi(s)dy.
		\end{equation*}
		for $s>0$. Therefore, the uniqueness of Schwartz kernels and \eqref{semi_res} imply the statement.
	\end{proof}
	
	\begin{remark}
		Due to Proposition \ref{KerisinvLapofG}, by applying the inverse Laplace transform to \eqref{DirichletGK}, \eqref{NeumannGK}, and \eqref{greens-dn}, we obtain expressions for the heat kernels for the Laplacian on an infinite circular sector with Dirichlet, Neumann, and mixed boundary conditions, respectively.  The heat kernels for an infinite circular sector were computed by Cheeger using separation of variables in polar coordinates \cite[p. 592 (3.42)]{cheeger}. Cheeger's formula simplifies in our setting to:
		\begin{equation}\label{eq:cheegerhk} 
			H(t,r,\theta,r',\theta')=\frac{1}{2t}\exp\left[-\frac{r^2+(r')^2}{4t}\right]\sum_{j=1}^{\infty}I_{\mu_j}\left(\frac{rr'}{2t}\right)\phi_j(\theta)\phi_j(\theta').
		\end{equation}
		Here $I_{\mu_j}$ are the modified Bessel functions, and $(\phi_j,\mu_j)$ are the eigenfunctions, and corresponding eigenvalues, of the appropriate eigenvalue problem (D-D, N-N, or D-N) on the interval $[0,\gamma]$.  By the uniqueness of the heat kernel we therefore obtain the equality of these expressions with the inverse Laplace transform of the expression for the Greens functions.  
	\end{remark}

	\section{Heat spaces} \label{s:heatspaces} 
	We consider curvilinear polygonal domains as in Definition \ref{def:curvpoly}; see examples illustrated in Figure \ref{fig:curvpoly}.
	\begin{figure} \includegraphics[width=200pt]{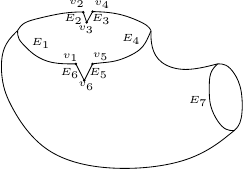}  \includegraphics[width=200pt]{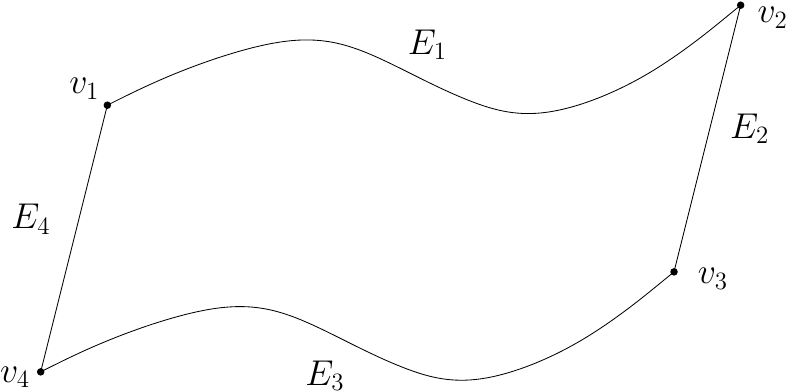} 
		\caption{Two examples of curvilinear polygonal domains, with edges and vertices.} 
		\label{fig:curvpoly} 
	\end{figure}

	\subsection{Manifolds with corners and polyhomogeneity} 
	Near a vertex, $\Omega$ has the differentiable structure of a manifold with corners after blowing up the vertex. Specifically, an open neighborhood of a vertex is diffeomorphic to a sector, $(0, \eps] \times C$, where $C$ is a circular arc.  If we include the point $\{0\} \times C$, then we obtain a smooth manifold with corners. This is what is meant by ``blowing up the vertex,'' in the sense that we replace the vertex with a copy of the link of the sector, namely $C$. This process may be thought of as pretending that polar coordinates are actually valid down to the origin. Doing this construction at each vertex yields a smooth surface with corners which we call $\Omega_0$. In this way, we may identify the differentiable structure of all the surfaces we consider here as the differentiable structure of \em manifolds with corners, \em defined below.  The definition is first due to Melrose \cite{calculus}, here we use the version which is introduced in \cite{Melrosenotes}.

	In order to define a manifold with corners, we must first define t-manifolds.   
	
	\begin{definition}\cite[Def. 1.6.1]{Melrosenotes}  An $n$-dimensional t-manifold $X$ is a  paracompact Hausdorff space such that at each point $x \in X$ there is a non-negative integer $k$ such that a neighborhood of $x$ is homeomorphic to a neighborhood of the origin in the product $[0,\infty)^k \times \R^{n-k}$, with all transition maps being smooth with respect to the subspace topology on $[0,\infty)^k\times \R^{n-k}\subseteq\R^n$.
	\end{definition}

	Now we define a manifold with corners.
	
	\begin{definition}\cite[Def. 1.8.5]{Melrosenotes}\label{defmwc}
		A manifold with corners is a t-manifold such that each boundary hypersurface is embedded.
	\end{definition}

	Since with this definition we may have $k=0$, we see that smooth manifolds without boundary also fit into the general class of ``manifolds with corners.'' 
	
	The purpose of the heat space construction is to create spaces on which the heat kernel and its trace are polyhomogeneous conormal distributions, abbreviated, pc.  This is a natural class of functions within which to study partial differential equations on manifolds with corners; see \cite{etdeo1} and references therein. We briefly recap the definition here. To begin, we say that a subset $F\subseteq\mathbb C\times\mathbb N_0$ is an \em index set \em if $F$ is a discrete set satisfying the following properties:
	\begin{itemize}
		\item For all $N$, $F\cap\{\Re(z)<N\}$ is finite;
		\item If $(s,p)\in F$, then $(s+1,p)\in F$;
		\item If $(s,p)\in F$ and $p>0$, then $(s,p-1)\in F$.
	\end{itemize}
	The latter two conditions are sometimes omitted from this definition, but they give pc functions nice invariance properties; see 
	\cite{grieser}. 
	
	Let $X$ be an $n$-dimensional manifold with corners, and let $\{M_i\}_{i=1}^J$ be the set of its boundary hypersurfaces, that is, the set of all boundary faces of codimension one. We say that $\mathcal F=(F_1,\ldots,F_J)$ is an \em index family \em for $X$ if each $F_i=\{(s_{ij},p_{ij})_{j=1}^{\infty}\}$ is an index set, ordered so that $s_{ij}\in\mathbb R$ are non-decreasing and $p_{ij}$ are non-increasing whenever $s_{ij}$ is unchanged. For each $i$, let $x^i$ be a \em boundary defining function \em for $M_i$; that is, a smooth, non-negative function $x^i : X \to \R$ such that $x^i$ vanishes precisely at $M_i$ but the differential $dx^i$ is non-zero on $M_i$. We may use boundary defining functions as coordinates on $X$. Finally, let $\cV_b$ denote the space of smooth vector fields on $X$ which are tangent to all boundaries. With this terminology, we define $\cA^{\mathcal F}(X)$, the space of polyhomogeneous conormal, or pc functions, to be the space of functions $f$ smooth on the interior of $X$ which have:
	\begin{itemize}
		\item generalized Taylor-like expansions at each boundary hypersurface $M_i$ of the form
		\[f\sim\sum_{j=1}^{\infty}(x^i)^{s_{ij}}(\log(x^i))^{p_{ij}}a_{ij}(x^1,\ldots,x^{i-1},x^{i+1},\ldots,x^n), \]
		where for each $i$, the set $\{(s_{ij},p_{ij})\}$ is the index set $F_i$, enumerated so that $s_{ij}$ is non-decreasing;
		\item product type expansions of the same form at each corner (polyhomogeneous), 
		\item and for which $Vf$ has expansions of the same type whenever $V$ is a product of elements of $\cV_b$ (conormal).
	\end{itemize} The union of these spaces over all possible index sets is denoted $\cA^*(X)$.
	Note that by definition these spaces are invariant under $\cV_b$, in the sense that for any $V\in\cV_b$ and any $u\in\cA^{\mathcal F}(X)$, $Vu\in\cA^{\mathcal F}(X)$ as well.  Observe also that smooth functions on $X$ are pc with each index set consisting of $\mathbb N_0\times\{0\}$.
	
	\subsubsection{Blowups}
	Consider the finite cone, $(0,1]_r \times \Sone_\theta$ with the Riemannian metric, $dr^2 + r^2 d\theta^2$, where $d\theta^2$ is the standard metric on $\Sone$, and the conical point is at $r=0$.  The simplest example of blowing up is replacing the conical point at $r=0$ with a copy of $\Sone$, so that the finite cone is now topologically identified with the cylinder $[0,1] \times \Sone$.  In this example, the point at $r=0$ is replaced with the set of all directions, that is all values of $\theta$, with which one can approach the point, $r=0$.  This type of blowup is known as a \emph{radial blowup}, or a \emph{normal} blowup.  
	
	More generally, we shall consider blowups along \em $p$-submanifolds. \em   An embedded submanifold $Y$ contained in a manifold with corners, $X$, is a \em $p$-submanifold  \em if near each point $q \in Y$, there exist local coordinates so that $Y$ is defined by the vanishing of a subset of these local coordinates.  For example, the boundary faces of $X$ are $p$-submanifolds.  The intersection of two or more boundary faces of $X$ is also a $p$-submanifold.  The normal blowup of $X$ around $Y$ is denoted by 
	$$[X; Y] = \ff \sqcup (X \setminus Y).$$
	Above, ff is the inward pointing spherical normal bundle of $Y$ which has replaced $Y$ in $[X;Y]$.  There is a unique minimal differentiable structure with respect to which $[X;Y]$ is a manifold with corners such that the following two conditions hold.     
	\begin{enumerate}
		\item There is a smooth ``blow-down'' map
		$$\beta : [X;Y] \to X$$
		which is the identity on $(X\setminus Y)$.
		\item Cylindrical coordinates around $Y$ are smooth coordinates on $[X;Y]$.
	\end{enumerate} 
	In case we wish to blow up two or more $p$-submanifolds, we write 
	$$[X; Y_1; Y_2]$$
	to indicate that we first blow up $Y_1$ and next blow up the lift of $Y_2$ to $[X;Y_1]$. This lift is the usual lift if $Y_2\subseteq Y_1$, and otherwise is the closure of $Y_2\setminus Y_1$ in $[X;Y_1]$.  
	
	
	
	\subsection{The single heat space}
	The first of the heat spaces we construct is the single heat space. Let $\Omega$ be a curvilinear polygonal domain. Let $E$ be the set of edges of $\Omega$ (maximal smooth boundary components) and $V$ the set of vertices. Throughout, we let $\Omega_0$ be $\Omega$ with the vertices blown up, so that $\Omega_0$ is a surface with corners. We also let $\tilde V$ be the lift of $V$ to $\Omega_0$, that is, the union of the faces $\{r=0\}$ at each vertex.
	
	Throughout, we will use the time coordinate $T=\sqrt{t}$. This changes the smooth structure at $t=0$ somewhat, but allows us to avoid the use of parabolic blow-ups.
	
	The heat kernel restricted to the diagonal is defined on $\Omega \times [0, \infty)_T$ and is dubbed \em the diagonal heat kernel.  \em  The single heat space is a natural habitat of the diagonal heat kernel in the sense that the diagonal heat kernel lifted to the single heat space is pc.  
	Note that the single heat space is the same for all the possible boundary conditions we consider.
	To create the single heat space, we begin with the manifold with corners $\Omega_0\times[0,1)_T$. We denote its $T=0$ boundary hypersurface\footnote{A boundary hypersurface will often be referred to as a boundary face or face. A ``side face'' is a boundary hypersurface arising from the boundary in $\Omega$.  This is in contrast to the $tf$ face as well as to the boundary hypersurfaces created by blowing up along $p$-submanifolds.} by tf. The remainder of the boundary hypersurfaces correspond either to an edge or to a vertex (which has been blown up, so there is a boundary hypersurface for each vertex). Denote the edge/side faces for positive $T$ by $\{e_j\}_{j=1}^{|E|}$, and the vertex faces $\tilde V_j\times[0,1)_T$ by
	\begin{equation} \label{firstbdhs} \{ s v_j\}_{j=1} ^{|V|}. \end{equation}
	
	Next we perform blow-ups, first of the vertices at $T=0$ and then of the edges at $T=0$, to create the \emph{single space}
	\footnote{In the notation $M_h$ for the single heat space, $M$ acts as a place-holder for the various model geometries we shall use to construct the heat kernel on our surface $\Omega$ (the analogous notation is used for the double and triple heat spaces).}
	$$M_h = [ \Omega_0 \times [0, 1)_T ;  \{s v_j\}_{j=1}^{|V|} \cap \{T=0\};  {E \times \{T=0\}}].$$

	We call the new faces obtained $\{pv_j\}_{j=1}^{|V|}$ and $\{pe_j\}_{j=1}^{|E|}$.
	\begin{figure}\includegraphics[width=200pt]{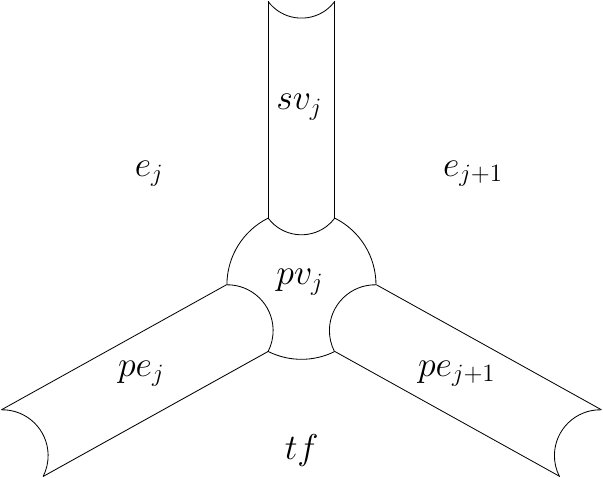} 
		\caption{Single heat space.}  
	\end{figure} 
	The sequence of blowups is therefore:  
	\begin{enumerate}
		\item the normal blowup about each vertex for all time (implicit in the starting point of $\Omega_0\times[0,1)_T$);
		\item the blowup of each vertex at $T=0$;
		\item the blowup of each edge at $T=0$.
	\end{enumerate} 
	The space $M_h$ has $2|V|+2|E|+1$ boundary hypersurfaces in total.
	
	To motivate this construction, consider the diagonal heat kernel on an infinite sector, which from \eqref{eq:cheegerhk} is
	$$
	H(t,r,\theta,r,\theta)=\frac{1}{2T^2}\exp \left[-\frac{r^2}{2T^2}\right]\sum_{j=1}^{\infty}I_{\nu_j}\left(\frac{r^2}{2T^2}\right)|\phi_j(\theta)|^2.
	$$
	The pre-factor is not a pc function of $(r,T)$, but it is a pc function of $(r,T/r)$. And indeed, after creating pv$_j$, $T/r$ may be taken as a boundary defining function for tf near pv$_j$.
	
	In relation to the literature, our heat space construction can be seen as a hybrid combining elements of Mooers's heat space for manifolds with isolated conical singularities \cite{mooers} and Mazzeo \& Vertman's heat space for manifolds with edges \cite{MaVe}.  The first step in Mooers's construction is to blow up the conical singularity by replacing the point with the cross-section (link) of the cone.  Next she takes the product with time.  It is completely equivalent to first take the product with time, and then perform a normal blowup of the cone point for all time.  This is the procedure we follow at the vertices (and cone points).  Next, we perform  blowups at $T=0$ of the edges (and smooth boundary), analogous to \cite{MaVe}. 
	
	In fact our construction can be viewed as an iterated version of \cite{MaVe}, where we perform each of their blow-ups at the vertices (of codimension 2) and then at the edges (of codimension 1). We expect that a similar construction, with additional iteration, could work for polyhedral domains in manifolds of arbitrary dimension.

	\subsection{The double heat space} The double heat space is a natural habitat of the heat kernel in the sense that the heat kernel, initially defined on $\Omega\times\Omega\times[0,\infty)$, lifts to be pc on it. As with the single space, our models guide the construction of the double heat space by indicating which p-submanifolds should be blown up to ensure that the heat kernel will be pc. The general philosophy is to mimic \cite{MaVe}, performing each of their blow-ups first at the vertices and then at the edges.
	
	Begin with $M^2:=\Omega_0\times\Omega_0\times[0,1)_T$. As we are using $\Omega_0$ rather than $\Omega$, this is now a manifold with corners, the analogue of the space with which Mazzeo \& Vertman begin \cite[\S 3.1]{MaVe}. Denote its $T=0$ boundary hypersurface by tf. All other boundary hypersurfaces are of one of the following forms:
	\begin{itemize}
		\item $E_j\times\Omega_0\times[0,1)$, which we call $E_{j0}$, 
		\item $\Omega_0\times E_j\times [0,1)$, which we call $E_{0j}$,
		\item $\tilde V_j\times\Omega_0\times [0,1)$, which is now a boundary hypersurface which we call hvrf$_j$;  
		\item $\Omega_0\times \tilde V_j\times [0,1)$, which we call hvlf$_j$. 
	\end{itemize}
	As a guide to the nomenclature here,``h'' indicates ``height'' because these blowups persist for all time, and time is usually the vertical axis in figures of this type.  As usual, ``v'' indicates vertices, and ``ff'', ``rf", and ``lf" indicate left, right, and front faces respectively\footnote{Some authors reverse the roles of ``right" and ``left" here -- our terminology is chosen to match \cite{MaVe}.}. Now, for each $j$ and $k$, blow up the intersection hvlf$_{j}\cap$hvrf$_k$ to create a new boundary hypersurface\footnote{This is different from \cite{MaVe}. It may be that only the blow-ups with $j=k$ are necessary, but doing all of them makes the proof of the composition theorem easier to read.} hvff$_{jk}$. At this point we call the new space $M_0^2$:
	\[M_0^2=[\Omega_0\times\Omega_0\times[0,1);\ \cup_{j,k}\textrm{hvlf}_j\cap\textrm{hvrf}_k].\]

	The next step is to blow up the union over all $j$ of hvff$_{jj}\cap\{T=0\}$.  This blowup creates $N$ boundary faces, one at each vertex, denoted by ff$_j$.  We shall collectively refer to these as ff for ``front face(s).''  The resulting space is 
	$$[M_0 ^2 ; \cup_j \hvff_{jj} \cap \{ T=0\}].$$
	This construction at the vertices needs to be imitated at the edges, so now lift the triple intersection of the diagonal, boundary, and $T=0$, namely $\{ (z, z', 0): z=z' \in E\}$, to $[M_0 ^2; \hvff_{jj} \cap \{T=0\}]$.  As the lift of a p-submanifold is a p-submanifold, this lift is a boundary p-submanifold which meets ff.  We blow it up, creating a new boundary hypersurface at each side face.  We denote their union by sf, for ``side face(s)". When we need to distinguish components, we shall refer to the $j^{th}$ component as sf$_{j}$. 
	Observe that the side faces are pairwise disjoint, as their intersection in $\Omega_0\times\Omega_0\times[0,1)$ has already been blown up to create ff. Let us call the space at this point the ``reduced double heat space" $M_{rh}^2$.  It is the analogue in our setting of the ``intermediate heat space'' in \cite[\S 3.1]{MaVe}. Specifically, we have
	$$M_{rh}^2=[M_0 ^2 ; \cup_j\hvff_{jj} \cap \{ T=0\}; \{ (z, z', 0): z=z' \in E\}].$$
	
	The final blowup is at the lift of the diagonal at $T=0$. This blowup produces the double heat space:
	$$M_h^2=[M_{rh} ^2; \{ (z, z', 0): z=z' \in \Omega\}].$$
	Call the new front face td, for ``time diagonal.''  It intersects both sf and ff, as well as tf, though no other boundary hypersurfaces (in particular none of the hvff, hvlf, or hvrf components, as the intersection of the diagonal with the boundary has already been blown up).  The double heat space is depicted in Figure \ref{fig:dubheatsp}. 
	
	\begin{figure}
		\centering
		\includegraphics[width=0.5\textwidth]{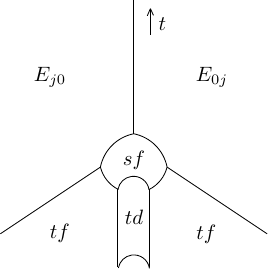}
		\caption{This is a schematic depiction of the double heat space.}
		\label{fig:dubheatsp}
	\end{figure}

	\subsection{The triple heat space} 
	The triple heat space, unlike the single and double heat spaces, is \em not \em a natural habitat.  Instead, it is an artificial environment to which we shall lift the Schwartz kernels of operators from their natural habitats on the double space in order to compose them. With the correct construction of the triple space, the process of composition returns an element which is pc on the double space.  Consequently, the construction of the triple space is guided by the desire to be able to lift and compose Schwartz kernels which live on the double heat space.
	
	Our construction is based on that of Mazzeo and Vertman \cite{MaVe}, and indeed is identical when $V=\emptyset$, that is when there are no vertices. Our guiding principle is that whenever Mazzeo and Vertman blow up a boundary, we first blow up $V$ and then $E$. However, our setting is further complicated by the additional blow-ups at the vertices for positive time.  To begin, we consider the original triple space, with the vertices blown up in each factor so that we have the structure of a manifold with corners:
	\begin{equation}
		\label{eq:triplespacenoblowups}
		M^3:=\Omega_0\times\Omega_0\times\Omega_0\times[0,1) \times[0,1)
	\end{equation}
	along with the three projections $\pi_C$, $\pi_L$, and $\pi_R$ defined by
	\begin{equation} \label{eq:piC} \pi_C:M^3\to \Omega_0^2\times\mathbb R_{\sqrt{(T')^2+(T'')^2}}, (z,z',z'',T',T'')\to(z,z'',\sqrt{(T')^2+(T'')^2}); \end{equation} 
	\begin{equation} \label{eq:piL} \pi_L:M^3\to \Omega_0^2\times\mathbb R_{T'}, (z,z',z'',T',T'')\to(z,z',T'); \end{equation} 
	\begin{equation} \label{eq:piR} \pi_R:M^3 \to \Omega_0^2\times\mathbb R_{T''}, (z,z',z'',T',T'')\to(z',z'',T''). \end{equation} 
	These projections will be used to re-interpret operator composition in terms of pullbacks and push-forwards. Modulo all the technical details, if we have two operators, one $A$ with Schwartz kernel $K_A$ and the other $B$ with Schwartz kernel $K_B$, the Schwartz kernel $K_C$ of the composition $C=A\circ B$ is given by
	\[K_C=(\pi_C)_*(\pi_L^*K_A\pi_R^*K_B).\]
	
	Fortunately, we do not need the full composition formula, just the version in which one of the operators vanishes to infinite order at td.  This is because we use the composition formula to run a Neumann series argument to construct our heat kernel, and with a good enough initial parametrix for the heat kernel, the error will vanish to infinite order at td. 
	For this reason, we will construct a reduced triple heat space $M^3_{rh}$ in order to prove a special case of the composition formula, the case in which $K_B$ has order $\infty$ at the face td.  In this case, $K_B$ is pc on $M^2_{rh}$.\footnote{In general, the operator kernel would be pc on $M^2 _h$ but not on $M^2 _{rh}$. }
	
	The triple space construction is guided by the following conditions, which are necessary and sufficient to obtain the composition formula we require:
	\begin{enumerate} 
		\item  we need the projection $\pi_C$ to lift to a b-fibration $\Pi_C:M^3_{rh}\to M^2_{rh}$; 
		\item we need the projection $\pi_L$ to lift to a b-map  $\Pi_L:M^3_{rh}\to M^2_h$; 
		\item since $K_B$ is pc  on $M^2_{rh}$, we need only that the projection $\pi_R$ lift to a b-map $\Pi_R:M^3_{rh}\to M^2_{rh}$.
	\end{enumerate}
	We therefore recall the definitions of these important ``b-notions.''
	
	\begin{definition}[\cite{calculus} p. 51] Let $f: X \to Y$ be a smooth map between manifolds with corners, $X$ with boundary hypersurfaces $M_1(X)$, and $Y$ with boundary hypersurfaces $M_1(Y)$.  Then $f$ is a \em b-map \em if for each $H \in M_1(Y)$, and boundary defining function $\rho_H$ 
		$$f^*(\rho_H) = 0 \quad \textrm{or} \quad f^*(\rho_H) = a \cdot \prod_{G \in M_1 (X)} \rho_G ^{e_f (G, H)}, \quad 0< a \in \cC^\infty (X).$$
		In the latter case, the numbers $e_f (G,H)$ are called the boundary exponents of $f$.  In this case, writing $M_1 (X) = \{ G_j\}_{j=1} ^n$ and $M_1 (Y) = \{H_k \}_{k=1} ^m$, the \em exponent matrix \em is the matrix whose entries are $\{e_f (G_j, H_k)\}_{j,k=1} ^{n,m} $.
		
	\end{definition} 
	
	\begin{definition}[\cite{calculus} p. 53] A b-map is a \em b-submersion \em if the b-differential is surjective for all $x \in X$.  For the definition of the b-differential, we refer to \cite[p.53--54]{calculus}.  
	\end{definition}
	
	\begin{definition}[\cite{calculus} p. 53] A b-map is \em b-normal \em if $\, ^b f_*$, defined as in \cite[(7) p. 53]{calculus} is surjective.
	\end{definition} 
	
	\begin{definition}[\cite{hmm} p. 124]  A b-map is a \em b-fibration \em  if $f_*$, acting on the b-tangent bundle, is surjective on each fibre, and the image of each boundary hypersurface in $X$ is either $Y$ or one boundary hypersurface $H \subset Y$.   We note that this holds if and only if the b-map is both b-normal and a b-submersion \cite[p. 53]{calculus}.  
	\end{definition} 
	
	\begin{definition}[\cite{hmm} p. 124] A total boundary defining function for $X$, which we denote $\rho_X$, is the product of boundary defining functions for all the boundary hypersurfaces.  We say that a b-fibration is \em simple \em  if 
		$$f^* \rho_Y = a \cdot \rho_X, \quad 0<a \in \cC^\infty (X).$$
		In terms of the exponent matrix, this is equivalent to requiring that the elements are either 0 or 1, and moreover for each $G \in M_1 (X)$ there exists precisely one $H \in M_1 (Y)$ with $e_f (G,H) = 1$.  
	\end{definition}

	\medskip
	
	To start constructing $M^3_{rh}$, we begin with \eqref{eq:triplespacenoblowups}. This space has a number of boundary hypersurfaces: the $T''=0$ face, which we denote $F_{TL}$, the $T'=0$ face, which we denote $F_{TR}$, and a number of faces of the form $E_j\times \Omega_0 \times  \Omega_0 \times [0,1)_{T''}\times [0,1)_{T'}$, which we call $F_{j00}$, with similar notation for products where $E_j$ is in the second or third factor. We also have a number of faces of the form $\tilde V_j\times \Omega_0 \times \Omega_0 \times [0,1)\times [0,1)$, which we call $F_{V_j00}$ and collectively $F_{V00}$, et cetera. The notation used in the triple space is a bit different from the single and double spaces, in particular the use of capitals.  This is in part to distinguish the triple space as an artefact to be used for the purpose of composition and in part to draw a parallel to related constructions in the literature \cite{ars1, vai} which use an analogous notation.
	
	A few blow-up facts will be useful throughout. 
	\begin{proposition}\label{prop:blowupfacts} The following are true (each is well-known in the geometric microlocal literature): 
		
		\begin{enumerate}
			\item Blow-ups which are nested, disjoint, or transverse commute \cite[Lemma 2.1]{hmm}.
			\item Blow-down maps are b-maps \cite[\S 2.3.3]{grieser}\cite[proof of Lemma 2.7]{hmm}. Moreover, if $Y\subset X$ is an intersection of boundary hypersurfaces of $X$, then the blow-down map from $[X;Y]\to X$ is a b-submersion. (This reduces to the case $X=\mathbb R^n_+$, $Y=0$, and then follows from a computation in local coordinates).
			\item The composition of b-maps is a b-map \cite[\S 2.3.3]{grieser}. Further, it follows immediately from the definition that the composition of b-submersions is a b-submersion.
			\item Once a b-map is known to be a b-submersion, checking b-normality is a matter of ensuring, by checking the exponent matrix, that no boundary hypersurface is mapped into a face of codimension $>1$ in the image \cite[definition 3.9]{grieser} \cite[remark B.4]{ars1}. 
		\end{enumerate}
	\end{proposition}

	We will also use the following.  
	\begin{lemma}\label{lem:newnested} Suppose $A$, $B$, $C$, and $D$ are p-submanifolds of a manifold with corners $X$, and suppose that
		\[A\subseteq B\subseteq D,\ A\subseteq C\subseteq D,\mbox{ and }B\cap C\subseteq A.\]
		Then
		\[[X;C;D;A;B]\cong[X;A;B;C;D].\]
	\end{lemma}
	\begin{proof} Nested blow-ups commute, so we may do $A$ before $C$ and $D$ and thus
		\[[X;C;D;A;B]\cong [X;A;C;D;B].\]
		Similarly, we may do $B$ before $D$ and thus
		\[[X;C;D;A;B]\cong [X;A;C;B;D].\]
		Now since $B\cap C\subseteq A$, the lifts of $B$ and $C$ are disjoint in the space $[X;A]$, and since disjoint blow-ups commute they may be done in either order. This completes the proof. \end{proof}
	
	Our strategy will be to repeatedly take advantage of the following lemma of Hassell, Mazzeo, and Melrose \cite{hmm}:
	\begin{lemma}\label{lem:hmmlemma}\cite[Lemma 2.5]{hmm} Suppose $f:X\to Y$ is a simple b-fibration of compact manifolds with corners. Suppose $U\subset Y$ is a closed p-submanifold. Then, with $S$ the minimal collection of p-submanifolds of $X$ into which the lift of $U$ under $f$ decomposes, $f$ extends from the complement of $f^{-1}(U)$ to a b-fibration $f_U:[X,S]\to [Y,U]$, for any order of blow-up of the elements of $S$.
	\end{lemma}
	This lemma guides our construction of the triple space using the construction of the double space.  
	
	\subsection{Lifting the projection maps}  
	To construct the triple space from $M^3$ we first blow up $\mathcal O=\{T'=T''=0\}$.  
	The spatial variables are unaffected by this blow-up and so the space
	\begin{equation}\label{eq:blowupseparate}
		[M^3;\mathcal O]=\Omega_0^3\times[[0,1)\times[0,1);\{0,0\}].
	\end{equation}
	Denote the new front face by $F_{\mathcal O}$, and as usual continue to denote the $T''=0$ face by $F_{TL}$ and the $T'=0$ face by $F_{TR}$. We claim:
	\begin{lemma}\label{lem:easybfib} The projections $\pi_C$, $\pi_L$, and $\pi_R$ lift by continuity to projections $\Pi_C$, $\Pi_L$, and $\Pi_R$ with domain $[M^3;\mathcal O]$ and ranges as in \eqref{eq:piC}, \eqref{eq:piL}, \eqref{eq:piR}.  Moreover, $\Pi_C$, $\Pi_L$, and $\Pi_R$ are all b-fibrations. Under $\Pi_C$, the image of $F_{\mathcal O}$ is the face $\{T=0\}$, and the faces at $T'=0$ and $T''=0$ are mapped into the interior.
	\end{lemma}
	\begin{proof} This follows immediately from \eqref{eq:blowupseparate} and the corresponding statement considering only the time variables. Specifically, the map
		\[\Pi:[[0,1)_{T'}\times [0,1)_{T''}; \{(0,0)\}]\to[0,1)_{T},\ T=\sqrt{(T')^2+(T'')^2}\]
		is a b-fibration, where the image of the front face is $\{T=0\}$, and the image of the other two faces is $[0,1)$. Now the lifted projection $\Pi_C$ is simply $\Pi$ in the time variables and the usual projection from $M^3$ to $M^2$ in the spatial variables, and thus is itself a b-fibration. Similar arguments take care of $\Pi_L$ and $\Pi_R$, as left and right projection lift to b-fibrations from $[\mathbb R_+^2 ; \{0\}]$ to $\mathbb R$.
	\end{proof}
	
	We will now make further blow-ups to $[M^3;\mathcal O]$ which allow all three of these maps to be b-fibrations onto $M_0^2$. In each we use \ref{lem:hmmlemma}. Recalling \eqref{eq:blowupseparate}, we define submanifolds $\mathcal P_{VVV}$, $\mathcal P_{VV0}$, $\mathcal P_{V0V}$, et cetera of $[M^3;\mathcal O]$ by restricting to $\tilde V$ (the lift of $V$) each of the spatial variables which have index $V$ rather than $0$. For example,
	\[\mathcal P_{V0V}=\tilde V\times\Omega_0\times \tilde V\times[[0,1)\times[0,1);\{0,0\}].\]
	Using this notation, the lift of hvff under $\Pi_C$ is $\mathcal P_{V0V}$, under $\Pi_L$ is $\mathcal P_{VV0}$, and under $\Pi_R$ is $\mathcal P_{0VV}$. Then the lifts of hvlf and hvrf may be computed under each map. Application of Lemma \ref{lem:hmmlemma} shows that
	\[\Pi_C:[M^3;\mathcal O;\mathcal P_{V0V}]\to M_0^2,\]
	\[\Pi_L:[M^3;\mathcal O;\mathcal P_{VV0}]\to M_0^2,\]
	\[\Pi_R:[M^3;\mathcal O;\mathcal P_{0VV}]\to M_0^2,\]
	are each b-fibrations. In fact we can define a common domain: let
	\[M_0^3:=[M^3;\mathcal O;\mathcal P_{VVV};\mathcal P_{V0V};\mathcal P_{0VV};\mathcal P_{VV0}].\]
	Denote the new faces created by $F_{VVV}$, $F_{V0V}$, et cetera.
	\begin{proposition} $\Pi_C$, $\Pi_L$, and $\Pi_R$ all lift to b-fibrations from $M_0^3$ to $M_0^2$.
	\end{proposition}
	\begin{proof}First observe that $M_0^3$ is a blow-up of each of the three domain spaces for $\Pi_C$, $\Pi_L$, and $\Pi_R$. Indeed, begin with the domain space for $\Pi_C$ and blow up the lift of $\mathcal P_{VVV}$. This blow-up is nested with the blow-up of $\mathcal P_{V0V}$, so it may be done before that. So a blow-up of the domain space for $\Pi_C$ is
		\[[M^3;\mathcal O;\mathcal P_{VVV};\mathcal P_{V0V}].\]
		Now we can blow up the lifts of $\mathcal P_{0VV}$ and $\mathcal P_{VV0}$, which are disjoint from the lift of $\mathcal P_{V0V}$ and thus can be done in any order. Analogous arguments show that $M_0^3$ is a blow-up of each of the three domain spaces. Since all our p-submanifolds are intersections of boundary hypersurfaces, by Proposition \ref{prop:blowupfacts}b), $\Pi_C$, $\Pi_R$, and $\Pi_L$ are all b-submersions from $M_0^3$ to $M_0^2$.
		
		At this point, proving these maps are b-fibrations is simply a matter of checking their exponent matrices to make sure no boundary hypersurface is mapped into a corner. This calculation is combinatorial, very similar to the proof of Lemma 3.14 in \cite{etdeo1}, and as a result we omit it.
	\end{proof}
	
	Having proven that $\Pi_C$ is a b-fibration from $M_0 ^3 \to M_0 ^2$, we shall now lift $\Pi_C$ to be a b-fibration onto $M^2 _{rh}$. 
	Let $\mathcal O_{V0V}=F_{\mathcal O}\cap\{z=z''\in \tilde V\}$, with a similar definition for $\mathcal O_{E0E}$, $\mathcal O_{VVV}, \mathcal O_{EEE},$ etc. Using Lemma \ref{lem:hmmlemma}, we obtain that $\Pi_C$ lifts to a b-fibration
	\[\Pi_C:[M_0^3;\mathcal O_{V0V};\mathcal O_{E0E}]\to M^2_{rh}.\]
	Call the new faces $F_{\mathcal OV0V}$ and $F_{\mathcal OE0E}$. We are now in good shape with $\Pi_C$ but we need to do more work in order for $\Pi_L$ and $\Pi_R$ to be b-maps onto $M_h^2$ and $M_{rh}^2$ respectively.
	
	To begin the extra work, we do some more blow-ups which preserve the b-fibration property of $\Pi_C$. Begin with $\mathcal O_{VVV}$, which in $M_0^3$ is $F_{\mathcal O}\cap\{z=z'=z''\in \tilde V \}$. Using notation as before, we claim that
	\[\Pi_C:M^3_{cen}:=[M_0^3;\mathcal O_{V0V};\mathcal O_{E0E};\mathcal O_{VVV};\mathcal O_{EEE}]\to M^2_{rh}\]
	is a b-fibration. Indeed, this follows from two applications of \cite[Lemma 2.7]{hmm}\footnote{See in particular the comment after the proof in \cite{hmm}; in this case $\Pi_C$ restricts to a b-fibration from $\mathcal O_{VVV}$ onto ff and also from $\mathcal O_{EEE}$ onto sf, which is sufficient.}. Moreover, by an application of Lemma \ref{lem:newnested}, we have
	\[M^3_{cen}=[M_0^3;\mathcal O_{VVV};\mathcal O_{EEE};\mathcal O_{V0V};\mathcal O_{E0E}].\]
	Continuing in this vein, let $L_{0VV}=F_{TL}\cap\{z'=z''\in \tilde V\}$, with similar notations for $L_{0EE}$, $R_{VV0}$, $R_{EE0}$, and let $R_{diag}=F_{TL}\cap\{z=z'\}$. Let
	\[M^3_{rh,c}:=[M^3_{cen};R_{VV0};R_{EE0};R_{diag};L_{0VV};L_{0EE}].\]
	Denote the five new boundary hypersurfaces by $F_{RVV0}$, $F_{REE0}$, $F_{TRD}$, 
	$F_{L0VV}$, and $F_{L0EE}$. None of the five new boundary hypersurfaces are mapped into a corner by $\Pi_C$; their images are hvrf, $\cup_j E_{j0}$, the interior, hvlf, and $\cup_j E_{0j}$ respectively. Thus by the same argument with \cite[Lemma 2.17]{hmm}, $\Pi_C:M^3_{rh,c}\to M^2_{rh}$ is a b-fibration.
	
	Finally, using similar notation, we define p-submanifolds $\mathcal O_{VV0}$, $\mathcal O_{EE0}$, $\mathcal O_{0VV}$, $\mathcal O_{0EE}$, and $\mathcal O_D$, which is the interior lift of $\{t'=t''=0,z=z'\}$ in our new space. Define the \emph{reduced triple heat space}
	\[M^3_{rh}:=[M^3_{rh,c};\mathcal O_{VV0};\mathcal O_{EE0};\mathcal O_{0VV};\mathcal O_{0EE};\mathcal O_{D}].\]
	
	Call the new boundary hypersurfaces $F_{\mathcal OVV0}$, $F_{\mathcal OEE0}$, $F_{\mathcal O0VV}$, $F_{\mathcal O0EE}$, and $F_{\mathcal OD}$. It is no longer true that $\Pi_C$ is a b-fibration from $M^3_{rh}$ onto $M^2_{rh}$, but it is a b-map.
	
	
	We must also know something about $\Pi_L$ and $\Pi_R$:
	\begin{proposition}
		$\Pi_L$ and $\Pi_R$ lift by continuity to well-defined b-maps from $M^3_{rh}$ to $M^2_h$ and $M^2_{rh}$ respectively. 
	\end{proposition}
	\begin{proof} It is immediate by composing with the blow-down map that $\Pi_L$ and $\Pi_R$ lift to well-defined b-maps from $M^3_{rh}$ to $M_0^2$. The question is whether they still lift to b-maps when $M_0^2$ is blown up to create ff, then sf, then (for $\Pi_L$) td. But this may be checked directly: computing the pullbacks of the boundary defining functions for the boundary hypersurfaces of $M^2_{h}$ and $M^2_{rh}$ shows that each is a product of boundary defining functions on $M^3_{rh}$. The specific products are given below, in Lemma \ref{lem:exponentmatrices}.
	\end{proof}
	
	\begin{figure} \includegraphics{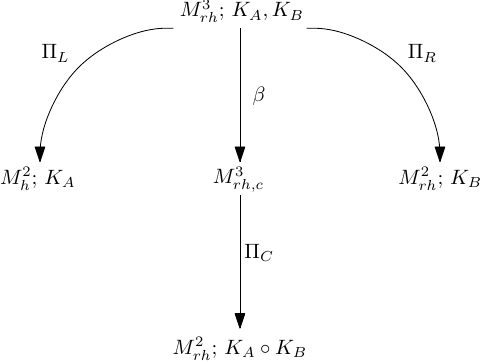} \caption{The schematic diagram for the construction of the reduced triple space as required here for our composition rule.  The kernel $K_A$ lifts from the double heat space, $M_h ^2$, to the triple heat space via pullback by the projection map $\Pi_L$.  The kernel $K_B$ vanishes to infinite order at td, so it is pc on the reduced double heat space.  It lifts to the triple heat space via pullback by the projection map $\Pi_R$.  On $M_{rh} ^3$ the two kernels are composed, and the result is then pushed forward by the blow-down map to $M_{rh,c} ^3$, followed by the projection map $\Pi_C$ to $M_{rh} ^2$.} 
	\end{figure}

	\subsection{Combinatorics of b-maps}
	Now we come to the key combinatorial lemma for composition. Recall that $\Pi_L:M^3_{rh}\to M^2_h$ and $\Pi_R:M^3_{rh}\to M^2_{rh}$ are b-maps and $\Pi_C:M^3_{rh,c}\to M^2_{rh}$ is a b-fibration.
	
	\begin{lemma}\label{lem:exponentmatrices} The exponent matrix entries for the b-maps $\Pi_L:M^3_{rh}\to M^2_h$, $\Pi_R: M^3_{rh}\to M^2_{rh}$, and $\Pi_C:M^3_{rh,c}\to M^2_{rh}$ are all zero, except for the following, which are 1:
		
		\begin{multline}\label{eq:exponentmatrixpil}
			\mbox{For }\Pi_L: (F_{TR},\tf),(F_{\cO},\tf),(F_{\cO V_j0V_j},\tf), (F_{\cO E_j0E_j},\tf), (F_{\cO 0V_jV_j},\tf),\\
			(F_{\cO 0E_jE_j},\tf), (F_{TRD},\td),(F_{\cO D},\td), (F_{\cO V_j V_j V_j},\ff_j), (F_{R V_j V_j 0}, \ff_j), (F_{\cO V_j V_j 0},\ff_j),\\
			(F_{\cO E_j E_j E_j},\sff_j), (F_{R E_j E_j 0},\sff_j), (F_{\cO E_j E_j 0},\sff_j),(F_{V_jV_kV_l},\hvff_{jk}),(F_{V_jV_k0},\hvff_{jk}),\\
			(F_{V_j0V_l},\hvrf_j),(F_{V_j00},\hvrf_j),(F_{\cO V_j0V_j},\hvrf_j),\\(F_{0V_kV_l},\hvlf_k),(F_{0V_k0},\hvlf_k),
			(F_{\cO 0V_jV_j},\hvlf_j), (F_{L0V_jV_j},\hvlf_j),\\
			(F_{\cO E_j0E_j},E_{j0}),(F_{j00},E_{j0}),(F_{\cO 0E_jE_j},E_{0j}),(F_{0j0},E_{0j}),(F_{L0E_jE_j},E_{0j}),
		\end{multline}
		with $F_{TL}$, $F_{00l}$, and $F_{00V_l}$ mapping to the interior.
		\begin{multline}\label{eq:exponentmatrixpir}
			\mbox{For }\Pi_R: (F_{TL},\tf), (F_{\mathcal O},\tf), (F_{\mathcal OD},\tf), (F_{\mathcal OV_j0V_j},\tf), (F_{\mathcal OE_j0E_j},\tf),\\
			(F_{\mathcal OV_jV_j0},\tf), (F_{\mathcal OE_jE_j0},\tf),
			(F_{\mathcal OV_jV_jV_j},\ff_j), (F_{L0V_jV_j}, \ff_j), (F_{\mathcal O0V_jV_j},\ff_j),\\
			(F_{\mathcal OE_jE_jE_j},\sff_j), (F_{L0E_jE_j},\sff_j), (F_{\mathcal O0E_jE_j},\sff_j),(F_{V_jV_kV_l},\hvff_{kl}),(F_{0V_kV_l},\hvff_{kl}),\\
			(F_{V_j0V_l},\hvlf_l),(F_{00V_l},\hvlf_l),(F_{\mathcal OV_j0V_j},\hvlf_j),\\
			(F_{V_jV_k0},\hvrf_k),(F_{0V_k0},\hvrf_k), (F_{\mathcal OV_jV_j0},\hvrf_j), (F_{RV_jV_j0},\hvrf_j),\\
			(F_{\mathcal OE_j0E_j},E_{0j}),(F_{00j},E_{0j}),(F_{\mathcal OE_jE_j0},E_{j0}),(F_{0j0},E_{j0}),(F_{RE_jE_j0},E_{j0}),
		\end{multline}
		with $F_{TR}$, $F_{j00}$, $F_{V_j00}$, and $F_{TRD}$ mapping to the interior.
		\begin{multline}\label{eq:exponentmatrixpic}
			\mbox{For }\Pi_C: (F_{\mathcal O},\tf), (F_{j00}, E_{j0}), (F_{RE_jE_j0},E_{j0}), (F_{00j}, E_{0j}),\\ (F_{L0E_jE_j},E_{0j}), (F_{\mathcal OV_jV_jV_j},\ff_j),(F_{\mathcal OV_j0V_j},\ff_j), (F_{\mathcal OE_jE_jE_j},\sff_j), (F_{\mathcal OE_j0E_j},\sff_j),\\
			(F_{V_jV_kV_l}, \hvff_{jl}), (F_{V_j0V_l}, \hvff_{jl}), (F_{V_jV_k0},\hvrf_j), (F_{V_j00},\hvrf_j), (F_{RV_jV_j0}, \hvrf_j),\\
			(F_{0V_kV_l}, \hvlf_l),(F_{00V_l},\hvlf_l),
			(F_{L0V_lV_l},\hvlf_l),
		\end{multline}
		with $F_{TR}$, $F_{TL}$, $F_{0j0}$, $F_{0V_j0}$, and $F_{TRD}$ mapping to the interior.
	\end{lemma}
	
	\begin{proof} All of the exponent matrices are computed the same way: by computing pullbacks of boundary defining functions. Consider, for example, the face $\ff_j$ of $M^2_{rh}$. The faces of $M_{rh,c} ^3$ which are in the preimage of $\ff_j$ under $\Pi_C$ are all of the faces where $z=z''=V_j$ and $\sqrt{(T')^2+(T'')^2}=0$, that is $T'=T''=0$. 
		These are precisely $F_{\mathcal OV_jV_jV_j}$ and $F_{\mathcal OV_j0V_j}$, so those two faces map to $\ff_j$, and the corresponding exponent matrix entries are 1. Computing these pullbacks for each boundary hypersurface of $M^2_{rh}$ yields the desired exponent matrix for $\Pi_C$. By a similar process we obtain the exponent matrices for $\Pi_L$ and $\Pi_R$.
		
		Note also that the fact that $\Pi_C$ is a b-fibration can be observed directly: for each boundary hypersurface $G$ of $M^3_{rh,c}$ there is at most one boundary hypersurface $H$ of $M^2_{rh}$ such that the $(G,H)$ exponent matrix entry for $\Pi_C$ equals 1.
	\end{proof}
	
	\subsection{Densities} Our kernels on the double space are most naturally considered as ``full right densities" with respect to the usual metric on $\Omega$. For example, the kernels of our operators $A$ and $B$ will be
	\[K_A(t',z,z')\,dt'\, dz';\ K_B(t'',z',z'')\, dt''\, dz'', \quad \textrm{ respectively.}\]
	Multiplying the two, then integrating over $t=t'+t''$ and $z'$, yields
	\[K_{A\circ B}(t,z.z')\,dt\, dz''.\]
	If we multiply everything in the expressions by $dz$, we have the full-density form we need for the pushforward theorem; see \cite[Theorems 4 and 5]{corners} and \cite[Theorem 2.3]{hmm}.
	
	To apply the aforementioned pushforward theorem, we need to transform our natural metric densities into canonical full densities and b-densities on $M^3_{rh,c}$ and $M^2_{rh}$. Here are the formulas for those transformations. Throughout, let $\nu(X)$ and $\nu_b(X)$ be canonical densities and b-densities on a manifold with corners $X$, and let $\rho_{tot}(X)$ be a product of boundary defining functions for all boundary hypersurfaces of $X$. It is immediate that $\nu_b(X)=\rho_{tot}^{-1}(X)\nu(X)$.
	
	\begin{proposition}\label{prop:densityblowup} The density bundles transform under blow-ups as follows:  
		\begin{align*} &\beta^*(\nu(\Omega\times\Omega\times[0,1)))=\rho_{ff}^4\rho_{sf}^3\rho_{hvff}^3\rho_{hvlf}\rho_{hvrf}\nu(M^2_{rh}); \\ 
			&\beta^*(\nu(\Omega\times\Omega\times\Omega\times[0,1)))=\rho_{\mathcal O}\rho_{\mathcal OVVV}^7\rho_{\mathcal OEEE}^6\rho_{\mathcal OV0V}^5\rho_{\mathcal OE0E}^4\rho_{VVV}^5\rho_{V0V}^3 \\ 
			&\cdot \rho_{0VV}^3\rho_{VV0}^3\rho_{V00}\rho_{0V0}\rho_{00V}
			\rho_{RVV0}^4\rho_{REE0}^3\rho_{L0VV}^4\rho_{L0EE}^3\rho_{TRD}^2\nu(M^3_{rh,c}).\end{align*}
	\end{proposition}
	\begin{proof} When blowing up a submanifold $F$ of a manifold with corners $W$, blow-up introduces a factor of $\rho^{\dim(W)-\dim(F)-1}$, equivalently $\rho^{\tiny{\mbox{codim}}(F)-1}$, with $\rho$ the defining function for the new (blown-up) face (see Proposition C.5 of \cite{ars1}). We repeatedly apply this\footnote{Note that even though $\Omega$ with the usual metric is not technically a manifold with corners, the same analysis works to write $\nu(\Omega)$ in terms of $\nu(\Omega_0)$.}.
		
		For $M^2_{rh}$, $\dim(W)=5$. The blow-up to produce ff blows up a finite collection of points, so the codimension is 5 and we acquire a $\rho_{\ff}^4$. The blow-up to produce sf, on the other hand, requires $t'=0$, $z=z'\in E$, so the codimension is 4 and we acquire a $\rho_{\sf}^3$. The blow-up to produce hvff has codimension 4, and the blow-ups to produce hvlf and hvrf have codimension 2. Putting this all together yields the result.
		
		A similar analysis works for $M^3_{rh,c}$, being careful about repeated blow-ups. For example, blowing up $\mathcal O$ introduces a factor of $\rho_{\mathcal O}$ at first. However, when blowing up a submanifold of $\mathcal O$, $\rho_{\mathcal O}$ itself continues to lift. For example, when creating $F_{\mathcal OVVV}$, $\rho_{\mathcal O}$ lifts to $\rho_{\mathcal O}\rho_{\mathcal OVVV}$.
	\end{proof}

	Another important observation is that rather than $t'$ and $t''$, we are treating their square roots as the boundary defining functions, so canonical densities have $dT'$ and $dT'$. For example,
	\[\nu(\Omega\times\Omega\times[0,1))=dT dz dz'',\]
	and recalling that $dt=2TdT$,
	\[K_{A\circ B}dt dz dz''=2T K_{A\circ B}\nu(\Omega\times\Omega\times[0,1)).\]
	Similarly,
	\begin{align*} & K_A(t',z,z')K_B(t'',z',z'')dt' dt'' dz dz' dz'' = \\ & 4T'T''K_A(t',z,z')K_B(t'',z',z'')\nu(\Omega\times\Omega\times\Omega\times[0,1)^2). 
	\end{align*}
	
	\subsection{Composition theorem}
	To define various classes of operator kernels we use notation which is similar but not identical to \cite{MaVe}.  For an index family,
	$$\mathcal F=(F_{\ff_j}, F_{\sff_j}, F_{\hvff_{jk}}, F_{\hvrf_j}, F_{\hvlf_j}, F_{j0},F_{0j}, F_{\td}),$$ 
	define $\mathcal A^{\mathcal F}_{h}(M^2_h)$ to be the space of functions on $M_h^2$ which are pc  with index sets given at the respective faces by $\mathcal F$ and which also vanish to infinite order at tf.  We also define $\Psi_h ^{\mathcal F}$ as the set of pseudodifferential operators whose Scheartz kernels, as functions on $M_h^2$, are elements of $\mathcal A^{\mathcal F}_{h}(M^2_h)$.  
	
	We may also use the notation 
	\[ \mathcal A^{\alpha_{\ff_j}, \alpha_{\sff_j}, \alpha_{\hvff_{jk}}, \alpha_{\hvrf_j}, \alpha_{\hvlf_j}, \alpha_{j0}, \alpha_{0j}, \alpha_{\td}} _{h}\]  
	to indicate functions on $M_h^2$ which are elements of $\mathcal A_h^{\mathcal F}$ for some index family $\mathcal F$ whose index set at each face is bounded below by the corresponding $\alpha$. In other words these are functions on $M_h^2$ which are pc and which have leading order at each face no worse than the corresponding $\alpha$ (and which furthermore do not have terms of the form $x^{\alpha}(\log x)^p$ for $p\geq 1$). We also use $\Psi ^{\alpha_{\ff_j}, \alpha_{\sff_j}, \alpha_{\hvff_{jk}}, \alpha_{\hvrf_j}, \alpha_{\hvlf_j}, \alpha_{j0}, \alpha_{0j}, \alpha_{\td}}$ in the analogous way.

	\begin{theorem} \label{thm:comp} Suppose that $A$ is an operator whose Schwartz kernel $K_A \in \mathcal A^{\mathcal F_A} _h$ with 
		\[\mathcal F_A = (A_{\ff_j}, A_{\sff_j}, A_{\hvff_{jk}}, A_{\hvrf_j}, A_{\hvlf_j}, A_{0j}, A_{j0}, A_{\td}).\] 
		Suppose that $B$ is an operator whose Schwartz kernel $K_B \in \mathcal A^{\mathcal F_B} _h$ with 
		\[ \mathcal F_B = (B_{\ff_j}, B_{\sff_j}, B_{\hvff_{jk}}, B_{\hvrf_j}, B_{\hvlf_j}, B_{0j}, B_{j0}, B_{\td} = \infty).\]
		Suppose, finally, that
		\begin{equation}\label{eq:integrabilitycondition}A_{0j}+B_{j0}>-1,\ A_{\hvlf_j}+B_{\hvrf_j}>-2,\mbox{ and }A_{\td}>-4.
		\end{equation}
		Then the Schwartz kernel of the composition $A \circ B$ is an element of $\mathcal A^{\mathcal F} _h$,
		where $\mathcal F$ has index sets
		\begin{eqnarray*} 
			\begin{gathered} A_{\ff_j}+B_{\ff_j}+4   \textrm{ at } \ff_j, \\
				A_{\sff_j}+B_{\sff_j}+4  \textrm{ at } \sff_j , \\ 
				(\overline\cup_k(A_{\hvff_{jk}}+B_{\hvff_{kl}}+2))\overline\cup(A_{\hvrf_j}+B_{\hvlf_j}) \textrm{ at } \hvff_{jl},  \\ 
				(\overline\cup_k(A_{\hvff_{jk}}+B_{\hvrf_k}+2))\overline\cup(A_{\hvrf_j})\overline\cup(A_{\ff_j}+B_{\hvrf_j}+4) \textrm{ at } \hvrf_j ,  \\ 
				(\overline\cup_k(A_{\hvlf_k}+B_{\hvff_{kl}}+2))\overline\cup(B_{\hvlf_l})\overline\cup(A_{\hvlf_l}+B_{\ff_l}+4) \textrm{ at } \hvlf_l, \\ 
				A_{j0}\overline\cup(A_{\sf_j}+B_{j0}+4)  \textrm{ at } E_{j0}, \\ 
				B_{0j}\overline\cup(A_{0j}+B_{\sf_j}+4)  \textrm{ at } E_{0j}, \\ 
				\infty  \textrm{ at } \td. \end{gathered}  \end{eqnarray*} 
		
	\end{theorem} 
	\begin{proof} 
		By the pullback theorem (see \cite[Theorem 3]{calculus} or \cite[Theorem 2.2]{hmm}), $\Pi_L^*K_A$ is polyhomogeneous on $M^3_{rh}$ with index sets:
		\begin{equation*}
			\begin{gathered}
				A_{\td}\mbox{ at }F_{TRD}\mbox{ and }F_{\mathcal OD};\ A_{\ff_j}\mbox{ at }F_{\mathcal O V_jV_jV_j}, F_{RV_jV_j0},\mbox{ and }F_{\mathcal OV_jV_j0};\\
				A_{\sff_j}\mbox{ at }F_{\mathcal O E_jE_jE_j},F_{RE_jE_j0},\mbox{ and }F_{\mathcal OE_jE_j0};\ A_{\hvff_{jk}}\mbox{ at }F_{V_jV_kV_l}\mbox{ and }F_{V_jV_k0};\\
				A_{\hvrf_j}\mbox{ at }F_{V_j0V_l}\mbox{ and }F_{V_j00}; A_{\hvlf_k}\mbox{ at }F_{0V_kV_l}, F_{0V_k0},\mbox{ and }F_{L0V_kV_k};\\
				A_{j0}\mbox{ at }F_{j00}; A_{0j}\mbox{ at } F_{0j0}\mbox{ and }F_{L0E_jE_j};0\mbox{ at }F_{TL}, F_{00l},\mbox{ and } F_{00V_l};\mbox{ and finally}\\
				\infty\mbox{ at }F_{TR}, F_{\mathcal O},F_{\mathcal OV_j0V_j}, F_{\mathcal OE_j0E_j}, F_{\mathcal O0V_jV_j}, F_{\mathcal O0E_jE_j}.
			\end{gathered} 
		\end{equation*}
		Note in particular that at the four hypersurfaces in the domain mapped to intersections of two hypersurfaces in the range, the pullback index set is the sum of the index sets at the two range hypersurfaces. At each of these the sum ends up being infinity.
		Similarly, $\Pi_R^*K_B$ is polyhomogeneous on $M^3_{rh}$ with index sets
		\begin{equation*}
			\begin{gathered} 
				B_{\ff_j}\mbox{ at }F_{\mathcal OV_jV_jV_j}, F_{L0V_jV_j},\mbox{ and }F_{\mathcal O0V_jV_j}; \\ 
				B_{\sff_j}\mbox{ at }F_{\mathcal OE_jE_jE_j},F_{L0E_jE_j},\mbox{ and }F_{\mathcal O0E_jE_j};\\
				B_{\hvff_{kl}}\mbox{ at }F_{V_jV_kV_l}\mbox{ and }F_{0V_kV_l}; \, B_{\hvlf_l}\mbox{ at }F_{V_j0V_l}\mbox{ and }F_{00V_l}; \\
				B_{\hvrf_k}\mbox{ at } F_{V_jV_k0}, F_{0V_k0}, \mbox{ and }F_{RV_kV_k0}; \, B_{0j}\mbox{ at }F_{00j}; B_{j0}\mbox{ at }F_{0j0}\mbox{ and }F_{RE_jE_j0}; \\ 
				0\mbox{ at }F_{TR}, F_{j00}, F_{V_j00},\mbox{ and }F_{TRD};\mbox{ and finally }\\
				\infty\mbox{ at }F_{TL}, F_{\mathcal O}, F_{\mathcal OD}, F_{\mathcal OV_jV_j0}, F_{\mathcal OV_j0V_j}, F_{\mathcal OE_jE_j0}\mbox{ and } F_{\mathcal OE_j0E_j}.
			\end{gathered} 
		\end{equation*}
		It is also easy enough to compute the pullbacks of $T'$ and $T''$; they each have order 1 at each face in the lift of $T'=0$ and $T''=0$ respectively. Therefore the product $4T'T''(\Pi_L^*K_A)(\Pi_R^*K_B)$ is polyhomogeneous on $M^3_{rh}$ with index sets
		\begin{equation}\label{eq:indexsets1}
			\begin{gathered}
				A_{j0}\textrm{ at }F_{j00}; A_{0j}+B_{j0}\textrm{ at }F_{0j0}; B_{0j}\textrm{ at }F_{00j}; \\ 
				A_{\ff_j}+B_{\ff_j}+2\mbox{ at }F_{\mathcal OV_jV_jV_j}; A_{\sff_j}+B_{\sff_j}+2\mbox{ at }F_{\mathcal OE_jE_jE_j};  \\
				A_{\hvff_{jk}}+B_{\hvff_{kl}}\mbox{ at }F_{V_jV_kV_l};  A_{\hvff_{jk}}+B_{\hvrf_k}\mbox{ at }F_{V_jV_k0};  \\  
				A_{\hvrf_j}+B_{\hvlf_l}\mbox{ at }F_{V_j0V_l}; A_{\hvlf_k}+B_{\hvff_{kl}}\mbox{ at }F_{0V_kV_l};  \\ 
				A_{\hvrf_j}\mbox{ at }F_{V_j00}; A_{\hvlf_k}+B_{\hvrf_k}\mbox{ at }F_{0V_k0}; B_{\hvlf_l}\mbox{ at }F_{00V_l};  \\
				A_{\ff_j}+B_{\hvrf_j}+1\mbox{ at }F_{RV_jV_j0}; A_{\sff_j}+B_{j0}+1 \mbox{ at }F_{RE_jE_j0}; \\ 
				A_{\hvlf_j}+B_{\ff_j}+1\mbox{ at }F_{L0V_jV_j}; A_{0j}+B_{\sff_j}+1 \mbox{ at }F_{L0E_jE_j};  \\
				A_{\td}+1\mbox{ at }F_{TRD};  \\ 
				\infty\mbox{ at }F_{TL}, F_{TR}, F_{\mathcal O}, F_{\mathcal O D}, F_{\mathcal OV_jV_j0}, F_{\mathcal OE_jE_j0}, F_{\mathcal OV_j0V_j}, \\ 
				F_{\mathcal OE_j0E_j}, F_{\mathcal O0V_jV_j},  \mbox{ and } F_{\mathcal O0E_jE_j}.
			\end{gathered}
		\end{equation}
		
		Now we make the observation that on the front faces of each of the five blow-ups needed to create $M^3_{rh}$ from $M^3_{rh,c}$, the product $(\Pi_L^*K_A)(\Pi_R^*K_B)$ vanishes to infinite order. Applying Proposition \ref{prop:blowupfacts} five times, we see that $4T'T''(\Pi_L^*K_A)(\Pi_R^*K_B)$ is in fact polyhomogeneous conormal on $M^3_{rh,c}$ with index sets the same as in \eqref{eq:indexsets1}, with the exception of deleting the five extra faces.
		
		We would like to use the pushforward theorem, but first we must view $4T'T''(\Pi_L^*K_A)(\Pi_R^*K_B)$, currently a section of $\nu(\Omega\times\Omega\times\Omega\times[0,1)^2)$, as a section of $\nu(M^3_{rh,c})$. By Proposition \ref{prop:densityblowup}, as a section of $\nu(M^3_{rh,c})$, its orders are:

		\begin{equation}\label{eq:indexsets2}
			\begin{gathered} 
				A_{j0}\textrm{ at }F_{j00}; A_{0j}+B_{j0}\textrm{ at }F_{0j0}; B_{0j}\textrm{ at }F_{00j};\\ 
				A_{\ff_j}+B_{\ff_j}+9\mbox{ at }F_{\mathcal OV_jV_jV_j}; A_{\sff_j}+B_{\sff_j}+8\mbox{ at }F_{\mathcal OE_jE_jE_j};  \\
				A_{\hvff_{jk}}+B_{\hvff_{kl}}+5\mbox{ at }F_{V_jV_kV_l}; A_{\hvff_{jk}}+B_{\hvrf_k}+3\mbox{ at }F_{V_jV_k0}; \\   A_{\hvrf_j}+B_{\hvlf_l}+3\mbox{ at }F_{V_j0V_l}; A_{\hvlf_k}+B_{\hvff_{kl}}+3\mbox{ at }F_{0V_kV_l}; \\ 
				A_{\hvrf_j}+1\mbox{ at }F_{V_j00}; A_{\hvlf_k}+B_{\hvrf_k}+1\mbox{ at }F_{0V_k0}; \\ 
				B_{\hvlf_l}+1\mbox{ at }F_{00V_l};  A_{\ff_j}+B_{\hvrf_j}+5\mbox{ at }F_{RV_jV_j0}; \\ 
				A_{\sff_j}+B_{j0}+4 \mbox{ at }F_{RE_jE_j0}; A_{\hvlf_j}+B_{\ff_j}+5\mbox{ at }F_{L0V_jV_j}; \\ 
				A_{0j}+B_{\sff_j}+4 \mbox{ at }F_{L0E_jE_j}; A_{\td}+3\mbox{ at }F_{TRD};  \\
				\mbox{and }\infty\mbox{ at }F_{TL}, F_{TR}, F_{\mathcal O}, F_{\mathcal OV_j0V_j}, \mbox{ and } F_{\mathcal O0E_jE_j}. 
			\end{gathered}
		\end{equation}
		
		Now we apply the pushforward theorem and push forward by $\Pi_C$. There is a condition in the pushforward theorem, see [\cite{hmm}, Theorem 2.3] or \cite{etdeo1}:  any face which is mapped to the interior must have index set greater than $-1$ (or, equivalently, $0$ as a b-density). This, however, is guaranteed by \eqref{eq:integrabilitycondition}. By the pushforward theorem, $(\Pi_C)_*((\Pi_L^*K_A)(\Pi_R^*K_B))$ is polyhomogeneous on $M^2_{rh}$ with index sets given as a section of $\nu(M^2_{rh})$ by
		
		\begin{equation*} \begin{gathered} 
				A_{\ff_j}+B_{\ff_j}+9 \mbox{ at }\ff_j; \ A_{\sff_j}+B_{\sff_j}+8 \mbox{ at }\sff_j;\ \infty\mbox{ at } \tf; \\
				(\overline\cup_k(A_{\hvff_{jk}}+B_{\hvff_{kl}}+5))\overline\cup(A_{\hvrf_j}+B_{\hvlf_j}+3)\mbox{ at }\hvff_{jl};\\
				(\overline\cup_k(A_{\hvff_{jk}}+B_{\hvrf_k}+3))\overline\cup(A_{\hvrf_j}+1)\overline\cup(A_{\ff_j}+B_{\hvrf_j}+5)\mbox{ at }\hvrf_j;\\
				(\overline\cup_k(A_{\hvlf_k}+B_{\hvff_{kl}}+3))\overline\cup(B_{\hvlf_l}+1)\overline\cup(A_{\hvlf_l}+B_{\ff_l}+5) \mbox{ at }\hvlf_l;\\
				A_{j0}\overline\cup(A_{\sff_j}+B_{j0}+4)\mbox{ at }E_{j0};\ B_{0j}\overline\cup(A_{0j}+B_{\sff_j}+4)\mbox{ at }E_{0j}.
			\end{gathered} 
		\end{equation*}
		
		Finally, we use Proposition \ref{prop:densityblowup} to go back from sections of $\nu(M^2_{rh})$ to sections of $\nu(\Omega\times\Omega\times[0,1))$, then divide by $2T$ to go back to $dt dz dz''$. The index sets of the composition are thus
		\begin{equation*}
			\begin{gathered}  \mbox A_{\ff_j}+B_{\ff_j}+4 \mbox{ at }\ff_j;\ A_{\sff_j}+B_{\sff_j}+4 \mbox{ at }\sff_j; \ \infty\mbox{ at }\tf \mbox{ and } \td;\\
				(\overline\cup_k(A_{\hvff_{jk}}+B_{\hvff_{kl}}+2))\overline\cup(A_{\hvrf_j}+B_{\hvlf_j})\mbox{ at }\hvff_{jl};\\
				(\overline\cup_k(A_{\hvff_{jk}}+B_{\hvrf_k}+2))\overline\cup(A_{\hvrf_j})\overline\cup(A_{\ff_j}+B_{\hvrf_j}+4)\mbox{ at }\hvrf_j;\\
				(\overline\cup_k(A_{\hvlf_k}+B_{\hvff_{kl}}+2))\overline\cup(B_{\hvlf_l})\overline\cup(A_{\hvlf_l}+B_{\ff_l}+4) \mbox{ at }\hvlf_l;\\
				A_{j0}\overline\cup(A_{\sff_j}+B_{j0}+4)\mbox{ at }E_{j0};\ B_{0j}\overline\cup(A_{0j}+B_{\sff_j}+4)\mbox{ at }E_{0j}.
			\end{gathered} 
		\end{equation*}
	\end{proof}
	
	\begin{remark} 
		It is instructive to compare our composition formula to that of Mazzeo-Vertman \cite[Theorem 5.3]{MaVe}. Indeed, the two settings coincide in the special case of a surface with boundary and no vertices. Our faces ff, hvff, hvrf, and hvlf do not exist in that case. Moreover, in the notation of Mazzeo-Vertman, $\ell=A_{\sff_j}+4$ and $\ell'=B_{\sff_j}+4$ (see Definition 3.1, and note that the dimension of the base $b=1$), so our calculations are in agreement.
	\end{remark}

	\section{The heat kernel on a surface with boundary} \label{halfspace3} 
	In this section we will build the heat kernel for a surface with smooth boundary in the absence of conical singularities.  This construction has been performed in the Neumann and Dirichlet settings by Grieser \cite{grieser}; the Robin case we give here is new.  The work we do here shall be used in the later construction of the heat kernel for a surface with corners.
	We follow the usual geometric microlocal strategy. Namely, we specify models at various boundary hypersurfaces to which $\{T=0\}$ lifts in the heat space, then look for a pc function which has those models as its leading order behavior at each boundary hypersurface. In order for this method to work the models must be compatible with each other at the surfaces where they intersect, in the sense that their restrictions to the intersection must be identical, as otherwise no pc function with the specified leading order behavior can exist. 
	
	This setting is a special case of the setting of surfaces with corners, and as such the double heat space will be a special case of the space $M_h^2$ constructed in the previous section. Many of the blow-ups are now trivial; in this setting, we need just two blowups, the first of which is
	$$[ \Omega \times \Omega \times [0, \infty); \{ (z, z, 0) : z \in \pa \Omega \}].$$
	The new boundary face is sf, and in the absence of corners, the resulting space is the reduced heat space $M^2_{rh}$.  To complete the construction, we perform one more blowup, 
	$$[ M^2 _{rh} ; \{ (z, z, 0) : z \in \Omega \}] = M^2 _h.$$
	The resulting blown-up face is td.  Thus, the heat space for a surface with smooth boundary has boundary faces sf, td, tf, as well as the two side faces $E_{10}$ and $E_{01}$, which are the lifts of $\pa \Omega \times \Omega \times [0, \infty)$ and $\Omega \times \pa \Omega \times [0, \infty)$, respectively. The boundary hypersurfaces which comprise the lift of $\{T=0\}$ to $M^2 _h$ are therefore sf, td, and tf.

	\begin{figure} 
		\begin{tabular}{|c|c|c|c|} 
			\hline
			bdry hypersurface & sf & td & tf \\ 
			\hline
			model hk  & hk for half plane & hk for $\R^2$ & $\infty$ order vanishing \\ 
			\hline
			compatibility check & tf and td & sf and tf & sf and td \\ 
			\hline
		\end{tabular}
		\caption{The model heat kernels (abbreviated hk) and the boundary faces which intersect are given above.  Note that the heat kernel for the half plane is taken with the corresponding boundary condition:  Dirichlet, Neumann, or Robin.}  
	\end{figure} 
	
	\subsection{Heat kernels on the half-plane} 
	Let $(x,y)$ be the usual Cartesian coordinates on $\mathbb R^2$, and consider the half-space
	\begin{equation}
		\mathbb R^2_{+}:=\{(x,y)\in\mathbb R^2:y\geq 0\}.
	\end{equation}
	The heat kernel on all of $\mathbb R^2$ is
	\begin{equation}
		H_{\mathbb R^2}(t,x,y,x',y'):=\frac{1}{4\pi t}\exp\left[-\frac{(x-x')^2+(y-y')^2}{4t} \right].
	\end{equation}
	
	\subsubsection{The Neumann and Dirichlet heat kernels} \label{s:hkhalf} 
	By the method of images, 
	the Neumann heat kernel on $\mathbb R^2_+$ is
	\begin{equation} \label{eq:neumannhalfspace}
		\begin{gathered}  
			H_{\mathbb R^2}(t,x,y,x',y')+H_{\mathbb R^2}(t,x,y,x',-y') =  \\
			\frac{1}{4\pi t}\exp\left[-\frac{(x-x')^2}{4t}\right]\left(\exp\left[-\frac{(y-y')^2}{4t}\right]+\exp\left[\frac{-(y+y')^2}{4t}\right]\right).  
		\end{gathered} 
	\end{equation}
	The first term above is known as the direct term, whereas the second term is known as the reflected term or image term. We examine the behavior of this Neumann heat kernel on the double heat space $M_h^2$, albeit in the simpler setting where there are no corners. Although we will not prove that this heat kernel is pc on the double space in this section (it is a consequence of later work), we will use our examination to determine the appropriate pc models in the next section.
	
	To examine the model heat kernels for the half-space, let us examine the blow-ups in local coordinates. To get $M_h^2$, we first blow up
	\[\{T=y=y'=0; x=x'\},\]
	then blow up the lift of the diagonal at $T=0$:
	\[\{T=0; y=y'; x=x'\}.\]
	After the first blow-up, coordinates near the interior of the new face sf, away from the intersection with $\{T=0\}$ (where the second blow-up takes place), are given by
	\begin{equation}\label{eq:modelcoordsintsf}
		X:=\frac{x-x'}{T};\ \xi:=\frac{y}{T};\ \xi'=\frac{y'}{T};\ x';\mbox{ and }T.
	\end{equation}
	In these coordinates, the expression \eqref{eq:neumannhalfspace} becomes
	\begin{equation}\label{eq:neumannsfinterior}
		T^{-2}\frac{1}{4\pi}\exp\left[-\frac 14X^2 \right]\left(\exp\left[-\frac 14(\xi-\xi')^2\right]+\exp\left[-\frac 14(\xi+\xi')^2\right]\right).
	\end{equation}
	To encode this, we write
	\begin{align}\label{eq:modelminus2sfN} 
		&\mathcal H_{-2,\sff,N}:= \\ 
		& \frac{1}{4\pi}\exp\left[-\frac 14X^2\right]\left(\exp\left[-\frac 14(\xi-\xi')^2\right]+\exp\left[-\frac 14(\xi+\xi')^2\right]\right) \nonumber 
	\end{align} 
	
	and say that $\mathcal H_{-2,\sff,N}$, which we view as a function on sf, is the leading order model of the Neumann heat kernel at the face sf, appearing at order $T^{-2}$. What this means is that the Neumann heat kernel, in a coordinate patch near the interior of sf, is given at least to leading order (in this case, exactly) by
	\[T^{-2}\mathcal H_{-2,\sff,N}.\]
	
	After the second blow-up, coordinates near the interior of the new face td, away from $y=0$, are given by
	\begin{equation}\label{eq:modelcoordsinttd}
		X=\frac{x-x'}{T};\ Y:=\frac{y-y'}{T};\ x';\ y';\mbox{ and }T.
	\end{equation}
	Consider \eqref{eq:neumannhalfspace} in these coordinates. Away from $y=0$ the image term is $O(T^{\infty})$, and so \eqref{eq:neumannhalfspace} becomes
	\begin{equation}\label{eq:tdinterior}
		T^{-2}\frac{1}{4\pi}\exp\left[-\frac 14X^2\right]\exp\left[-\frac 14Y^2\right]+ O(T^{\infty}).
	\end{equation}
	So we may similarly define
	\begin{equation}\label{eq:modelminus2td}
		\mathcal H_{-2,\td}:=\frac{1}{4\pi}\exp\left[-\frac 14X^2\right]\exp\left[-\frac 14Y^2\right],
	\end{equation}
	where we have omitted the $N$ since the model will be the same for all boundary conditions. The leading order of the Neumann heat kernel at td is given by
	\[T^{-2}\mathcal H_{-2,\td}.\]
	
	To check for compatibility, we show that our model for the Neumann heat kernel on $\mathbb R^2_{+}$ is compatible with the model we have defined at td. Specifically, we want
	\[\mathcal H_{-2,\sff,N}|_{\sff\cap \td}=\mathcal H_{-2,\td}|_{\sff\cap \td}.\]
	The coordinate patches we have described to this point are not necessarily valid systems of coordinates near the intersection sf$\cap$td. However, it is easy enough to show that
	\begin{equation}\label{eq:modelcoordssfandtd}
		\eta:=\frac{T}{y'}=\frac{1}{\xi'};\ X;\ Y;\ x';\ y'
	\end{equation}
	are valid coordinates in a neighborhood of this intersection, away from tf. In this new coordinate patch,
	\[y=YT+y',\mbox{ so }\xi=Y+\frac{1}{\eta},\mbox{ so }\xi+\xi'=Y+\frac{2}{\eta},\]
	and we get
	\[\mathcal H_{-2,\sff,N}=\frac{1}{4\pi}\exp\left[-\frac 14X^2\right]\left(\exp\left[-\frac 14Y^2\right]+\exp\left[-\frac 14(Y+\frac{2}{\eta})^2\right]\right);\]
	\[\mathcal H_{-2,\td}=\frac{1}{4\pi}\exp\left[-\frac 14X^2\right]\exp\left[-\frac 14Y^2\right].\]
	Restricting to sf$\cap$td means letting $\eta \to 0$ in the first term, which corresponds to approaching sf$\cap$td from the interior of sf, and letting $y' \to 0$ in the second term, which corresponds to approaching sf$\cap$td from the interior of td.  We see immediately that the second exponential in $\mathcal H_{-2,\sff,N}$ tends to zero when $\eta$ tends to zero, and $H_{-2,\td}$ is independent of $y'$, so the restrictions are well-defined and they match. This proves compatibility of $\mathcal H_{-2,\sff,N}$ and $\mathcal H_{-2,\td}$.
	
	An identical analysis works for the Dirichlet heat kernel; the only difference is the sign of the image term. So we write
	\begin{align}\label{eq:modelminus2sfD}
		&\mathcal H_{-2,\sff,D} := \\ 
		&\frac{1}{4\pi}\exp\left[-\frac 14X^2\right]\left(\exp\left[-\frac 14(\xi-\xi')^2\right]-\exp\left[-\frac 14(\xi+\xi')^2\right]\right),  \nonumber 
	\end{align}
	and just as before, $\mathcal H_{-2,\sff,D}$ is compatible with $\mathcal H_{-2,\td}$.
	
	Note that the models $\mathcal H_{-2,\sff,N}$ and $\mathcal H_{-2,\sff,D}$ themselves satisfy Neumann or Dirichlet boundary conditions, respectively. In particular, looking at the expansion of $\mathcal H_{-2,\sff,N}$ in $\xi$ at $\{\xi=0\}$ (i.e. $y=0$), we observe that there is a complete Taylor expansion with no first-order term. If we take $\frac{\partial}{\partial \xi}\mathcal H_{-2,\sff,N}$ and restrict to $\{\xi=0\}$, we get zero. Similarly, $\mathcal H_{-2,\sff,D}$ has a Taylor expansion with no zeroth-order term; that is, its restriction to $\{\xi=0\}$ is zero.
	
	It is also useful to consider the heat operator, lifted from the left (that is, acting in the unprimed coordinates):
	\[\mathcal L:=\partial_t-\partial_{xx}-\partial_{yy}=\frac{1}{2T}\partial_T-\partial_{xx}-\partial_{yy}.\]
	This operator lifts under the blow-down maps to an operator on the double heat space, which we also call $\mathcal L$, abusing notation. The lift of the operator $t\mathcal L=T^2\mathcal L$ is more useful, because $T^2 \cL$ lifts to the double heat space to be a b-operator, except in a neighborhood of $E_{10}$, where it is merely smooth. In particular $\rho_{E_{10}}^2T^2\cL$ is a b-operator. Therefore $T^2\cL$ (1) preserves polyhomogeneity and (2) preserves infinite order vanishing at tf. These two facts shall be useful.

	The operator $T^2\cL$ can be analyzed in the coordinate systems \eqref{eq:modelcoordsintsf}, \eqref{eq:modelcoordsinttd}, and \eqref{eq:modelcoordssfandtd}. In the coordinates \eqref{eq:modelcoordsintsf} we have by a chain rule calculation:
	\begin{equation}\label{eq:liftintsf}
		T^2\mathcal L=\frac 12T\partial_{T}-\partial_{XX}-\frac 12X\partial_X-\partial_{\xi\xi}-\frac 12\xi\partial_{\xi} - \frac 1 2 \xi' \pa_{\xi'}.
	\end{equation}
	In the coordinates \eqref{eq:modelcoordsinttd}, we have
	\begin{equation}\label{eq:liftinttd}
		T^2\mathcal L=\frac 12T\partial_{T}-\partial_{XX}-\frac 12X\partial_X-\partial_{YY}-\frac 12Y\partial_Y.
	\end{equation}
	Finally, in the coordinates \eqref{eq:modelcoordssfandtd}, we get
	\begin{equation}\label{eq:liftsfandtd}
		T^2\mathcal L=\frac 12\eta\partial_{\eta}-\partial_{XX}-\frac 12X\partial_X-\partial_{YY}-\frac 12Y\partial_Y.
	\end{equation}
	The point of all of this is to show that our leading order models solve model problems at their designated boundary hypersurfaces.  
	Specifically, since the heat kernel solves the heat equation, we must have
	\begin{eqnarray}\left .\beta^* (T^2 \cL) \beta^*(T^{-2} \cH_{-2, \td}) \right|_{\td} &=& \left . \beta^* (T^2 \cL) \beta^*(T^{-2} \cH_{-2, \sff, N}) \right|_{\sff} \nonumber \\ 
		&=& \left . \beta^* (T^2 \cL) \beta^*(T^{-2} \cH_{-2, \sff, D}) \right|_{\sff} = 0. \label{eq:halfspacemodelproblemswork0}  \end{eqnarray}
	Observe additionally that when lifting to the double heat space, in the coordinate systems near these boundary faces, the three factors of $\cH$ are independent of the $T$ coordinate.  Consequently,
	$$\frac 1 2 T \pa_T (T^{-2} \cH) = - T^{-2} \cH,$$ 
	and so 
	\begin{eqnarray}\label{eq:halfspacemodelproblemswork}
		(T^2\mathcal L-\Id)|_{\sff}\mathcal H_{-2,\sff,N} &=& (T^2\mathcal L-\Id)|_{\sff}\mathcal H_{-2,\sff,D} \\ 
		&=& (T^2\mathcal L-\Id)|_{\td}\mathcal H_{-2,\td}=0. \nonumber 
	\end{eqnarray}
	
	\begin{remark} The Neumann and Dirichlet heat kernels for a half-space are indeed pc on $M_h^2$. This can be seen directly in local coordinates, and alternately follows from a reflection argument similar to the proof of Lemma \ref{lem:yesitspc}.
	\end{remark}
	
	\subsubsection{The Robin heat kernel}
	Now we consider the heat kernel on $\mathbb R^2_{+}$ with a Robin boundary condition, namely
	\begin{equation} \label{eq:robinbc}  \left . \frac{\partial u(x,y)}{\partial y}\right|_{y=0}=\kappa u(x,0),\mbox{ for some constant }\kappa>0. \end{equation} 
	We recall that this condition is for the \em inward \em pointing normal derivative.  The explicit expression for this heat kernel is known \cite{bondfull}. It is
	\[H_{\mathbb R_+^2,\textrm{Robin}}:=H_{\mathbb R_{+}^2,\textrm{Neumann}}+H_{corr},\]
	where
	\[H_{corr}(t,x,y,x',y'):=-\frac{\kappa e^{\kappa(y+y')}e^{\kappa^2 t}}{\sqrt{4\pi t}}\exp\left[-\frac{1}{4t}(x-x')^2\right]\mbox{erfc}\left(\frac{y+y'}{\sqrt{4t}}+\kappa\sqrt t\right).\]
	Recall that the complementary error function is smooth in $z$, bounded by $1$ for $z\geq 0$, and decaying to infinite order as $z\to\infty$:
	$$\erfc(z) = 1 - \erf(z) = \frac{2}{\sqrt \pi} \int_z ^\infty e^{-s^2}ds.$$
	
	Let us examine the behavior of $H_{corr}$ in the coordinate systems \eqref{eq:modelcoordsintsf} and \eqref{eq:modelcoordssfandtd}. In \eqref{eq:modelcoordsintsf}:
	\[H_{corr}=-T^{-1}\frac{\kappa e^{\kappa T(\xi+\xi')}e^{\kappa^2 T^2}}{2\sqrt{\pi}}\exp\left[-\frac 14X^2\right]\mbox{erfc}\left(\frac 12(\xi+\xi')+\kappa T\right).\]
	The restriction of $TH_{corr}$ to sf, that is to $T=0$, is well-defined. Based on our previous notation, we give it a name:
	\begin{equation}\label{eq:modelminus1sfR}
		\mathcal H_{-1,\sff,R}:=-\frac{\kappa}{2\sqrt{\pi}}\exp\left[-\frac 14X^2\right]\mbox{erfc}\left(\frac 12(\xi+\xi')\right).
	\end{equation}
	On the other hand, in the coordinate system \eqref{eq:modelcoordssfandtd} that is valid near the intersection of sf and td, we have
	\[H_{corr}=-T^{-1}\frac{\kappa e^{\kappa y'(\eta Y+2)}e^{\kappa^2\eta^2(y')^2}}{2\sqrt{\pi}}\exp\left[-\frac 14X^2\right]\mbox{erfc}\left(\frac 12Y+\frac 1{\eta}+\kappa\eta y' \right).\]
	We approach td$\cap$sf from the interior of sf by letting $\eta \to 0$. As this happens, the erfc function decays to infinite order due to the $1/\eta$ term in its argument, and we see that $H_{corr}$ vanishes to infinite order at td$\cap$sf. Thus $\cH_{-1,\sff,R}$ vanishes to infinite order at td, and so adding $H_{corr}$ to the Neumann heat kernel does not disrupt compatibility at td$\cap$sf.
	
	The Robin correction term also solves a model problem at sf. The model problem in this case is slightly different, because $\cH$ has a factor of $T^{-1}$ rather than $T^{-2}$.  Consequently, 
	$$\frac 1 2 T \pa_T (T^{-1} \cH) = - \frac 1 2 T^{-1} \cH,$$ 
	so here our model problem is 
	\begin{equation}\label{eq:minus1sfRmodelworks}
		\left . \left(t\mathcal L-\frac 12\Id\right)\right|_{\sff}\mathcal H_{-1,\sff,R}=0.
	\end{equation}
	For the sake of completeness we include this calculation.  First, we compute (dropping the subscripts for notational simplicity) 
	$$\pa_X \cH = \frac{\kappa X}{4\sqrt \pi} e^{-X^2/4} \erfc \left( \frac{\xi + \xi'}{2} \right) \implies - \frac{X}{2} \pa_X \cH = \frac{X^2}{4} \cH.$$
	Similarly we compute 
	$$\pa_{XX} \cH = \frac{\kappa}{4\sqrt \pi} e^{-X^2/4} \erfc \left( \frac{\xi + \xi'}{2} \right) - \frac{\kappa X^2}{8\sqrt \pi} e^{-X^2/4} \erfc \left( \frac{\xi + \xi'}{2} \right),$$
	thus 
	$$-\pa_{XX} \cH = \frac 1 2 \cH - \frac{X^2}{4} \cH, \quad - \frac{X}{2} \pa_X \cH - \pa_{XX} \cH = \frac 1 2 \cH.$$
	Noting that $\erfc'(z) = - e^{-z^2}$, we also compute 
	$$\pa_\xi \cH = \frac{\kappa}{4 \sqrt \pi} e^{-X^2/4} e^{-(\xi+\xi')^2/4}, \quad - \frac{\xi}{2} \pa_\xi \cH = \frac{-\xi \kappa}{8\sqrt \pi} e^{-X^2/4} e^{-(\xi+\xi')^2/4}.$$
	By the symmetry in $\xi$ and $\xi'$, we also have 
	$$- \frac{\xi'}{2} \pa_{\xi'} \cH = \frac{-\xi' \kappa}{8\sqrt \pi} e^{-X^2/4} e^{-(\xi+\xi')^2/4}.$$
	Finally, we compute 
	$$\pa_{\xi \xi} \cH = - \frac{\kappa (\xi + \xi')}{8 \sqrt \pi} e^{-X^2/4} e^{-(\xi+\xi')^2/4}, \quad - \pa_{\xi \xi} \cH = \frac{\kappa (\xi + \xi')}{8 \sqrt \pi} e^{-X^2/4} e^{-(\xi+\xi')^2/4}.$$
	In total, we therefore have 
	\begin{equation*} \begin{gathered} 
			T^2\cL \cH = \frac 1 2 \cH + \frac{-\xi \kappa}{8\sqrt \pi} e^{-X^2/4} e^{-(\xi+\xi')^2/4} + \frac{-\xi' \kappa}{8\sqrt \pi} e^{-X^2/4} e^{-(\xi+\xi')^2/4} \\ + \frac{\kappa (\xi + \xi')}{8 \sqrt \pi} e^{-X^2/4} e^{-(\xi+\xi')^2/4}= \frac 1 2 \cH,
	\end{gathered} \end{equation*}
	verifying \eqref{eq:minus1sfRmodelworks}. 
	
	The upshot of all of this is that the Robin heat kernel on a half-space may be seen as a correction of the Neumann heat kernel, where the correction is lower order in the sense that it appears in the asymptotic behavior half an order below the leading order in the asymptotic regime corresponding to sf, namely at $T^{-1}$ rather than $T^{-2}$. The correction also vanishes to infinite order at the $T=0$ diagonal in the interior, indicating that it has no effect on the interior heat asymptotics. So the Robin heat kernel on a half space has:
	\begin{itemize}
		\item Leading order behavior at td of order $T^{-2}$ given by $\mathcal H_{-2,\td}$, in particular the same as the Neumann and Dirichlet heat kernels;
		\item Leading order behavior at sf of order $T^{-2}$ given by $\mathcal H_{-2,\sff,N}$, in particular the same as the Neumann heat kernel; and
		\item Additional sub-leading order behavior at sf of order $T^{-1}$ given by $\mathcal H_{-1,\sff,R}$ (and some subsequent terms at higher powers of $T$).  However, there is no additional sub-leading order behavior at td.
	\end{itemize}
	
	\subsection{Construction of the heat kernel on a surface with boundary} \label{s:mwb} 
	We are now poised to construct the Dirichlet, Neumann, and Robin heat kernels on a surface with boundary,  prove they are pc on the double heat space, and identify their leading order terms in their pc expansion. Our construction is inspired by \cite{MaVe},  \cite{grieser} and indeed has already been done in \cite{MaVe} for Dirichlet boundary conditions (which are the Friedrichs extension for a cone-edge structure with one-dimensional edge and zero-dimensional link).  These references were both largely inspired and guided by \cite{tapsit}.  
	
	We use the same double space $M_h^2$ that we have been using, which in the boundary-only case for surfaces is the same as the double space of \cite{MaVe}, with a one-dimensional edge and a zero-dimensional cone link. Note that the faces have the same names, with the exception of our $E_{10}$ which corresponds to rf in \cite{MaVe} and our $E_{01}$ which corresponds to lf. Our composition formula, Theorem \ref{thm:comp}, agrees with that of \cite{MaVe} as well.
	
	Throughout, we will use boundary normal coordinates $(x,y)$ on our surface with boundary $\Omega$. In these coordinates, the boundary is defined by $y=0$. The Riemannian metric in these boundary normal coordinates near the boundary takes the form 
	$$g(x,y)dx^2+dy^2,$$ 
	with $g(0,y)=1$, and $g(x,y)$ smooth in $x$ and $y$. The Laplacian has the following expression:
	\[\Delta:=-\partial_{xx}-\partial_{yy}+a_1(x,y)\partial_{xx}+a_2(x,y)\partial_x+a_3(x,y)\partial_y,\]
	where $a_1$, $a_2$, and $a_3$ are smooth, with $a_1(x,0)=0$ and $a_2(x,0)=0$.  In the interior, we let $z$ be a local coordinate patch on $M$; then let $Z=(z-z')/T$ and use $(T,Z,z')$ near the interior of td.
	
	\subsubsection{The Dirichlet and Neumann heat kernels}
	Since the heat space here only has boundary faces td, sf, tf, and the side faces, for an index family $\mathcal F=(F_{\td},F_{\sff},F_{E10},F_{E01})$, define $\mathcal A^{\mathcal F}_{h}$ to be the space of kernels in $\mathcal A^{\mathcal F}_{h}(\Omega^2_h)$, as functions of $(T,z,z')$, which vanish to infinite order at tf. Similarly, we define $\Psi^{a,b,c,d}$ to be the set of operators whose kernels are in $\mathcal A ^{\mathcal F} _h$ for some index family $\mathcal F$ which has leading orders $a$, $b$, $c$, $d$, at the corresponding faces.
	
	Let us begin with the Neumann heat kernel. We follow \cite{MaVe}. We construct first a parametrix:
	\begin{proposition} There exists an element of $\Psi^{-2,-2,0,0}$ whose Schwartz kernel 
		$H^{(1)}$ satisfies Neumann boundary conditions in the left (unprimed) variable, whose limit as $T \to 0$ is $\delta(z-z')$, and with
		\[T^2\mathcal LH^{(1)}\in\mathcal A_h^{\infty,-1,0,0}.\]
	\end{proposition} 
	\begin{proof} The idea is to solve our model problems to infinite order at td and first order at sf, and to do so in a way that satisfies Neumann boundary conditions.
		
		In the interior of td, we use the ansatz
		\begin{equation}\label{eq:ansatzattd}
			H^{(1)}(T,Z,z')\sim\sum_{j=0}^{\infty}T^{-2+j}\mathcal H_{-2+j,\td}(Z,z').
		\end{equation}
		As in \cite{MaVe}, we formally apply $t\mathcal L$ to this expansion and set the result equal to zero.  We can solve, inductively, for each coefficient function $\mathcal H_{-2+j,td}$. For example, the equation for $j=0$ is
		\[-T^{-2}\mathcal H_{-2,\td}-T^{-2}\left (\partial_{ZZ}+\frac 12Z\partial_Z \right)\mathcal H_{-2,\td}=0.\]
		By direct computation \eqref{eq:halfspacemodelproblemswork} letting  $\mathcal H_{-2,\td}$ be the expression \eqref{eq:modelminus2td}, namely
		\[\mathcal H_{-2,\td}(Z,z')=\frac{1}{4\pi}e^{-\frac 14Z^2},\] 
		it solves this equation for $j=0$. For the higher order terms, we have to expand $T^2\mathcal L$ in a power series in $T$, and more terms than just its restriction to td will be involved. Nevertheless, one may show inductively that there exist terms $\mathcal H_{-2+j,\td}(Z,z')$ for all $j$ that satisfy the formal ansatz. These terms each decay rapidly in $Z$. Since $\Omega$ is a subset of a smooth manifold $M$ and this construction is uniform over the interior of $M$, all terms $\mathcal H_{-2+j,\td}$ are smooth up to sf$\cap$td. The result also satisfies the delta function initial condition. We omit the details, as they may be found in \cite{MaVe} and \cite{tapsit}, as well as other references. 
		
		We verified in \S \ref{s:hkhalf} that $\mathcal H_{-2,\sff,N}$, defined in \eqref{eq:modelminus2sfN}, is compatible with $\mathcal H_{-2,\td}$, so it is possible to choose an element of $\Psi^{-2,-2,0,0}$ whose kernel $H^{(1)}$ simultaneously has leading order $T^{-2}\mathcal H_{-2,\sff,N}$ at sf and has full expansion \eqref{eq:ansatzattd} at td, which vanishes to infinite order at tf, and which is smooth down to $E_{10}$ and $E_{01}$.  In fact it is also possible to choose $H^{(1)}$ to satisfy Neumann boundary conditions. To see this, examine the expansion of $\mathcal H_{-2,\sff,N}$ as we approach $E_{10}$ and $E_{01}$. Boundary defining functions for those faces are $\xi$ and $\xi'$ respectively. Indeed, \eqref{eq:modelminus2sfN} is smooth in $\xi$ and $\xi'$, and it has no order 1 term at $\xi=0$ or $\xi'=0$. We may thus choose $H^{(1)}$ so that there is no term of order 1 in its expansions at $E_{10}$ and $E_{01}$. This $H^{(1)}$ satisfies Neumann boundary conditions, as claimed.
		
		It remains to show that 
		\[T^2\mathcal LH^{(1)}\in\mathcal A_h^{\infty,-1,0,0}.\]
		First we have to show that $T^2\mathcal LH^{(1)}$ is polyhomogeneous. However, the operator $T^2 \cL$ lifts to one which is tangent to all boundary hypersurfaces except for $E_{10}$, and at $E_{10}$ it may be written as $\rho_{E_{10}}^{-2}$ times such an operator. Since such operators preserve polyhomogeneity and also preserve the infinite order vanishing at tf, the polyhomogeneity follows.
		
		Now we compute the leading orders.  Since $H^{(1)}$ has the full expansion \eqref{eq:ansatzattd}, which is annihilated by $t\mathcal L$, we see that $T^2\mathcal LH^{(1)}$ has order $\infty$ at td.
		
		At sf, we claim that the model problem is the same as for a half-space.  We have solved the model problem for a half-space to first order, so we get an improvement of one order, from $-2$ to $-1$. To see this, compute $t\mathcal L$ in the coordinates \eqref{eq:modelcoordsintsf}. We get
		\begin{multline}\label{eq:liftintsfbnc}
			T^2\mathcal L=\frac 12T\partial_{T}-\partial_{XX}-\frac 12X\partial_X-\partial_{\xi\xi}-\frac 12\xi\partial_{\xi} - \frac 1 2 \xi' \pa_{\xi'}\\
			+a_1(XT+x',\xi T)\partial_{XX}+Ta_2(XT+x',\xi T)\partial_X+Ta_3(XT+x',\xi T)\partial_{\xi}.
		\end{multline}
		We apply $T^2\mathcal L$ to our pc expansion at sf, namely
		\[T^{-2}\mathcal H_{-2,\sff,N}+T^{-1}\mathcal H_{-1,\sff,N}+\dots.\]
		Since $T^2\mathcal L$ is tangent to sf, the leading order of the result will be at worst $-2$. We claim it is actually $-1$. Indeed, as with the half-space, the application of the first six terms in \eqref{eq:liftintsfbnc} to $T^{-2}\mathcal H_{-2,\sff,N}$ yields zero by \eqref{eq:halfspacemodelproblemswork}. Moreover, applying $T^2\mathcal L$ to only the terms with order $T^{-1}$ or higher yields something of order at most $-1$. So the only possible term of order $-2$ in the expansion of $t\mathcal L H^{(1)}$ at sf comes from
		\begin{align*}& T^{-2}(a_1(XT+x',\xi T)\partial_{XX}+Ta_2(XT+x',\xi T)\partial_X \\ 
			&+Ta_3(XT+x',\xi T)\partial_{\xi})\mathcal H_{-2,\sff,N}.\end{align*}
		However, the coefficients of the last two terms vanish at sf to first order in $T$, and since $a_1(x',0)=0$ identically, so does the coefficient of the first term. In other words, the Laplacian in boundary normal coordinates is the same as that for a half-space up to terms which are lower order at sf. Thus $t\mathcal L H^{(1)}$ has order $-1$ at sf, as desired.
		
		At $E_{01}$, $T^2\mathcal L$ is tangent to $E_{01}$, so the leading order remains unchanged at 0.  Finally, at $E_{10}$, at any point in the interior of $E_{10}$, we can use the coordinates $(T,x,y,x',y')$, in which $y$ is the defining function for $E_{10}$. So $H^{(1)}$ has a smooth expansion in $y$ down to $y=0$, with smooth dependence on all other variables. Applying $T^2\mathcal L$ would usually turn a term of order $y^{\gamma}$ into a term of order $y^{\gamma-2}$, but since $H^{(1)}$ has a smooth expansion, it stays smooth. So the leading order of $T^2\mathcal LH^{(1)}$ at $E_{10}$ is 0, completing the proof.
	\end{proof}
	
	As in \cite{MaVe}, this can be improved at $E_{10}$, the analogue of rf:
	\begin{proposition}\label{prop:improvementatrf}There exists an element $H^{(2)}\in\mathcal A_h^{-2,-2,0,0}$ which satisfies Neumann boundary conditions in the left variable, with 
		\[\lim_{t\to 0}H^{(2)}=\delta(z-z') \quad \textrm{and} \quad T^2\mathcal LH^{(2)}\in\mathcal A_h^{\infty,-1,\infty,0}.\]
	\end{proposition}
	\begin{proof} This is a standard argument, as in \cite[p. 32]{etdeo1} and  \cite{MaVe}. We use the fact that $T^2\mathcal L$ is elliptic in the $y$-direction to iteratively solve away the Taylor expansion of $T^2\mathcal LH^{(1)}$ at $E_{10}$. To be concrete, let $A_2$ be a pc kernel on $\Omega^2_h$, smooth at $E_{10}$, whose expansion in any coordinate neighborhood $(t,x,y,x',y')$ is
		\[\frac 12y^2\left(\mathcal LH^{(1)}\right)(t,x,0,x',y')+ O(y^3).\]
		Then $(T^2\mathcal L)A_2$ is pc as well, and its leading order term at $E_{10}$ is 
		\[\left(T^2\mathcal LH^{(1)}\right)(t,x,0,x',y').\] Furthermore, using the coordinates \eqref{eq:modelcoordsintsf} it is straightforward to see that $A_2$ may be chosen to be pc down to sf and have the same order as $y^2\mathcal LH^{(1)}$ at sf, namely $2-2=0$. So if we consider $H^{(1)}-A_2$, then $(T^2\mathcal L)(H^{(1)}-A_2)$ is pc and vanishes to first order, rather than zeroth order, at $E_{10}$. Moreover $H^{(1)}-A_2$ still satisfies Neumann boundary conditions.
		
		This construction may now be iterated to produce $A_3$, $A_4$, et cetera, so that $H^{(1)}-\sum_{j=2}^{k}A_j$ vanishes to order $k-1$ at $E_{10}$. The $A_j$ may be summed, and then setting
		\[H^{(2)}=H^{(1)}-\sum_{j=2}^{\infty}A_j\]
		gives us the result. Note that since each $A_j$ has order $-1$ or better at sf ($\mathcal LH^{(1)}$ has order $-3$ there, but $y^2$ has order 2), the leading order term of $H^{(2)}$ at sf is still $\mathcal H_{-2,\sff,N}$.
	\end{proof}
	
	Now, as in \cite{MaVe}, let $P^{(2)}=(t\mathcal L)H^{(2)}$. Then, as a kernel,
	\[\mathcal LH^{(2)}=\frac 1{T^2}P^{(2)}\in\mathcal A_h^{\infty,-3,\infty,0}.\]
	Kernels on $\Omega_h^2$ may be naturally identified as convolution operators on $[0,\infty)\times\Omega$, acting in the usual way. In this sense, as in \cite[(7.67)]{tapsit} we have
	\begin{equation}\label{eq:kerneltoconvop}
		\mathcal LH^{(2)}=\Id-\left(-\frac{1}{T^2}P^{(2)}\right).
	\end{equation}

	To see this for any $g\in C^{\infty}([0,\infty)\times \Omega)$, we have
	\[\mathcal LH^{(2)}*g(t)=(\partial_t+\Delta)\int_0^t[H^{(2)}g](s)(t-s)\, ds.\]
	By the fundamental theorem of calculus and the definition of $P^{(2)}$, this becomes
	\[ [H^{(2)}g(t)](0)+\int_0^t \left[\frac{1}{s}[P^{(2)}g](s)\right](t-s)\, ds.\]

	Since $[H^{(2)}g(t)](0)=g(0)$ by the delta function initial condition, we have \eqref{eq:kerneltoconvop}.
	
	To invert the right-hand side of \eqref{eq:kerneltoconvop}, we use the Neumann series
	\[\left(\Id+\frac{1}{T^2}P^{(2)}\right)^{-1}=\Id-\sum_{j=1}^{\infty}\left(-\frac{1}{T^2}P^{(2)}\right)^{j}=:\Id+P^{(3)}.\]
	By our composition Theorem \ref{thm:comp},
	since $-T^{-2}P^{(2)}\in\mathcal A_h^{\infty,-3,\infty,0}$, we obtain that for each $j$,
	\[\left(-\frac{1}{T^2}P^{(2)}\right)^j\in\mathcal A_h^{\infty,-4+j,\infty,0}.\]
	This series may therefore be asymptotically summed, and we obtain that $P^{(3)}\in\mathcal A_h^{\infty,-3,\infty,0}$.
	
	In fact this asymptotic sum is a legitimately convergent sum. This is asserted in \cite{MaVe} in the edge case and proven in \cite[p. 270]{tapsit} for compact manifolds; the same applies here as well.
	
	Finally, set
	\[H^{(3)}=H^{(2)}\left(\Id+P^{(3)}\right).\]
	By the definition of convolution, the Neumann boundary conditions, being satisfied by $H^{(2)}$, are also satisfied by $H^{(3)}$.
	Since $H^{(2)}\in\mathcal A_h^{-2,-2,0,0}$ and $P^{(3)}\in\mathcal A_h^{\infty,-3,\infty,0}$, our composition Theorem \ref{thm:comp} tells us that 
	\[H^{(3)}\in\mathcal A_h^{-2,-2,0,0}+\mathcal A_h^{\infty,-1,0,0}.\]
	Since $\mathcal LH^{(3)}=\Id$, $H^{(3)}$ satisfies the delta function initial condition. By uniqueness for the Neumann heat kernel on a manifold with boundary, $H^{(3)}$ must therefore be the true heat kernel.
	
	\begin{theorem}\label{thm:neumannheat}The Neumann heat kernel on $\Omega$ is pc on $\Omega^2_h$, and is an element of $\mathcal A_h^{-2,-2,0,0}$. Its expansion at td is given by \eqref{eq:ansatzattd}, and its leading term at sf is given by $\mathcal H_{-2,\sff,N}$. Moreover it is smooth down to both $E_{j0}$ and $E_{0j}$, and its expansion at sf is $T^{-2}$ times a smooth expansion.
	\end{theorem}
	\begin{proof} The heat kernel is smooth down to $E_{j0}$ because $H^{(2)}$ is smooth by construction, and the composition theorem implies $H^{(3)}$ is smooth. The required statement at sf follows from the same logic. Finally, it is smooth down to $E_{0j}$ because it is symmetric. For the leading term statements, the leading terms of $H^{(2)}$ have the claimed properties, and $H^{(2)}P^{(3)}$ vanishes rapidly at td and is lower order than $H^{(2)}$ at sf.  This completes the proof.
	\end{proof}
	
	An analogous theorem, proved in an identical fashion, holds for the Dirichlet heat kernel. Since the Dirichlet boundary condition is the Friedrichs extension for a one-dimensional cone, this is in fact a special case of \cite{MaVe}. Note that the Dirichlet boundary condition implies that the heat kernel vanishes to first order at $E_{10}$ and $E_{01}$:
	\begin{theorem}\label{thm:dirichletheat}[\cite{MaVe}]The Dirichlet heat kernel on $\Omega$ is pc on $\Omega^2_h$, and is an element of $\mathcal A_h^{-2,-2,1,1}$. Its expansion at td is given by \eqref{eq:ansatzattd}, and its leading term at sf is given by $\mathcal H_{-2,\sff,D}$. Moreover it is smooth down to both $E_{j0}$ and $E_{0j}$, and is $T^{-2}$ times a smooth expansion at sf.
	\end{theorem}
	
	\subsubsection{The Robin heat kernel} We now construct the Robin heat kernel on $\Omega$ as a correction, or perturbation, of the Neumann heat kernel on $\Omega$. To distinguish it, let $H_{Neumann}(t,z,z')$ be the Neumann heat kernel, which is pc on $\Omega^2_h$ by the previous section. Our Robin boundary condition is
	\[\left . \frac{\partial u(x,y)}{\partial y} \right|_{y=0}=\kappa(x)u(x,0).\]
	
	Our model will be the Robin heat kernel for a half-space with constant $\kappa$. With that in mind, define
	\[H^{(0)}_{Robin}:=H_{Neumann}-\frac{\kappa(XT+x')}{2\sqrt{\pi} T}\exp\left[-\frac 14X^2\right]\mbox{erfc}\left(\frac 12(\xi+\xi')\right).\]
	The distinction here is that now $\kappa$ depends on the variable $x$, which can be expressed in terms of the coordinates $X$, $T$, and $x'$ as $x=XT+x'$.

	Both terms are pc. The first term is an element of $\mathcal A_h^{-2,-2,0,0}$ and the second term is an element of $\mathcal A_h^{\infty,-1,0,0}$. Now we compute, using erfc$'(s)=-2e^{-s^2}/\sqrt{\pi}$, that
	\begin{equation} \label{eq:robinthingtokill} \begin{gathered}
			\left . \left(\frac{\partial}{\partial y}-\kappa\right)H^{(0)}_{Robin}\right|_{y=0} =  \\ 
			\left . \frac 1T\frac{\partial}{\partial\xi}\right|_{\xi=0}\left(-\frac{\kappa(XT+x')}{2\sqrt{\pi} T}\exp\left[-\frac 14X^2\right]\erfc\left(\frac 12(\xi+\xi')\right)\right) \\ 
			\left .-\kappa H^{(0)}_{Robin}\right|_{y=0} = \\   
			\frac{\kappa(XT+x')}{2\pi T^2}\exp\left[-\frac 14X^2\right]\exp\left[-\frac 14\xi'^2\right] \\ 
			-\kappa(XT+x') \left . H_{Neumann}\right|_{y=0} \\ 
			+\frac{(\kappa(XT+x'))^2}{2\sqrt{\pi} T}\exp\left[-\frac 14X^2\right]\mbox{erfc}\left(\frac 12\xi'\right). 
		\end{gathered} 
	\end{equation}
	Let this right-hand side be $c(T,X,x',\xi')$. To put it politely, this is not zero. We correct this defect by defining
	\[H^{(1)}_{Robin}:=H^{(0)}_{Robin}-ye^{-(y/T)^2}c(T,X,x',\xi').\]
	This fixes the Robin defect: the derivative in $y$ at $y=0$ of the second term above is precisely $-c(T,X,x',\xi')$, and $\kappa$ times this term is zero at $y=0$, so we have 
	\[ \left . \left(\frac{\partial}{\partial y}-\kappa(x)\right)H^{(1)}_{Robin}\right|_{y=0}=0.\]

	\begin{lemma}\label{lem:basicfactsrobinparam} The function $H^{(1)}_{Robin}$ is an element of $\mathcal A^{-2,-2,0,0}_h$, satisfying Robin boundary conditions in the left variable. At td, it has the same expansion as $H_{Neumann}$. At sf, the first two terms of its asymptotic expansion are the first two terms for $H_{Neumann}$ plus $T^{-1}\mathcal H_{-1,\sff,R}$, where
		\[\mathcal H_{-1,\sff,R}:=-\frac{\kappa(x')}{2\sqrt{\pi}}\exp\left[-\frac 14X^2\right]\erfc\left(\frac 12(\xi+\xi')\right).\]
		(Comparing to \eqref{eq:modelminus1sfR}, the only change is that $\kappa$ is now a function of $x'$ rather than a constant.) Furthermore,
		\[t\mathcal L H^{(1)}_{Robin}\in\mathcal A^{\infty,0,0,0}_h.\]
	\end{lemma}
	\begin{remark}Note that $H^{(1)}_{Robin}$ is a slightly better parametrix than $H^{(1)}$ was for the Neumann and Dirichlet problems.  This is because we have solved the model problem to two orders at sf, rather than just to first order. This is necessary because we want to identify the sub-leading term of the true Robin heat kernel at sf.
	\end{remark}

	\begin{proof}We just showed $H^{(1)}$ does satisfy the Robin boundary condition. Now we claim it is pc. At first it appears there are problems with the second term at $y=T=0$ away from the diagonal, and that we might need to blow up the intersection of $E_{10}$ and tf. However, at this intersection, $c(T,X,x',\xi')$ decays rapidly, so in fact $H^{(1)}_{Robin}$ is already pc. It decays to infinite order at tf because both terms in its definition do. We also have
		\begin{align*} 
			H_{Robin}^{(1)}-H_{Neumann} = &-ye^{-(y/T)^2}c(T,X,x',\xi') \\ & -\frac{\kappa(XT+x')}{2\sqrt{\pi} T} \exp\left[-\frac 14X^2 \right] \erfc \left(\frac 12(\xi+\xi')\right).
		\end{align*}
		It is immediate that the leading order of the second term at sf is in fact $T^{-1}\mathcal H_{-1,\sff,R}$, and that the second term vanishes to infinite order at td and is smooth up to all other boundary hypersurfaces. We claim that
		\begin{equation}\label{eq:claim1robinparam}
			ye^{-(y/T)^2}c(T,X,x',\xi')\in\mathcal A_h^{\infty,0,0,0},
		\end{equation}
		which immediately implies the statements concerning the asymptotic expansions of $H^{(1)}_{Robin}$ at td and sf.
		
		To prove \eqref{eq:claim1robinparam}, we need to check decay. The statements at $E_{10}$ and $E_{01}$ are obvious, and the presence of $ye^{-(y/T)^2}$ implies the requisite infinite-order decay at td. The trickier face is sf. Since $y$ decays to first order at sf, and $e^{-(y/T)^2}$ is smooth, we need to examine $c(T,X,x',\xi')$ and show it has order at worst $-1$. The third term in \eqref{eq:robinthingtokill} is already order -1.  However, the first and second terms have order $-2$, but we shall compute that their difference has order $-1$. As $T\to 0$, because $H_{Neumann}$ is polyhomogeneous, its restriction to $y=0$ has an expansion at sf, and the leading term is \eqref{eq:modelminus2sfN}, so
		\begin{align*} \left . H_{Neumann}\right|_{y=0} &= 
			\left .\frac{1}{T^2}\mathcal H_{-2,\sff,N}\right|_{\xi=0}+ O\left(\frac 1T\right) \\
			& =\frac{1}{2\pi T^2}\exp\left[-\frac 14X^2\right]\exp\left[-\frac 14(\xi')^2\right]+ O\left(\frac 1T\right). \end{align*} 
		Examining \eqref{eq:robinthingtokill}, this $T^{-2}$ term here cancels the first $T^{-2}$ term, and thus $c(T,X,x',\xi')$ is $O(\frac 1T)$ at sf. Therefore the order of $ye^{-(y/T)^2}c(T,X,x',\xi')$ is at worst $1+(-1)=0$ at sf, proving the claim \eqref{eq:claim1robinparam}.
		
		Now consider $(t\mathcal L)H^{(1)}_{Robin}$. Since $\mathcal LH_{Neumann}=0$ because it is the heat kernel, we have
		\begin{align}  (t\mathcal L)H^{(1)}_{Robin} & =(t\mathcal L)(H^{(1)}_{Robin}-H_{Neumann}) \\ 
			& =(t\mathcal L)(-ye^{(-y/T)^2}c(T,X,x',\xi')) \nonumber \\ & -(t\mathcal L)\left(\frac{\kappa(XT+x')}{2\sqrt{\pi} T}\exp\left[-\frac 14X^2\right]\erfc\left(\frac 12(\xi+\xi')\right)\right). \nonumber
		\end{align}
		The first term is an element of $\mathcal A^{\infty,0,0,0}_h$ before applying $t\mathcal L$. Since $t\mathcal L$ is tangent to all  boundary hypersurfaces except $E_{10}$, it preserves the orders, and at $E_{10}$, $t\cL$ takes a smooth expansion to a smooth expansion.
		The second term is an element of $\mathcal A^{\infty,-1,0,0}_h$, but we can write the Taylor expansion of $\kappa(XT+x')$ in powers of $T$. The $T^0$ term is just $\kappa(x')$, yielding $\mathcal H_{-1,\sff,R}$, which $t\mathcal L$ annihilates by \eqref{eq:minus1sfRmodelworks}. All terms except the $T^0$ term are elements of $\mathcal A^{\infty,0,0,0}_h$, and thus remain so after the application of $t\mathcal L$. This completes the proof of Lemma \ref{lem:basicfactsrobinparam}.
	\end{proof}
	
	From here the argument is very similar to the Neumann and Dirichlet arguments.
	\begin{proposition} There exists an element $H^{(2)}_{Robin}\in\mathcal A_h^{-2,-2,0,0}$ which satisfies Robin boundary conditions in the left factor, with $\lim_{t\to 0}H^{(2)}=\delta(z-z')$, and with
		\[t\mathcal LH^{(2)}\in\mathcal A_h^{\infty,0,\infty,0}.\]
		Moreover the expansions of $H^{(2)}_{Robin}$ and $H^{(1)}_{Robin}$ are identical for all terms at td and for the terms of order $-2$ and $-1$ at sf.
	\end{proposition}
	\begin{proof} We add terms at order $y^2$ and up at the face $E_{10}$, as in Proposition \ref{prop:improvementatrf}. Note that each of these terms is $y^k$ for $k\geq 2$ times a term which has order $-2$ at sf, and thus each of these terms has order greater than or equal to zero at sf, so there is no effect on the first two terms of the expansion there.
	\end{proof}
	Now we let
	\[P^{(2)}_{Robin}:=t\mathcal LH^{(2)}_{Robin};\ P^{(3)}_{Robin}:=-\sum_{j=1}^{\infty}\left(-\frac 1{T^2}P^{(2)}_{Robin}\right)^j.\]
	We have $T^{-2}P^{(2)}_{Robin}\in\mathcal A_h^{\infty,-2,\infty,0}$, so its $j$th power is an element of $A_h^{\infty,-4+2j,\infty,0}$ by the composition theorem, and thus $P^{(3)}_{Robin}\in\mathcal A_h^{\infty,-2,\infty,0}$. As before the sum is convergent, not just asymptotically convergent. 
	Then set
	\[H^{(3)}_{Robin}=H^{(2)}_{Robin}\left(\Id+P^{(3)}_{Robin}\right).\]
	This satisfies the Robin boundary conditions and the initial condition, so by uniqueness it is the true Robin heat kernel. By composition, $H^{(2)}_{Robin}P^{(3)}_{Robin}\in\mathcal A_h ^{\infty,0,0,0}$, so $H^{(3)}_{Robin}$ has the same first two terms at sf and same full expansion at td as $H^{(2)}_{Robin}$. We have now proved:
	\begin{theorem}\label{thm:bigrobinmwb}The Robin heat kernel on $\Omega$, with smooth non-negative Robin parameter $\kappa(x)$, is pc on $\Omega^2_h$, and is an element of $\mathcal A_h^{-2,-2,0,0}$, smooth down to both $E_{10}$ and $E_{01}$ and equal to $T^{-2}$ times a smooth expansion at sf. It is equal to the Neumann heat kernel on $\Omega$ plus a correction term which is an element of $\mathcal A_h^{\infty,-1,0,0}$ and which has leading order $\mathcal H_{-1,\sff,R}$ at sf.
	\end{theorem}
	\noindent Indeed, the only part of this we have not addressed is the smoothness, and it follows as in the Neumann case.

	\section{The heat kernel on a curvilinear polygonal domain}  \label{s:construction} 
	Consider a two-dimensional sector $S_{\gamma}$ with angle $\gamma\in(0,2\pi)$. We investigate heat kernels on $S_{\gamma}$ to serve as models for our eventual construction of the heat kernel on curvilinear polygonal domains. 
	
	\subsection{Properties of the heat kernel for an infinite sector}  \label{s:exactwedge} 
	We will consider the D-D, N-N, and D-N heat kernels.  Recall the expression from \cite[p. 592 (3.42)]{cheeger}, which in our setting is simply
	\begin{equation}\label{eq:heatkernelsectorformula}
		H(t,r,\theta,r',\theta')=\frac{1}{2t}\exp\left[-\frac{r^2+(r')^2}{4t}\right]\sum_{j=1}^{\infty}I_{\mu_j}\left(\frac{rr'}{2t}\right)\phi_j(\theta)\phi_j(\theta').
	\end{equation}
	Here $I_{\mu_j}$ are the modified Bessel functions, and $(\phi_j,\mu_j)$ are the eigenfunctions, and corresponding eigenvalues, of the appropriate eigenvalue problem (D-D, N-N, or D-N) on the interval $[0,\gamma]$.  
	
	The pc properties of \eqref{eq:heatkernelsectorformula} are not obvious from the expression alone.  They are equally non-obvious from the equivalent expression given by the inverse Laplace transform of the Green's function. However, we claim:
	\begin{lemma}\label{lem:yesitspc} In each of the three settings, D-D, N-N, and D-N, the heat kernel \eqref{eq:heatkernelsectorformula} is pc on our double space $(S_{\gamma})_h^2$.
	\end{lemma}

	\begin{proof} The proof is based on the reflection argument in \cite[\S 3]{polyakov}. Consider the D-D case for the moment. The sector $S_{\gamma}$ doubles to an infinite flat cone $C_{2\gamma}$, and if we let $L$ cut this cone in half, then we claim that the D-D heat kernel on $S_{\gamma}$ is  
		\begin{equation}\label{eq:doubledhkdirichlet}
			H_{S_{\gamma}}(t,r,\theta,r',\theta')=H_{C_{2\gamma}}(t,r,\theta,r',\theta')-H_{C_{2\gamma}}(t,r,\theta,\textrm{ref}_L(r',\theta')).
		\end{equation}
		Above $H_{C_{2\gamma}}$ is the Friedrichs heat kernel on $C_{2\gamma}$, and $\textrm{ref}_L$ is reflection across $L$. 
		Indeed it is clear that the difference of heat kernels satisfies the heat equation and the initial condition on $S_{\gamma}$, as well as the Dirichlet boundary condition. 
		By uniqueness of the heat kernel, we have \eqref{eq:doubledhkdirichlet}.
		
		The pc properties of $H_{S_{\gamma}}$ may now be deduced from those of $H_{C_{2\gamma}}$, as in \cite{polyakov}. By \cite{mooers}, \cite{MaVe}, $H_{C_{2\gamma}}$ is pc on a double heat space. In the notation of \cite{MaVe}, the $x$-coordinate is $r$, there is no $y$-coordinate, and the $z$-coordinate is $\theta$. The Mazzeo-Vertman heat space is not exactly the same as our heat space $(C_{2\gamma})_h^2$, as \cite{MaVe} do not create a face hvff, but nevertheless:
		
		\begin{proposition}\label{prop:MVheatvsourheat} The Mazzeo-Vertman heat space is a blow-down of $(C_{2\gamma})_h^2$, and therefore $H_{C_{2\gamma}}$ is pc on $(C_{2\gamma})_h^2$.
		\end{proposition}
		\begin{proof}Begin with the manifold with corners $[0,\infty)\times (C_{2\gamma})_0\times (C_{2\gamma})_0$. The Mazzeo-Vertman heat space is created by blowing up:
			\begin{itemize}
				\item $\{0\}\times \tilde V\times \tilde V$; and
				\item the lift of the interior $T=0$ diagonal.
			\end{itemize}
			From this space, we make a further blow-up at $[0,\infty)\times \tilde V\times \tilde V$. We claim that the resulting space is $(C_{2\gamma})^2_h$, which is all we need. Indeed, this further blow-up is disjoint from the lift of the interior $T=0$ diagonal and thus may be done second instead of third, by Proposition \ref{prop:blowupfacts}. It may then be done first instead of second, since nested blow-ups commute (again by Proposition \ref{prop:blowupfacts}). Hence our heat space $(C_{2\gamma})_h^2$ is a blow-up (indeed, an overblown version) of the Mazzeo-Vertman heat space.
		\end{proof}
		This takes care of the direct term, which is the first term in the right side \eqref{eq:doubledhkdirichlet}. The reflected term (the second term in the right side of \eqref{eq:doubledhkdirichlet}) is pc on a nearly identical space, the only difference being that we blow up the $T=0$ anti-diagonal $\{T=0,r=r',\theta=\textrm{ref}_L(\theta')\}$ in the last step rather than the $t=0$ diagonal. The lifts of the diagonal and anti-diagonal are not disjoint.  They intersect at the lift of $\{T=0,r=r',\theta=\theta'\in L\}$.  So in order to obtain a space on which both the direct and reflected terms are pc we blow up that lift before dealing with the diagonal and anti-diagonal. We do this: first blow up that lift, then blow up the diagonal and anti-diagonal, and we have obtained a space on which the direct and reflected terms are both pc.
		
		When we restrict our spatial arguments to lie in $S_{\gamma}$, we claim that this space is the double heat space $(S_{\gamma})^2_h$. Indeed, blowing up the lift of $\{T=0,r=r',\theta=\theta'\in L\}$ is precisely what is needed to create the face sf. We are doing it after blowing up hvff, hvlf, and hvrf, rather than before, but these blow-ups are disjoint (since we have already created ff) and therefore commute. The anti-diagonal does not appear once we have restricted our arguments to lie in $S_{\gamma}$. Therefore  $H_{S_{\gamma}}$ is pc on $(S_{\gamma})_h^2$, as desired.
		
		The argument for the N-N heat kernel is identical; there is a plus sign instead of a minus sign in \eqref{eq:doubledhkdirichlet}.  For the D-N heat kernel, we double twice, to the cone $C_{4\gamma}$, and use the method of images with four terms rather than two. The details are very similar and we omit them here. This proves Lemma \eqref{lem:yesitspc}.
	\end{proof}
	
	Having proven that $H_{S_{\gamma}}$ is pc on the double space, we may write down its leading order models at the various boundary hypersurfaces. We begin at ff. In the interior of ff, good coordinates are given by
	\begin{equation}\label{eq:modelcoordsintff}
		T=\sqrt t,\ R:=\frac{r}{T},\ R':=\frac{r'}{T},\ \theta,\ \theta'.
	\end{equation}
	In fact these coordinates are good uniformly down to hvlf and hvrf in the Mazzeo-Vertman double space, but to create $(S_{\gamma})_h^2$ we have made an additional blowup at hvff, which in these coordinates is $\{R=R'=0\}$. Fortunately this is not important for our present concerns. Writing \eqref{eq:heatkernelsectorformula} in the coordinates \eqref{eq:modelcoordsintff} gives \begin{equation}\label{eq:heatkernelsectorformulaff}
		H_{S_{\gamma}}=\frac{1}{2}T^{-2}\exp\left[-\frac 14(R^2+(R')^2)\right]\sum_{j=1}^{\infty}I_{\mu_j}\left(\frac 12RR'\right)\phi_j(\theta)\phi_j(\theta').
	\end{equation}
	This motivates the definition of the models
	\begin{multline}\label{eq:modelminus2ffall}
		\mathcal H_{-2,\ff,DD},\textrm{ resp. } H_{-2,\ff,DN},\textrm{ resp. } H_{-2,\ff,NN}\\ :=\frac 12\exp\left[-\frac 14(R^2+(R')^2)\right]\sum_{j=1}^{\infty}I_{\mu_j}\left(\frac 12RR'\right)\phi_j(\theta)\phi_j(\theta'),
	\end{multline}
	where $(\phi_j,\mu_j)$ are the eigenfunctions and eigenvalues of the appropriate problems on $[0,\gamma]$. It is then true that the leading term of the expansion of $H_{S_{\gamma}}$ at ff, with D-D, D-N, or N-N boundary conditions, is $T^{-2}\mathcal H_{-2,\ff,DD}$, $T^{-2}\mathcal H_{-2,\ff,DN}$, or $T^{-2}\mathcal H_{-2,\ff,NN}$, respectively.
	
	As usual, the heat kernel is decaying to infinite order at tf. We claim that its models at sf and at td are familiar:
	\begin{proposition} The leading order models at sf and td of $H_{S_{\gamma}}$ are the same as the models \eqref{eq:modelminus2sfD}, \eqref{eq:modelminus2sfN}, and \eqref{eq:modelminus2td} for a manifold with boundary, namely $\mathcal H_{-2,\sff,D}$ or $\mathcal H_{-2,\sff,N}$ at each of the two components of sf (depending on the boundary condition) and $\mathcal H_{-2,\td}$ at td.
	\end{proposition}
	\begin{proof}This is true because of locality; in fact, all the models are the same, not just the leading order. In any patch of sf away from ff, the spatial variables are near the boundary of $S_{\gamma}$ but bounded away from the corner. Since we are looking at short-time asymptotics, Kac's principle holds: the heat kernel can be approximated to infinite order in $T$ by the heat kernel on a half-plane. To make this precise, we quote \cite[Theorem 3]{luck}; see also \cite{nrs1}.
		Since this works on any patch of sf away from ff, and the models themselves are pc on sf, they must agree on all of sf, including down to ff.
	\end{proof}
	
	Since $H_{S_{\gamma}}$ is pc on the double heat space, its leading order models must be compatible with each other at the intersections. We will use this in the construction of the heat kernel for curvilinear polygonal domains. In particular:
	\begin{corollary}\label{cor:exactconecompatibility} We have the following compatibility conditions for the D-D heat kernel:
		\[\mathcal H_{-2,\ff,DD}|_{\ff\cap \sff}=\mathcal H_{-2,\sff,D}|_{\ff\cap \sff};\ \mathcal H_{-2,\ff,DD}|_{\ff\cap \td}=\mathcal H_{-2,\td}|_{\ff\cap \td}.\]
		An analogous result holds for the N-N heat kernel at both intersections, and for the D-N heat kernel at ff $\cap$ td. Moreover, if sf$_{1}$ is the Dirichlet component of sf and sf$_2$ is the Neumann component of sf, we have the appropriate compatibility conditions for the D-N heat kernel at ff $\cap$ sf:\[\mathcal H_{-2,\ff,DN}|_{\ff\cap \sff_1}=\mathcal H_{-2,\sff_1,D}|_{\ff\cap \sff_1};\ \mathcal H_{-2,\ff,DN}|_{\ff\cap \sff_2}=\mathcal H_{-2,\sff_2,N}|_{\ff\cap \sff_2}.\]
	\end{corollary}
	
	\begin{remark} The preceding two results may be of independent interest, as it is not obvious from the explicit expressions of the leading order models that they satisfy these compatibility conditions. 
	\end{remark} 
	
	\subsection{Construction of the heat kernel} As before, let $\Omega$ be a curvilinear polygonal domain, a subdomain of a larger surface $M$. Label its edges $E_j$ and its vertices $V_j$, with $V_j$ connecting $E_{j}$ and $E_{j+1}$ (with the appropriate generalization for multiple connected boundary components, which are allowed). For each $j$, let $\Omega_j$ be a surface with smooth boundary, also a subdomain of $M$, such that $E_{j}$ is a subset of the boundary of $\Omega_j$. Such a surface may always be created. In fact, using the tubular neighborhood theorem, $\Omega_j$ may be chosen to be contained within a small neighborhood (in $M$) of $E_j$.
	
	\begin{figure} \includegraphics[width=0.7\textwidth]{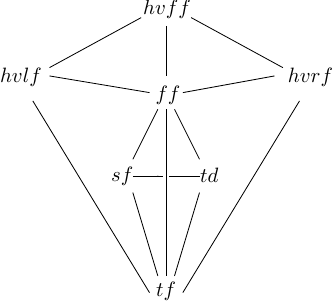} 
		\caption{The lines between boundary faces of the double heat space indicate faces whose boundaries have non-empty intersection. The $E_{j0}$ and $E_{0j}$ faces are omitted for the sake of simplicity. We note that $E_{j0}$ has non-empty intersection with hvrf, sf, and tf, whereas $E_{0j}$ has non-empty intersection with hvlf, sf, and tf. }
		\label{fig:schem} 
	\end{figure} 
	
	\subsection{Dirichlet and Neumann boundary conditions}
	We now construct the heat kernel in the setting where the boundary conditions on each side are either Dirichlet or Neumann, rather than Robin.  Consider the heat space $\Omega_h^2$. We define a kernel $H^{(1)}$ on this heat space by specifying its leading order behavior at various boundary hypersurfaces.  In Figure \ref{fig:schem}, we show the faces of the double heat space whose boundaries have non-empty intersection, and in Figure \ref{fig:zoom}, we zoom in on the double heat space near the intersection of $\overline{V}$ and $E$.  In fact, we will define $H^{(1)}$ on a blown-down version of the heat space, without the faces hvff$_{jk}$. Call this space $\widetilde\Omega_h^2$. By the proof of Proposition \ref{prop:MVheatvsourheat}, the hvff blowup may be done last, so $\widetilde\Omega_h^2$ is in fact a blow-down of $\Omega_h^2$. The reason is that the blowup at hvff is not necessary for the heat kernel for an exact cone, and is not done in Mazzeo-Vertman \cite{MaVe}.  Here we only require this blowup to obtain the composition formula in Theorem \ref{thm:comp}. 
	
	\begin{figure} \includegraphics[width=0.5\textwidth]{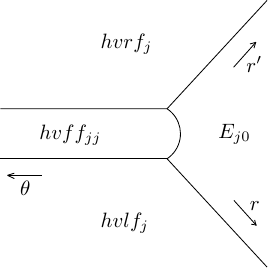} 
		\caption{This is a schematic illustration of the intersection between the side and front faces.}
		\label{fig:zoom} 
	\end{figure} 
	
	First, at td, we require $H^{(1)}$ to have the usual local asymptotic expansion \eqref{eq:ansatzattd}, namely
	\begin{equation}\label{eq:reqtattd}
		H^{(1)}\sim \sum_{j=0}^{\infty}T^{-2+j}\mathcal H_{-2+j,\td}.
	\end{equation}
	Naturally we also ask that $H^{(1)}$ decay to infinite order at tf.
	
	Observe that since $\Omega$ is a subdomain of a smooth manifold $M$, each model $\mathcal H_{-2+j,\td}$ is smooth up to the boundary of $\Omega$, so in particular smooth up to each sf$_j$ and each ff$_j$.
	
	At each side face sf$_j$, we use the heat kernel on $\Omega_j$, in boundary normal coordinates, as a model, where $\Omega_j$ is the surface with boundary defined above. By Theorems \ref{thm:neumannheat} and \ref{thm:dirichletheat}, the heat kernel $H^{\Omega_j}$ is pc on $(\Omega_j)_h^2$, and its expansion at the face sf of $(\Omega_j)_h^2$ may be written
	\[\sum_{k=0}^{\infty}T^{-2+k}\mathcal H^{\Omega_j}_{-2+k,\sff}\]
	for models $\mathcal H^{\Omega_j}_{-2+k,\sff}$ which are pc on sf. Note that the leading order $\mathcal H^{\Omega_j}_{-2,\sff}$ is either $\mathcal H_{-2,\sff,D}$ or $\mathcal H_{-2,\sff,N}$, defined in \eqref{eq:modelminus2sfD} and \eqref{eq:modelminus2sfN}, as appropriate. Now our face sf$_j$ is simply a subdomain of the face sf of $(\Omega_j)_h^2$, namely the restriction of sf to the region where both spatial variables are elements of $E_j\subseteq\partial\Omega_j$. So we simply set
	\begin{equation}\label{eq:reqtatsf}
		H^{(1)}\sim\sum_{k=0}^{\infty}T^{-2+k}\mathcal H^{\Omega_j}_{-2+k,\sff}.
	\end{equation}
	The requirements \eqref{eq:reqtatsf} and \eqref{eq:reqtattd} are compatible since the heat kernel $H^{\Omega_j}$ is pc and has the same models. They are also both compatible with infinite-order decay at tf, and with the appropriate boundary conditions on $E_{j0}$, for the same reason.
	
	Near each face ff$_j$ it is necessary to pick local polar coordinates. Define $r$ and $\theta$ to be the polar-coordinate version of boundary normal coordinates along $E_j$, so that i.e. $x=r\cos\theta$ and $y=r\sin\theta$. In these coordinates, $E_j$ is given by $\theta=0$, and $E_{j+1}$ is given by a curve with angle $\theta=\alpha_j$ at the origin. Of course, we could also have chosen polar coordinates from the boundary normal coordinates along $E_{j+1}$, so that $E_{j+1}$ is precisely $\theta=\alpha_j$ and $E_j$ is a curve with angle $\theta=0$ at the origin. These two choices of polar coordinate systems agree to second order in a neighborhood of $r=0$. With this described, coordinates valid in the interior of ff are \eqref{eq:modelcoordsintff}. Moreover, in these coordinates, we have $R=0$ at hvlf$_j$, $R'=0$ at hvrf$_j$, $\theta=0$ at $E_{j0}$, $\theta=\alpha_j$ at $E_{j+1,0}\cap ff_j$, $\theta'=0$ at $E_{0j}$, and $\theta'=\alpha_j$ at $E_{0,j+1}\cap ff_j$. The face hvff$_j$ is an extra, overblown face at $R=R'=0$.
	
	The point is that at ff$_j$, we can just use one of the models $\mathcal H_{-2,\ff,DD}$, $\mathcal H_{-2,\ff,DN}$ (or its flipped variant $\mathcal H_{-2,\ff,ND}$) or $\mathcal H_{-2,\ff,NN}$, depending on which boundary conditions we are imposing on $E_j$ and $E_{j+1}$. These models are defined in \eqref{eq:modelminus2ffall}. So we require that at each ff$_j$,
	\begin{equation}\label{eq:reqtatff}
		H^{(1)}\sim T^{-2}\mathcal H_{-2,\ff,DD}(R,\theta,R',\theta'),
	\end{equation}
	with DD replaced by DN, ND, or NN depending on the boundary conditions. Since the Laplacian is equal to the Laplacian for a straight sector to leading order at ff, these models solve the model problem at each ff$_j$.
	
	By Corollary \ref{cor:exactconecompatibility}, the requirement \eqref{eq:reqtatff} is compatible with \eqref{eq:reqtattd} and \eqref{eq:reqtatsf} at td and at sf$_j$. At sf$_{j+1}$, it also follows from Corollary \ref{cor:exactconecompatibility}, because even though we no longer have an exact cone and thus the boundary normal coordinates for $E_{j+1}$ do not agree with the polar coordinates $(r,\theta)$ everywhere, these two coordinate systems do agree to second order in $\rho_{ff_j}$. A similar argument, from the exact cone condition, shows that \eqref{eq:reqtatff} is compatible with the appropriate boundary conditions at $E_{j0}$ and $E_{j+1,0}$. 
	
	Since \eqref{eq:reqtatff} is identical to the model for the exact sectorial heat kernel, this is also compatible with $H^{(1)}$ being pc on $\widetilde\Omega_h^2$ rather than $\Omega_h^2$.
	
	The point of checking compatibility is that as a result, we know that there exists a kernel $H^{(1)}$, pc on $\widetilde\Omega_h^2$, with the expansions \eqref{eq:reqtattd}, \eqref{eq:reqtatsf}, and \eqref{eq:reqtatff} at the boundary hypersurfaces td, sf$_j$, and ff$_j$ respectively, and which decays to infinite order at tf and satisfies the appropriate boundary conditions at each $E_{j0}$. If we let $\nu_{0,j}$ be the smallest eigenvalue of the appropriate cross-sectional Laplacian in $\theta$ at each component ff$_j$ 
	(note that in the Neumann-Neumann case we have $\nu_{0,j}=0$, and otherwise $\nu_{0,j}>0$), then the leading orders of $H^{(1)}$ may be chosen as follows:
	\begin{itemize}
		\item $-2$ at td, sf$_j$, and ff$_j$, with only integer powers in the expansion;
		\item $0$ at $E_{j0}$ and $E_{0j}$ for each $j$, with only integer powers in the expansion; and
		\item $\nu_{0,j}$ at hvlf$_j$ and hvrf$_j$, with other fractional powers.
	\end{itemize}
	As in the case of manifolds with boundary, we consider $P^{(1)}:=t\mathcal L H^{(1)}$. It is pc on $\widetilde\Omega_h^2$. Since both $\Delta H^{(1)}$ and $\partial_t H^{(1)}$ satisfy the boundary conditions (by the eigenfunction expansion, the Laplacian preserves the boundary conditions when it is applied), so does $P^{(1)}$. The leading orders of $P^{(1)}$ are as follows:
	\begin{itemize}
		\item $\infty$ at td and at each sf$_j$, since we have solved the model problem to all orders;
		\item $-1$ at ff$_j$, since we have solved the model problem to one order;
		\item $0$ at $E_{0j}$, and $\nu_{0,j}$ at hvrf$_j$, since the lift of $t\mathcal L$ is tangent to these hypersurfaces;
		\item $0$ at $E_{j0}$, since $t\mathcal L$ decreases the index set by 2, but not when applied to a \emph{smooth} expansion; and
		\item $\nu_{0,j}-1$ at hvlf$_j$, since $t\mathcal L$ decreases the index set by 2 at this face but the leading term is killed, as in Mazzeo-Vertman \cite{MaVe}. The leading order term of the Laplacian is the same as for the flat Laplacian, and our model is the flat heat kernel there. Technically this requires us to choose $H^{(1)}$ to be equal to \eqref{eq:reqtatff} in a neighborhood of hvlf$_j$, which we can do. Note this is compatible with the boundary condition at $E_{j0}$ as well.
	\end{itemize}
	
	We now construct an improved parametrix which has error decaying to infinite order at both hvlf and $E_{j0}$. We do this in two steps, first eliminating the error at hvlf. This proceeds exactly as in Mazzeo-Vertman \cite{MaVe}, as hvlf is the analogue of their face rf. In the interior of hvlf, $r$ is a boundary defining function for hvlf; the other variables are $\theta\in[0,\alpha_j]$, $T$, and $z'\in\Omega$. To remove a term $r^{\gamma}a(\theta,t,z')$ in the expansion of $P^{(1)}$ at hvlf$_j$, we need to solve the indicial equation on the cone $C([0,\alpha_j])$ in $(r,\theta)$, with $t$ and $z'$ as parameters:
	\[\left(-\partial_{rr}-\frac{1}{r^2}(\partial_{\theta\theta}+\frac 14)\right)u(r,\theta,T,z')=r^{\gamma}T^{-2}a(\theta,T,z'),\]
	with the appropriate boundary conditions at $\theta=0$ and $\theta=\alpha_j$. Since $a$ is a term in the expansion of $P^{(1)}$, it, itself, satisfies those boundary conditions. Therefore, as in \cite{MaVe} and \cite{etdeo1} a solution $u$ exists with asymptotic behavior at $r=0$ given by either $r^{\gamma+2}$ or possibly $r^{\gamma+2}\log r$ in case of an unlucky indicial root coincidence. The dependence in $t$ and $z'$ is purely parametric, so $u(r,\theta,T,z')$ is pc in a neighborhood of hvlf$_j$. We multiply $u(r,\theta,T,z')$ by a cutoff function equal to 1 on a neighborhood of hvlf$_j$, choosing the cutoff function so that its gradient is parallel to each edge $E_{j0}$ and thus preserves the boundary conditions at $\theta=0$ and $\theta=\alpha_j$. Then subtracting this product from $H^{(1)}$ eliminates the term $r^{\gamma}a(\theta,T,z')$ in the expansion at hvlf and does not change the leading order of $H^{(1)}$ at any other boundary hypersurface. In particular, since hvlf does not intersect td or sf, the expansion of $H^{(1)}$ there is unchanged. Moreover, $r$ vanishes at ff, so $u$ actually decays to the same order as $a$ at ff, and thus the leading order of the expansion of $H^{(1)}$ at ff is unchanged.
	
	Iterating this process produces a parametrix $H^{(2a)}$ and an error $P^{(2a)}$ with all the same properties as $H^{(1)}$ and $P^{(1)}$, except for two differences. First, there may be logarithmic terms at ff$_j$ (as well as at hvlf$_j$) once we go one order down in the expansion. Second, $P^{(2a)}$ now vanishes to infinite order at hvlf$_j$ for each $j$.
	
	To remove the error at $E_{j0}$, we follow the same template as for manifolds with boundary, using boundary ellipticity. Namely, if $y$ is the boundary normal coordinate for the side $\Omega_j$, add a kernel which is equal to
	\[\frac 12y^2\left(\mathcal LH^{(2a)}\right)(T,x,0,x',y')+ O(y^3)\]
	and supported in a neighborhood of $E_{j0}$. This improves the order of the error at $E_{j0}$ from 0 to 1. Note also that $\mathcal LH^{(2a)}$ vanishes to infinite order at sf$_j$, so there is no effect on the expansion at sf$_j$. Iterating this process and taking an asymptotic sum, as for manifolds with boundary, we obtain the following.
	\begin{proposition}There exists a kernel $H^{(2)}$ pc on $\widetilde\Omega_h^2$, satisfying the appropriate combination of Dirichlet and Neumann boundary conditions, with $\lim_{t\to 0}H^{(2)}=\delta(z-z')$, where if we let $P^{(2)}=t\mathcal LH^{(2)}$,
		\begin{itemize}
			\item $H^{(2)}$ vanishes to infinite order at tf and has the full expansions \eqref{eq:reqtattd} and \eqref{eq:reqtatsf} at td and each sf$_j$ respectively;
			\item $H^{(2)}$ has leading term given by \eqref{eq:reqtatff} at each ff$_j$, with the next term being one full order lower (possibly logarithmic);
			\item $P^{(2)}$ vanishes to infinite order at tf, td, each sf$_j$, each hvlf$_j$, and each $E_{j0}$;
			\item $P^{(2)}$ has leading order $-1$ at each ff$_j$, $0$ at each $E_{0j}$, and $\nu_{0,j}$ at each hvrf$_j$.
		\end{itemize}
	\end{proposition}
	
	We will need to compose, so we will blow up to pass to $\Omega_h^2$ by creating hvff$_{jk}$.
	\begin{corollary} The kernels $H^{(2)}$ and $P^{(2)}$ also lift to be pc on $\Omega_h^2$, with leading orders $2\nu_{j,0}$ and $\infty$ respectively at hvff$_{jk}$.
	\end{corollary}
	Now we eliminate the last error by forming the formal Neumann series
	\[\Id+P^{(3)}:=\Id-\sum_{k=1}^{\infty}\left(-\frac{1}{T^2}P^{(2)}\right)^k.\]
	Note that $T^{-2}P^{(2)}$ vanishes to infinite order at all faces except for ff$_j$, $E_{0j}$, and hvrf$_j$, where it has leading orders $-3$, $0$, and $\nu_{0,j}$ respectively. We use Theorem \ref{thm:comp} to analyze the power $(-T^{-2}P^{(2)})^{k}$. We see immediately that it also vanishes to infinite order at all other faces and has leading order $-4+k$ at ff$_j$. At hvrf$_j$, the index sets have an inductive relationship: the index set for the $k^{th}$ power is the extended union of the index set for the $(k-1)^{st}$ power with $k$ plus the index set for the $0^{th}$ power. The union of all of these is indeed a legitimate index set.  In particular, there are only a finite number of extended unions involved at order less than $s$ for each value of $s$.  The leading order is $\nu_{0,j}$, and there is no logarithmic term at that leading order.   At $E_{0j}$, the index set is the same as that for $P^{(2)}$, with leading order zero.
	
	All of this allows us to asymptotically sum the Neumann series, and as before, the sum is convergent. 
	The sum $P^{(3)}$ has the same leading orders as $P^{(2)}$ at each boundary hypersurface. As before, we let
	\[H^{(3)}=H^{(2)}\left(\Id+P^{(3)}\right),\]
	and deduce that $H^{(3)}$ is the true heat kernel.
	
	By Theorem \ref{thm:comp}, the term $H^{(2)}P^{(3)}$ vanishes to infinite order at tf, td and sf, with leading order $-1$ at ff$_j$, so it does not affect the expansion at td and sf and does not affect the first term at ff$_j$. It has leading order greater than or equal to zero everywhere else, with no logarithmic terms. This tells us the following:
	\begin{theorem}\label{thm:5point8}
		The heat kernel for $\Omega$, with Dirichlet or Neumann boundary conditions on each side $E_j$, is pc on $\Omega_h^2$, vanishing to infinite order at tf and continuous down to all boundary hypersurfaces except for td, sf$_j$, and ff$_j$.
		
		Its full expansions at td and each sf$_j$ are \eqref{eq:reqtattd} and \eqref{eq:reqtatsf}, which are the same as those for a closed manifold, and a manifold with boundary and the appropriate boundary condition, respectively.
		
		Its expansion at ff$_j$ has leading term \eqref{eq:reqtatff} and no other terms within one order.
	\end{theorem}
	
	\begin{remark}
		It is certainly possible to push through the composition formula and compute the index sets of the heat kernel for $\Omega$ at other faces (hvff, hvrf, hvlf). However, we do not think that the results obtained in this fashion are optimal -- there are quite a lot of log terms which may not actually exist -- so we omit the statements.  In fact, it is possible that the off-diagonal faces for positive time, namely hvff$_{jk}$ for $j\neq k$, are not necessary at all, but our results are easier and likely quicker to prove this way.  
	\end{remark}
	
	The following corollary is a version of Kac's principle of not feeling the boundary for the Dirichlet boundary condition \cite{kac}; see also \cite{nrs1} for the Neumann and Robin boundary conditions.  
	\begin{corollary} The full expansions at td, sf$_j$, and ff$_j$ are local, in the sense that if two domains with corners $\Omega$ and $\Omega'$ are isometric in a region $R$, then the expansions at the corresponding faces of the heat spaces $\Omega_h^2$ and $(\Omega')_h^2$ agree to all orders when the spatial variables are restricted to lie within the interior of $R$.
	\end{corollary}
	Note that these are all the faces in the lift of $\{t=0\}$ where the heat kernel has nontrivial behavior, so the statement implies that any global contribution to the heat kernel at $t=0$ is  $O(t^{\infty})$.
	\begin{proof} The corollary follows immediately for td and sf$_j$ from the construction, since the expansions there are the same as for $H^{(2)}$. For ff$_j$, it is also true: although the powers $(T^{-2}P^{(2)})^j$ are compositions and thus not local, by the composition theorem, their expansions at ff$_j$ only depend on the expansion of $P^{(2)}$ itself at ff$_j$, which \emph{is} local. Thus the expansion of $P^{(3)}$ at ff$_j$ is local, and using the composition theorem again, so is the expansion of the true heat kernel. \end{proof}
	
	\subsection{Robin boundary conditions}
	
	The construction of the Robin heat kernel proceeds very similarly to that of the Neumann heat kernel, though the boundary condition is somewhat more complicated. For each edge $E_j$, let $\kappa_j(x)$ be a smooth function on $E_j$. The key lemma is as follows.
	
	\begin{lemma}\label{lem:H1Robin} There exists a kernel $H^{(1)}_{Robin}$, pc on $\widetilde \Omega_h^2$ and with $\lim_{t\to 0}H^{(2)}=\delta(z-z')$, such that, letting $P^{(1)}_{Robin}=H^{(1)}_{Robin}$,
		\begin{itemize}
			\item $H^{(1)}_{Robin}$ satisfies Robin boundary conditions with parameter $\kappa_j(x)$ on each edge $E_j$;
			\item $H^{(1)}_{Robin}$ vanishes to infinite order at tf and has the full expansion \eqref{eq:reqtattd} at td;
			\item At each sf$_j$, $H^{(1)}_{Robin}$ has the same full expansion as the Robin heat kernel on $\Omega_j$, with a parameter agreeing with $\kappa_j(x)$ upon restriction to $E_j$;
			\item At each ff$_j$, $H^{(1)}_{Robin}$ has the same leading term \eqref{eq:reqtatff} as in the Dirichlet, Neumann, or mixed cases, using the Neumann model at every Robin edge; and
			\item $P^{(1)}_{Robin}$ has leading orders $\infty$ at td and at sf$_j$, $-1$ at ff$_j$, $0$ at $E_{0j}$ and $E_{j0}$, $\nu_{j,0}$ at hvrf$_j$, and $\nu_{j,0}-1$ at hvlf$_j$.
		\end{itemize}
	\end{lemma}
	\begin{proof} The issue is compatibility of all these requirements, noting that Robin boundary conditions are more complicated than Dirichlet or Neumann boundary conditions at the intersections of $E_{j0}$ with sf$_j$ and ff$_j$. However, it turns out that Robin boundary conditions only affect the \emph{sub-leading} terms of the expansion of $H^{(1)}$ at ff. This is why the Robin heat kernel may be viewed as a correction of the Neumann heat kernel.
		
		We require the full expansion \eqref{eq:reqtattd} at td, and observe that this is compatible with the expansion at sf$_j$, whose form is guaranteed by Theorem \ref{thm:bigrobinmwb}, and the leading term \eqref{eq:reqtatff} at ff$_j$. Indeed 
		the compatibility betwen td and sf$_j$ follows from the fact that the Robin heat kernel on $\Omega_j$ is pc. The compatibility between td and ff$_j$ follows from the fact that the Neumann heat kernel on $\Omega$ is pc.  The compatibility between sf$_j$ and the leading term at ff$_j$ follows from the fact that the leading term at sf$_j$ is the same as for the Neumann problem on $\Omega$, so we can use Corollary \ref{cor:exactconecompatibility} as in the previous section. It remains only to show that we can find such a kernel which also satisfies Robin boundary conditions.
		
		To do this, note that Robin boundary conditions imply that if $u$ is the leading order term (zeroth order) of the expansion of $H^{(1)}_{Robin}$ at ff, then the next term must be
		\[\kappa(x)y\cdot u.\]
		Since $y$ vanishes at ff$_j$ and sf$_j$ as well as $E_{0j}$, this term vanishes to an order at ff$_j$ and sf$_j$ which is one higher than the order of $u$ there. Hence any compatibility requirements only affect the lower order terms.
		
		In order to dissect the compatibility requirements imposed by Robin boundary conditions, we zoom in near a triple intersection $E_{j0}\, \cap\, $\sff$_j\, \cap\, $\ff$_j$. Let boundary defining functions $\rho_E$, $\rho_{\sff}$, and $\rho_{\ff}$ be chosen so that the product of all three is $y$; we use these three coordinates and suppress the (parametric) dependence in all other coordinates. We write out the (previously specified) expansion at sf$_j$ as well as the (unknown save for the first term) expansion at ff$_j$, doing both for $tH^{(1)}$ rather than $H^{(1)}$ to keep notation simple:
		\begin{align}\label{eq:compatexpansions}
			& tH^{(1)}_{Robin}\cong\sum_{i=0}^{\infty}\rho_{\sff}^i g_i(\rho_E,\rho_{\ff})\textrm{ at \sff}_j; \\  
			& tH^{(1)}_{Robin}\cong\sum_{j=0}^{\infty}\rho_{\ff}^j h_j(\rho_E,\rho_{\sff})\textrm{ at \ff}_j. \nonumber 
		\end{align}
		We also write the expansion of each $g_i$ at $\rho_{\ff}=0$:
		\begin{equation}\label{eq:doubleexpansion} g_i(\rho_E,\rho_{\ff})\cong \sum_{k=0}^{\infty}\rho_{\ff}^ka_{ik}(\rho_E)+O(\rho_{\ff}^{\infty}).\end{equation}
		In order for the expansions \eqref{eq:compatexpansions} to be compatible with each other, for each $j$, we need
		\begin{equation}\label{eq:sfffcomp}
			h_{j}(\rho_{E},\rho_{\sff})\cong\sum_{i=0}^{\infty}\rho_{\sff}^{i}a_{ij}(\rho_E) + O(\rho_{\sff}^{\infty}).
		\end{equation}
		
		On the other hand, in these coordinates, our Robin boundary condition becomes
		\begin{align*} \left(\frac{1}{\rho_{\sff}\rho_{\ff}}\frac{\partial}{\partial\rho_E}-\kappa(\rho_{\sff},\rho_{\ff})\right)H^{(1)}_{Robin}=0, \\ 
			\mbox{ i.e. }\left(\frac{\partial}{\partial\rho_E}-\rho_{\sff}\rho_{\ff}\kappa(\rho_{\sff},\rho_{\ff})\right)tH^{(1)}_{Robin}=0. 
		\end{align*} 
		Plugging in \eqref{eq:compatexpansions}, organizing, and equating the coefficients of the $\rho_{\sff}^i$ terms tells us that the compatibility condition at sf$_j\cap E_{j0}$ is, for each $i\geq 1$:
		\begin{align}\label{eq:recursivecondsf}
			& (g_i)_{\rho_E}(0,\rho_{\ff})= \\ 
			& \rho_{\ff}\cdot\left(\textrm{ the coefficient of }\rho_{\sff}^{i}\textrm{ in }\sum_{\ell=0}^i\kappa(\rho_{\sff},\rho_{\ff})g_{\ell-1}(0,\rho_{\ff})\rho_{\sff}^{\ell}\right), \nonumber 
		\end{align}
		and that this derivative is zero when $i=0$. 
		Similarly, the compatibility condition at ff$_j\cap E_{j0}$ is, for each $j\geq 1$,
		\begin{align}\label{eq:recursivecondff}
			&(h_j)_{\rho_E}(0,\rho_{\sff})= \\ 
			& \rho_{\sff}\cdot\left(\textrm{ the coefficient of }\rho_{\ff}^{j}\textrm{ in }\sum_{m=0}^j\kappa(\rho_{\sff},\rho_{\ff})h_{m-1}(0,\rho_{\sff})\rho_{\ff}^{m}\right), \nonumber
		\end{align}
		and that the derivative is zero when $j=0$.

		Recall that the full expansion of $H^{(1)}_{Robin}$ is specified at sf$_j$; since that heat kernel satisfies Robin conditions, we assume the compatibility condition \eqref{eq:recursivecondsf}. We have also specified the first term $h_0(\rho_{E},\rho_{\sff})$ at ff$_j$. Since it satisfies a Neumann boundary condition, its $\rho_E$ derivative at $E_{0j}$ is indeed zero, as required. We need to show that lower-order terms $h_j$, $j\geq 1$, may be chosen to simultaneously guarantee \eqref{eq:sfffcomp} and \eqref{eq:recursivecondff}. Working one $j$ at a time, \eqref{eq:sfffcomp} prescribes the full expansion of $h_j(\rho_E,\rho_{\sff})$ at $\rho_{\sff}=0$, and \eqref{eq:recursivecondff} prescribes the order 1 term of $h_j$ at $\rho_E=0$ in terms of the order-0 term of $h_{j-1}$. As long as these two requirements are consistent we are fine. 
		
		To check this, we just plug \eqref{eq:sfffcomp} into \eqref{eq:recursivecondff}. After rearrangement and equating like terms, we see that we need for each $i$ and $j>1$,
		\begin{align}\label{eq:needthisequation}
			& a_{ij}'(0)= \\ 
			& \textrm{ the coefficient of }\rho_{\sff}^{i}\rho_{\ff}^j\textrm{ in }\sum_{\ell=0}^{i-1}\sum_{m=0}^{j-1}\kappa(\rho_{\sff},\rho_{\ff})a_{\ell,m}(0)\rho_{\sff}^{\ell+1}\rho_{\ff}^{m+1}. \nonumber
		\end{align}
		This, in turn, is guaranteed by plugging \eqref{eq:doubleexpansion} into \eqref{eq:recursivecondsf}, completing the proof of Lemma \ref{lem:H1Robin}.
	\end{proof}
	
	The construction of the Robin heat kernel is now analogous to the Dirichlet and Neumann cases. We solve away the error at hvlf$_j$ and then at $E_{j0}$. When solving away the error at hvlf$_j$, we need to remove a term $r^{\gamma}a(\theta, t,z')$. Since $\partial_{\theta}=r\partial_y$, the coefficient $a(\theta,t,z')$ actually solves Neumann conditions, rather than Robin conditions, at $\theta=0$ and $\theta=\alpha_j$. So as in the Neumann construction, the indicial equation may be solved and the solution, which has leading order $\gamma+2$ at hvlf$_j$, may be added to our parametrix in a neighborhood of hvlf$_j$. Of course this does not preserve the Robin condition at $E_{j0}$. However, the error has leading order $\gamma+2$ at hvlf$_j$, and $y$ has order 1 there. If we just add back $\kappa y$ times this Robin error in a neighborhood of hvlf$_j$, the result satisfies the Robin boundary condition. Moreover, after applying $t\mathcal L$, the result has error at worst $(\gamma+2)+1-2=\gamma+1$ there. So this construction may be iterated to remove the error at hvlf$_j$. 
	
	The error at $E_{j0}$ may be eliminated in the same way as before, since adding terms at order 2 at $E_{j0}$ does not affect the Robin boundary condition there. The construction of the formal Neumann series proceeds precisely as before, and yields:
	
	\begin{theorem}The heat kernel for $\Omega$, with Dirichlet, Neumann, or Robin boundary conditions on each side $E_j$, is pc on $\Omega_h^2$, vanishing to infinite order at tf and continuous down to all boundary hypersurfaces except for td, sf$_j$, and ff$_j$.
		
		Its full expansion at td is \eqref{eq:reqtattd}, which is the same as that for a closed manifold.
		
		Its full expansion at each sf$_j$ is \eqref{eq:reqtatsf}, which is identical to that for the manifold with boundary $\Omega_j$ and the appropriate (Dirichlet/Neumann/Robin) boundary condition. Note that by Theorem \ref{thm:bigrobinmwb}, at any Robin component of sf$_j$, the leading term is equal to the leading term for the heat kernel on $\Omega_j$ with Neumann boundary conditions, the second term is $\mathcal H_{-1,\sff,R}$, and all other terms are at order $T=t^{1/2}$.
		
		Its expansion at each ff$_j$ has leading term \eqref{eq:reqtatff}, with Neumann conditions at any Neumann OR Robin component, and Dirichlet conditions at any other Dirichlet component. There are no other terms within one order in $T=t^{1/2}$.
	\end{theorem}
	
	\section{Heat trace on a curvilinear polygonal domain}
	Let $\Omega$ be a curvilinear polygonal domain as defined previously, with a Dirichlet, Neumann, or Robin condition along each side. Assume any Robin parameters $\kappa(x)$ are smooth along each side. In the previous section we have constructed the heat kernel for $\Omega$ and shown that it is pc on $\Omega_h^2$. We now pass to the heat trace.
	
	The first thing to do is to restrict to the diagonal. The lifted diagonal in $\Omega_h^2$ is a p-submanifold and is diffeomorphic to $\Omega_h$ via the lift of the map $(t,z,z)\to (t,z)$. The identification of faces is hvff $\to$ sv, ff $\to$ pv, sf $\to$ pe, td $\to$ tf.  Therefore, by restriction:
	\begin{proposition} The diagonal heat kernel $H_{\Omega}(t,z,z)$ is pc on $\Omega_{h}$, with leading order $-2$ at tf, each $pv_j$, and each $pe_j$, as well as non-negative leading orders at all other boundary hypersurfaces.
	\end{proposition}
	\begin{remark} Naturally, all locality statements about the kernel still hold when it is restricted to the diagonal. For example, the expansion at tf is the same as that for a closed manifold. The expansion at $pe_j$ is the same as that for a manifold with boundary. If there are any Robin edges, the expansion at the corresponding $pe_j$ is the same as the Neumann expansion, plus the restriction to the diagonal of $\mathcal H_{-1,\sff,R}$, plus terms of order zero.
	\end{remark}
	
	Let $\pi_{1}$ be the lift of the projection map from $\Omega_0\times [0,1)_T$ to $[0,1)_T$ to a map from $\Omega_{h}$ to $[0,1)_T$. This map is the composition of a projection map and a blow-down map and therefore is a b-map which is a b-submersion. Since the image space has no corners it is automatically b-normal, and therefore $\pi_{1}$ is a b-fibration. Thus, from the pushforward theorem:
	\begin{theorem} The heat trace $\Tr H_{\Omega}(t)$ has a pc expansion in $T=t^{1/2}$.
	\end{theorem}
	We can say substantially more, and in fact can explicitly identify all terms in this expansion up to and including the $t^0$ term, by carefully analyzing push-forward by this integration map. The integration is with respect to the usual measure $dz$ on $\Omega$. Multiplying both sides by the canonical density $dT$, we get
	\[\int_{\Omega}H_{\Omega}(T^2,z,z)\, dz\, dT=\Tr H_{\Omega}(T^2)\, dT.\]
	The density $dz\, dT$ is $\nu(\Omega\times[0,1)_T)$, but it is not $\nu(\Omega_{h})$. Using an analogous process to the proof of Proposition \ref{prop:densityblowup},
	\[(\beta)^*(dz\, dT)=\rho_{\ff}^2\rho_{\sff}\nu(\Omega_{h}).\]
	So, writing integration as a push-forward by $\pi_{1}$, we obtain
	\[(\pi_{1})_*(H^{\Omega}(T^2,z,z)\rho_{pv}^2\rho_{pe}\cdot \nu(\Omega_{h}))=\Tr H_{\Omega}(T^2)\cdot \nu([0,1)_T).\]
	From this we see that we really need to understand $H^{\Omega}(T^2,z,z)\rho_{pv}^2\rho_{pe}$.
	\begin{remark}This transformation to canonical densities explains why the leading terms at $pe$ and $pv$, though they both have order $-2$, only contribute at orders $-1$ and $0$ respectively to the heat trace.\end{remark}
	
	Consider the function $H^{\Omega}(T^2,z,z)\rho_{pv}^2\rho_{pe}$. Its expansions are as follows:
	\begin{itemize}
		\item At tf, there is an expansion in integer powers of $T$ beginning with $T^{-2}$.
		\item At $pe$, there is an expansion in integer powers of $T$ beginning with $T^{-1}$.
		\item At $pv$, there is an expansion with leading term at $T^0$ which may have logarithmic terms beginning at $T\log T$.
	\end{itemize}
	We may say more about these expansions. Each of them is inherited from the expansion at the corresponding face in the double space. From that analysis, each term of the expansion of $H^{\Omega}(T^2,z,z)$ at the face tf is $T^j$ times a smooth function of $z$. Since $T=\rho_{\tf}\rho_{pe}\rho_{pv}$ for suitable boundary defining functions, the coefficient of the term of order $\rho_{\tf}^j$ at tf has leading order $j$ at $pe$ and at $pv$. When multiplying by $\rho_{pv}^2\rho_{pe}$, though, this coefficient has leading order $j+1$ at $pe$ and $j+2$ at $pv$. Similarly, the order $\rho_{pe}^j$ term at $pe$ has leading order $j+1$ at $pv$.
	
	What this means is that no extended unions appear in the pushforward theorem. Recall that an extended union only occurs when the coefficient of a term of order $j$ at one boundary hypersurface in the preimage of $\{t=0\}$ itself has leading order at most $j$ at an adjacent such boundary hypersurface, which may produce a term $t^j\log t$ (or $t^j(\log t)^2$ if all three boundary hypersurfaces are involved). The preceding discussion shows that this does not happen. So any logarithmic terms in the heat trace expansion must come from logarithmic terms at the face ff (i.e. pv), and thus arise at order $T^{1/2}$ at the earliest. Therefore
	\[\Tr H_{\Omega}(t)=a_{-1}t^{-1}+a_{-1/2}t^{-1/2}+a_0t^0+O(t^{1/2}\log t).\]
	Moreover, the coefficients $a_{-1}$, $a_{-1/2}$, and $a_0$ are the sum of the contributions from each of the three faces tf, $pe$, and $pv$.
	
	These contributions are easy to evaluate. At tf, the expansion is just the usual heat trace expansion from the interior of a manifold (as the coefficients are all the same), giving a contribution of 
	\[\frac{A(\Omega)}{4\pi t}+\frac{1}{12\pi}\int_{\Omega} K(z)\, dz+O(t).\] 
	At $pe_j$, for the same reason, the expansion is the heat trace expansion for a manifold with boundary, giving a contribution for each edge $E_j$. The McKean-Singer asymptotics \cite{mc-s} tell us what this term must be in the Dirichlet and Neumann settings. In the Robin setting, there is an extra contribution at $t^0$ coming from the integral of $\mathcal H_{-1,\sff,Robin}$, and it is easy to see that it will be an integral of $\kappa(x)$ over the boundary times a constant. From \cite[Theorem 5.2]{gilkey}, we know what the constant must be.\footnote{An expression for this term also appears in \cite{zayed}. However, it differs by an overall sign from the expression in \cite{gilkey}. A direct computation due to F\'elix Houde \cite{houdepc} indicates that the reference \cite{gilkey} has the correct sign.} All in all, the contribution from $pe_j$ is, where $k_g(x)$ is the geodesic curvature on the boundary,
	\[-\frac{\ell(E_j)}{8\sqrt{\pi}} t^{-1/2}+\frac{1}{12\pi}\int_{E_j}k_g(x)\, dx+O(t^{1/2})\textrm{ in the Dirichlet setting;}\]
	\[\frac{\ell(E_j)}{8\sqrt{\pi}} t^{-1/2}+\frac{1}{12\pi}\int_{E_j}k_g(x)\, dx+O(t^{1/2})\textrm{ in the Neumann setting;}\]
	\[\textrm{ and } \frac{\ell(E_j)}{8\sqrt{\pi}} t^{-1/2}+\frac{1}{12\pi}\int_{E_j}k_g(x)\, dx-\frac{1}{2\pi}\int_{E_j}\kappa(x)\, dx+O(t^{1/2})\] 
	in the Robin setting.  As discussed previously, at$pv$, the leading order contribution to the heat trace is at $T^0$. This reflects the fact that $r dr$ lifts to $T^2 R dR$, thereby canceling the factor of $T^{-2}$. The leading order term in the expansion of the diagonal heat kernel at $pv$ is the same as it is for the heat kernel on an exact sector of the same angle, and therefore may be calculated by studying the model heat kernel on that sector. 

	\subsection{Vertex contributions} 
	We recall our explicit calculations of the Green's kernels for infinite circular sectors to compute the ``vertex contribution'' to the short time asymptotic expansion of the heat trace.  For this purpose it is convenient to define:  
	$$A := \int_{0}^{\infty}K_{i\mu}(r \sqrt s)K_{i\mu}(r_0 \sqrt s) \cosh(\pi-|\phi_0-\phi|)\mu d\mu,$$
	$$B := \int_{0}^{\infty}K_{i\mu}(r \sqrt s)K_{i\mu}(r_0 \sqrt s) \frac{\sinh\pi\mu}{\sinh\gamma\mu}\cosh(\phi+\phi_0-\gamma)\mu d\mu$$
	$$C :=  \int_{0}^{\infty}K_{i\mu}(r \sqrt s)K_{i\mu}(r_0 \sqrt s) \frac{\sinh(\pi-\gamma)\mu}{\sinh\gamma\mu}\cosh(\phi-\phi_0)\mu  d\mu,$$
	$$F := \int_{0}^{\infty}K_{i\mu}(r \sqrt s)K_{i\mu}(r_0 \sqrt s)\frac{\sinh(\pi\mu)}{\cosh\gamma\mu}\sinh((\phi+\phi_0-\gamma)\mu) d\mu$$
	and
	$$E := - \int_{0}^{\infty}K_{i\mu}(r \sqrt s)K_{i\mu}(r_0 \sqrt s) \frac{\cosh(\pi-\gamma)\mu}{\cosh\gamma\mu}\cosh((\phi-\phi_0)\mu)d\mu.$$ 
	The Dirichlet and Neumann Green's functions are, respectively, 
	$$G_D = \frac{1}{\pi^2} \left( A - B + C \right), \quad G_N = \frac{1}{\pi^2} \left( A + B + C \right).$$
	For the Dirichlet condition at $\phi=0$ and Neumann condition at $\phi=\gamma$, the Green's function is 
	\begin{equation*}
		\frac{1}{\pi^2}(A+F+E).
	\end{equation*}
	In \cite[\S 3]{nrs1} we have computed the contributions of the terms $A$, $B$, and $C$ to the heat trace; see also \cite{vdbs} for an earlier computation along similar lines.  In particular, we computed the integral of each of these expressions, along the diagonal $r=r_0$ and $\phi=\phi_0$ over the region $[0, R]_r \times [0, \gamma]_\phi$ with respect to polar coordinates $(r, \phi)$.  The vertex contribution comes solely from the $C$ term in the D-D and N-N cases.  There we see that the C term contributes to the heat trace \cite[\S 3.1.3]{nrs1} 
	\begin{equation} \label{Ctrace} \frac{\pi^2 - \gamma^2}{24\pi \gamma}.  \end{equation} 
	In the D-N case, the vertex contribution arises from the terms F and E.  
	
	\subsubsection{Contribution from the $F$ term} 
	Let us make some manipulations
	\begin{multline*}
		\frac{\sinh(\pi\mu)}{\cosh\gamma\mu}\sinh((\phi+\phi_0-\gamma)\mu) \\ 
		= \frac{\sinh(\pi\mu)}{\cosh\gamma\mu}\sinh((\phi+\phi_0-\gamma)\mu)
		-\frac{\sinh(\pi\mu)}{\sinh\gamma\mu}\cosh((\phi+\phi_0-\gamma)\mu)\\
		+\frac{\sinh(\pi\mu)}{\sinh\gamma\mu}\cosh((\phi+\phi_0-\gamma)\mu)
	\end{multline*}
	\begin{multline*}
		=\frac{\sinh((\phi+\phi_0-\gamma)\mu)\sinh\gamma\mu-\cosh((\phi+\phi_0-\gamma)\mu)\cosh\gamma\mu}{\sinh\gamma\mu\cosh\gamma\mu}\sinh(\pi\mu)\\
		+\frac{\sinh(\pi\mu)}{\sinh\gamma\mu}\cosh((\phi+\phi_0-\gamma)\mu).
	\end{multline*}
	This expression simplifies to:
	
	\begin{equation*}
		\begin{gathered}
			-\frac{2\sinh(\pi\mu)}{\sinh(2\gamma\mu)}\cosh((\phi+\phi_0-2\gamma)\mu)+\frac{\sinh(\pi\mu)}{\sinh\gamma\mu}\cosh((\phi+\phi_0-\gamma)\mu)\\
			=:-2B_1+B_2.
		\end{gathered}
	\end{equation*}
	By the calculation of the trace of the $B$ term in \cite[\S 3]{nrs1}, the contribution of $B_2$ is $\frac{R}{4\sqrt{\pi t}}+O(\sqrt{t})$. Next we note that, for $\phi=\phi_0$,
	\begin{equation*}
		\int_{0}^{\gamma}B_1d\phi=\frac{\sinh\pi\mu}{2\mu}=\int_{0}^{\gamma}B_2d\phi.
	\end{equation*}
	Hence the contributions of $B_1$ and $B_2$ are the same, so that $F$ contributes
	\begin{equation}\label{B}
		-2\frac{R}{4\sqrt{\pi t}}+\frac{R}{4\sqrt{\pi t}}+O(\sqrt{t})=-\frac{R}{4\sqrt{\pi t}}+O(\sqrt{t}).
	\end{equation}
	Consequently, this gives no contribution because the coefficient of $t^0$ vanishes.  
	
	\subsubsection{Contribution from the $E$ term} 
	Finally, we study the term $E$.  We need to compute 
	$$-\frac{1}{\pi^2}\int_{0}^{\infty}K_{i\mu}(r\sqrt{s})K_{i\mu}(r_0\sqrt{s})\frac{\cosh(\pi-\gamma)\mu}{\cosh\gamma\mu}\cosh((\phi-\phi_0)\mu) d\mu$$
	This is similar to the computation of the $C$ term, which we would like to recycle.  Hence, we add and subtract: 
	\begin{multline*}
		-\frac{\cosh(\pi-\gamma)\mu}{\cosh\gamma\mu}\cosh((\phi-\phi_0)\mu)
		+\frac{\sinh(\pi-\gamma)\mu}{\sinh\gamma\mu}\cosh((\phi-\phi_0)\mu)\\
		-\frac{\sinh(\pi-\gamma)\mu}{\sinh\gamma\mu}\cosh((\phi-\phi_0)\mu)
	\end{multline*}
	\begin{multline*}
		=\frac{-\cosh(\pi-\gamma)\mu\sinh\gamma\mu+\sinh(\pi-\gamma)\mu\cosh\gamma\mu}{\sinh\gamma\mu\cosh\gamma\mu}\cosh((\phi-\phi_0)\mu)\\
		-\frac{\sinh(\pi-\gamma)\mu}{\sinh\gamma\mu}\cosh((\phi-\phi_0)\mu).
	\end{multline*}
	This reduces to:  
	\begin{equation*}
		\begin{gathered}
			\frac{2\sinh(\pi-2\gamma)\mu}{\sinh(2\gamma\mu)}\cosh((\phi-\phi_0)\mu)-\frac{\sinh(\pi-\gamma)\mu}{\sinh\gamma\mu}\cosh((\phi-\phi_0)\mu)\\
			=:2C_1-C_2.
		\end{gathered}
	\end{equation*}
	We recognize the term 
	$$\int_0 ^\infty K_{i\mu}(r\sqrt{s})K_{i\mu}(r_0\sqrt{s}) C_2 d\mu = C.$$
	Consequently, we already know the contribution to the trace from $C_2$, because it is the same as that which we computed for $C$ 
	\begin{equation*}
		\frac{\gamma}{2\pi}\cdot\frac{\pi^2-\gamma^2}{12\gamma^2} + O(t^\infty), \quad t \downarrow 0. 
	\end{equation*}
	The reason we write it in this way is to recall the differences between the contributions of $C_1$ and $C_2$. The factor of $\gamma$ in $\frac{\gamma}{2\pi}$ comes from the trace calculation in which we integrate the angular coordinate over $(0,\gamma)$.  This factor is therefore the same in $C_1$.  Hence when we consider $C_1$, we just need to change $\gamma$ to $2\gamma$ in the second factor only. The contribution of $C_1$ is 
	\begin{equation*}
		\frac{\gamma}{2\pi}\cdot\frac{\pi^2-(2\gamma)^2}{12(2\gamma)^2},
	\end{equation*}
	and hence the trace contribution of $E$ is
	\begin{equation}\label{D}
		\frac{\pi^2-4\gamma^2}{48\pi\gamma}-\frac{\pi^2-\gamma^2}{24\pi\gamma}=-\frac{\pi^2+2\gamma^2}{48\pi\gamma} + O(t^\infty).
	\end{equation}
	We have now computed the contribution of the vertex to the $t^0$ term. Any Robin-Dirichlet, or Robin-Neumann, or Robin-Robin corner is treated as if the Robin conditions were Neumann conditions, as the corresponding models at ff$_{diag,j}$ are the same.
	
	The vertex contribution for an interior angle of $\gamma$ is therefore:  
	\begin{equation} \label{corner-most} \frac{\pi^2 - \gamma^2}{24 \pi \gamma} \textrm{ for D-D, N-N, R-R, and N-R boundary conditions} \end{equation} 
	or 
	\begin{equation} \label{corner-dndr} -\frac{\pi^2+2\gamma^2}{48\pi\gamma} \textrm{ for D-N and D-R mixed boundary conditions.}\end{equation} 
	The vertex contribution \eqref{corner-dndr} appears to be new and may be of independent interest.  In \S \ref{s:altcalcs} we show how, given the D-D corner contribution, one may also use the more familiar series expression for the heat kernel as in \cite{cheeger} to compute the N-N and D-N corner contribution.  The result is of course the same as we have computed here.  In summary, we have Theorem \ref{thm:mainthm}.

	\begin{remark} The Gauss-Bonnet theorem dictates that
		\[2\pi\chi(\Omega)=\int_{\Omega}K(z)\, dz+\int_{\partial\Omega}k_g(x)\, dx+\sum_{j=1}^n(\pi-\alpha_j).\]
		This yields an alternate expression for $a_0$:
		\begin{align*}
			&a_0= \frac 16\chi(\Omega)-\frac{1}{12\pi}\sum_{j=1}^{n}(\pi-\alpha_j) - \frac{1}{2\pi}\sum_{j\in\mathcal E_R}\int_{E_j}\kappa_j(x)\, dx \\ 
			& +\sum_{j\in V_{=}}\frac{\pi^2-\alpha_j^2}{24\pi\alpha_j}+\sum_{j\in V_{\neq}}\frac{-\pi^2-2\alpha_j^2}{48\pi\alpha_j}.
		\end{align*}
	\end{remark}
	
	\begin{remark} It is straightforward to allow for surfaces which may also have isolated conical singularities.  An isolated conical singularity with opening angle $2\alpha$ will give contribute to the heat trace:    
		$$\frac{\pi^2 - \alpha^2}{12 \pi \alpha}.$$
	\end{remark}
	
	\subsection{Vertices as spectral invariants}  \label{s:isosp}  
	Here we apply our results, presenting several contexts in which the presence, or lack, of vertices is spectrally determined.  We also show that a jump in boundary condition is spectrally determined.  
	
	\begin{theorem} Let $\Sigma$ be a surface with at least one vertex with interior angle not equal to $\pi$ and either the Dirichlet boundary condition or the Neumann boundary condition.  Let $\Omega$ be a smoothly bounded surface with either the Dirichlet boundary condition or the Neumann boundary condition such that $\chi(\Omega) \leq \chi(\Sigma)$.  Then $\Sigma$ and $\Omega$ are not isospectral.   
	\end{theorem} 
	
	\begin{proof} 
		It suffices to compare the short time asymptotic expansion of the heat traces and demonstrate that the coefficients cannot be the same for $\Sigma$ and $\Omega$.  The coefficient $a_0$ for $\Sigma$ is: 
		$$a_0 (\Sigma) = \frac{\chi(\Sigma)}{6} - \frac{1}{12\pi} \sum_{j=1} ^n (\pi - \alpha_j) + \sum_{j=1} ^n \frac{\pi^2 - \alpha_j^2}{24 \pi \alpha_j},$$
		where $\Sigma$ has $n$ vertices with interior angles $\alpha_j$.  This expression simplifies to:
		$$a_0 (\Sigma) = \frac{\chi(\Sigma)}{6} - \frac{n}{12} + \sum_{j=1} ^n \frac{\pi^2 + \alpha_j ^2}{24 \pi \alpha_j}.$$
		On the other hand, 
		$$a_0 (\Omega) = \frac{\chi(\Omega)}{6} \leq \frac{\chi(\Sigma)}{6}.$$
		Since at least one $\alpha_j \neq \pi$, it is a straightforward exercise in multivariable analysis \cite{corners} to demonstrate the strict inequality 
		$$a_0 (\Sigma) > \frac{\chi(\Sigma)}{6} \geq \frac{\chi(\Omega)}{6} = a_0(\Omega).$$
	\end{proof} 
	
	We obtain a similar result for the Robin boundary condition.  Recall the Robin boundary condition is, 
	$$u = \kappa \frac{\pa u}{\pa \nu}, \quad \textrm{ on all smooth boundary components,} \quad  \kappa \geq 0.$$
	Above, $\frac{\pa u}{\pa \nu}$ is the \em inward \em pointing unit normal, as in \eqref{eq:robinbc}.  
	
	\begin{theorem} Let $\Sigma$ be a surface with at least one vertex with interior angle not equal to $\pi$ with the Robin boundary condition as above, with constant Robin parameter.  Let $\Omega$ be a smoothly bounded surface with $\chi(\Omega) \leq \chi(\Sigma)$.  Assume the same Robin boundary condition on $\pa \Omega$.  Then $\Sigma$ and $\Omega$ are not isospectral.  
	\end{theorem} 
	\begin{proof} 
		
		We argue by contradiction.  Assume that $\Sigma$ and $\Omega$ are isospectral.  Then, they must have the same heat trace coefficients.  The terms $a_{-1/2}(\Sigma)$ and $a_{-1/2} (\Omega)$ show that the boundaries of $\Omega$ and $\Sigma$ have the same length.  Hence, since at least one of the angles $\alpha_j$ is not equal to $\pi$, we have 
		\[a_0 (\Sigma) = \frac{\chi(\Sigma)}{6} - \frac{n}{12} + \sum_{j=1} ^n \frac{\pi^2 + \alpha_j^2}{24 \pi \alpha_j}- \frac{\kappa |\pa \Sigma|}{2\pi} >  \frac{\chi(\Sigma)}{6}  -  \frac{\kappa |\pa \Sigma|}{2\pi}.\]
		Above, $|\pa \Sigma|$ is the length of the boundary of $\Sigma$, $n$ is the number of vertices, and $\alpha_j$ is the interior angle at the $j^{th}$ vertex. 
		On the other hand 
		$$a_0(\Omega) = \frac{\chi(\Omega)}{6} - \frac{\kappa |\pa \Omega|}{2\pi} =  \frac{\chi(\Omega)}{6} - \frac{\kappa |\pa \Sigma|}{2\pi}
		< a_0(\Sigma).$$
		This is the desired contradiction.  
	\end{proof} 
	
	For the case of smoothly bounded surfaces, the spectrum also detects a jump in the boundary condition, even without vertices.  This is depicted in Figure \ref{fig:dbcnbc}.
	
	\begin{figure} \includegraphics[height=2in]{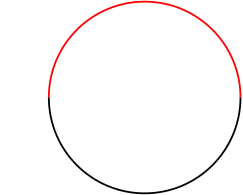} \caption{For a circular domain, impose the Dirichlet boundary on the red arc and the Neumann boundary on the black arc, taking the Friedrichs extension at the intervace. Such a domain is not isospectral to any simply connected smoothly bounded domain which has either the Dirichlet or Neumann condition (but not mixed).  In fact, one may take the red and black pieces of the boundary to be of \em any \em proportions, not necessarily equal.} \label{fig:dbcnbc}   \end{figure}

	\begin{theorem} Let $\Sigma$ be a smoothly bounded surface which has Dirichlet boundary condition and Neumann boundary condition on a single boundary component (that is, a nontrivial Zaremba boundary condition), with a Friedrichs extension at the interface.  Let $\Omega$ be a smoothly bounded surface which has either Neumann or Dirichlet boundary condition (not mixed). Assume that 
		$$\chi(\Omega) \geq \chi(\Sigma).$$
		Then $\Sigma$ and $\Omega$ are not isospectral.
	\end{theorem} 
	
	\begin{proof} For $\Sigma$, the heat trace coefficient 
		$$a_0 (\Sigma) = \frac{\chi(\Sigma)}{6} - \frac{n}{16},$$
		where $n$ is the number of times the boundary condition jumps between Dirichlet and Neumann.  We obtain this because the boundary is smooth, and hence the angle at the ``vertex'' where the boundary condition jumps is equal to $\pi$.  On the other hand, 
		$$a_0 (\Omega) \geq \frac{\chi(\Omega)}{6} \geq \frac{\chi(\Sigma)}{6} > a_0 (\Sigma),$$
		since $n\geq 1$. 
	\end{proof} 
	
	In conclusion, we determine contexts in which entirely mixed Dirichlet-Neumann vertices are spectrally determined.  In particular, this shows that we may distinguish between the presence of mixed-boundary condition vertices versus vertices with the same boundary condition on both sides; see Figure \ref{fig:squaremix}.  
	
	\begin{theorem} Assume that $\Sigma$ is a surface with vertices with mixed Dirichlet and Neumann boundary condition such that each vertex has Dirichlet on one side and Neumann on the other side.  Moreover, assume that all interior angles are less than $\frac{\pi}{\sqrt 2}$.  Let $\Omega$ be any surface which is either: 
		\begin{enumerate} 
			\item smoothly bounded and with either the Dirichlet or Neumann, but not mixed, boundary condition;
			\item a surface with vertices with either the Dirichlet or Neumann, but not mixed, boundary condition.
		\end{enumerate} 
		Assume further that $\chi(\Sigma) \leq \chi(\Omega)$.  Then $\Sigma$ and $\Omega$ are not isospectral.
	\end{theorem} 
	
	\begin{proof}  We compute the heat trace coefficient for $\Sigma$, 
		$$a_0 (\Sigma) = \frac{\chi(\Sigma)}{6} - \frac{n}{12} + \sum_{j=1} ^n \frac{-\pi^2 + 2 \alpha_j^2}{48\pi \alpha_j}.$$
		Above, $n$ is the number of vertices, and $\alpha_j$ is the interior angle at the $j^{th}$ vertex.  By the assumption that $\alpha_j < \frac{\pi}{\sqrt 2}$ for all $j$ we have 
		$$a_0 (\Sigma) < \frac{\chi(\Sigma)}{6} - \frac{n}{12}.$$
		On the other hand, if $\Omega$ has smooth boundary and Dirichlet or Neumann boundary condition (not mixed), we have 
		$$a_0 (\Omega) \geq \frac{\chi(\Omega)}{6} \geq \frac{\chi(\Sigma)}{6} > a_0 (\Sigma).$$
		This shows that $\Sigma$ and $\Omega$ are not isospectral.  
		
		In case $\Omega$ has $m$ vertices, and a single fixed boundary condition then 
		$$a_0 (\Omega) \geq  \frac{\chi(\Omega)}{6} - \frac{n}{12}  + \sum_{j=1} ^m \frac{\pi^2 + \beta_j^2}{24 \pi \beta_j}.$$
		Here the interior angle at the $j^{th}$ vertex is $\beta_j$.  By the assumption that $\Omega$ has vertices, at least one $\beta_j \neq \pi$, and therefore 
		$$a_0 (\Omega) > \frac{\chi(\Omega)}{6} \geq a_0 (\Sigma).$$
		Consequently, $\Sigma$ and $\Omega$ are not isospectral.  
	\end{proof}

	We conclude with a familiar example which satisfies the hypotheses of the preceding theorem.  Let us consider domains in the plane which do not have holes.  Let $\Sigma$ be a rectangular domain with the Dirichlet boundary condition on two opposite sides, and Neumann boundary condition on the other two sides; see Figure \ref{fig:squaremix}.  Then, the interior angles are all equal to $\frac \pi 2 < \frac{\pi}{\sqrt 2}$.  Consequently, the theorem shows that such a domain is not isospectral to \em any \em smoothly bounded domain with either Dirichlet or Neumann (but not mixed) boundary condition, nor is it isospectral to \em any \em domain with corners but which has a single fixed boundary condition, either Dirichlet or Neumann.  An analogous result holds for any polygonal domain which has an even number of sides and alternating Dirichlet and Neumann boundary conditions, such that the interior angles do not exceed $\frac{\pi}{\sqrt 2}$.  
	
	\begin{figure} \includegraphics[height=2in]{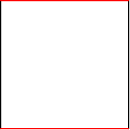} \caption{For a rectangular domain we impose the Dirichlet boundary condition on the red sides and Neumann boundary condition on the black sides.} \label{fig:squaremix}   \end{figure} 
	
	\begin{appendix} 
		\section{Calculation of the Green's function for Dirichlet, Neumann, and mixed Dirichlet-Neumann boundary conditions} \label{app1} 
		
		Here we will explain how to obtain the explicit expression for the Green's functions introduced in Section \ref{Green's functions}. Inspired by Fedosov \cite{fedosov}, we consider the Kontorovich-Lebedev transform
		\begin{equation*}
			F(x)=\int_{0}^{\infty}K_{ix}(z) f(z) \frac{dz}{z}
		\end{equation*}
		and its inverse transform
		\begin{equation*}
			f(y)=\frac{2}{\pi^2}\int_{0}^{\infty}x\sinh(\pi x)K_{ix}(y)F(x)dx.
		\end{equation*}
		Above, $K_\nu$ is the modified Bessel function of second kind. At least formally 
		\begin{align*}
			f(r_0 \sqrt{s} )=\int_{0}^{\infty}\frac{2}{\pi^2r}\int_{0}^{\infty}x\sinh(\pi x)K_{ix}(r \sqrt{s})K_{ix}(r_0 \sqrt{s})dx\cdot f(r \sqrt{s})dr.
		\end{align*}
		Hence, in the distributional sense we obtain
		\begin{equation}\label{delta_for_r}
			\frac{2}{\pi^2r}\int_{0}^{\infty}x\sinh(\pi x)K_{ix}(r\sqrt{s})K_{ix}(r_0\sqrt{s})dx=\delta(r-r_0).
		\end{equation}
		
		We will search for the Green's function of the following form $G(s,r,\phi,r_0,\phi_0)=$
		\begin{equation}\label{G_est}
			\frac{2}{\pi^2}\int_{0}^{\infty}K_{i\mu}(r\sqrt{s})K_{i\mu}(r_0\sqrt{s})\mu\sinh(\pi\mu)\Phi(\mu,\phi,\phi_0)d\mu.
		\end{equation}
		Inserting \eqref{G_est} into \eqref{Green_fnc} and using the definition of $K_\nu$, we want to solve: 

		\begin{multline*}
			\frac{2}{\pi^2}\int_{0}^{\infty}K_{i\mu}(r \sqrt{s})K_{i\mu}(r_0 \sqrt{s})\mu\sinh(\pi\mu)\frac{1}{r^2}\left[-(i\mu)^2\Phi(\mu,\phi,\phi_0)-\Phi''(\mu,\phi,\phi_0)\right]d\mu\\
			=\frac{1}{r}\delta(r-r_0)\delta(\phi-\phi_0).
		\end{multline*}
		By \eqref{delta_for_r}, it will suffice to find $\Phi$ such that
		\begin{multline} \label{eq:app_slp}
			-(i\mu)^2\Phi(\mu,\phi,\phi_0)-\Phi''(\mu,\phi,\phi_0)=\delta(\phi-\phi_0), \\ 
			\alpha\Phi(0)+\beta\Phi'(0)=0, \quad  \alpha\Phi(\gamma)+\beta\Phi'(\gamma)=0, 
		\end{multline}
		with either $\alpha=1$ and $\beta = 0$ or $\alpha = 0$ and $\beta = 1$. Solving the first case will yield \eqref{DirichletGK}, and the second case will yield \eqref{NeumannGK}.  For this purpose, we note that 
		\[ 
		\Phi_1(\phi):=\alpha\sinh \phi\mu-\mu\beta\cosh\phi\mu, \quad 
		\Phi_2(\phi):=\alpha\sinh (\phi-\gamma)\mu-\mu\beta\cosh(\phi-\gamma)\mu 
		\] 
		are solutions of \eqref{eq:app_slp} and satisfy the first and second boundary conditions, respectively. Hence, the Green's function is obtained by inserting 
		\begin{equation} \label{eq:app_green} 
			\Phi:=-\frac{\Phi_1(\phi)\Phi_2(\phi_0)}{W(\Phi_1,\Phi_2)},
			\qquad
			\text{for }\phi<\phi_0,
		\end{equation}
		into \eqref{G_est} where above $W(\Phi_1,\Phi_2)$ is the Wronskian of $\Phi_1$ and $\Phi_2$.
		
		Similarly, with the Dirichlet-Neumann mixed boundary condition, the Green's function is of the form \eqref{G_est}, but in this case $\Phi$ solves 
		\begin{equation*}
			-\Phi''+\mu^2\Phi=0, \qquad \Phi(0)=0, \qquad \Phi'(\gamma)=0.
		\end{equation*}
		Now defining 
		\begin{equation*}
			\Phi_1(\phi):=\sinh(\phi\mu), 
			\qquad
			\Phi_2(\phi):=\cosh((\phi-\gamma)\mu)
		\end{equation*}
		the Green's function is obtained from \eqref{eq:app_green} inserted into \eqref{G_est}.

		\section{The corner contribution using Cheeger's series expression for the heat kernel on an infinite sector}  \label{s:altcalcs} 
		Here we show how, given the contribution for a D-D corner, which has been computed in \cite{vdbs} and \cite{nrs1}, we may use the series expression of the heat kernel from \cite{cheeger} to compute the contribution for both N-N and D-N corners.  For a corner of angle $\alpha$, the corner contribution is obtained by computing the renormalized integral, (see also \cite{hmm}, \cite{hassell}) 
		\[\textrm{f.p.}_{\epsilon=0}\int_{R=0} ^{1/\epsilon} \int_{0}^{\alpha}\frac 12R\exp[-\frac 12R^2]\sum_{j=1}^{\infty}I_{\mu_j}(\frac 12R^2)|\phi_j(\theta)|^2\, d\theta\, dR.\]
		Since the cross-sectional eigenfunctions, $\phi_j$, have unit $\cL^2$ norm, this simplifies to 
		\begin{equation}\label{eq:gencornercontrib}
			\textrm{f.p.}_{\epsilon=0}\int_{R=0} ^{1/\epsilon} \frac 12Re^{-\frac 12R^2}\sum_{j=1}^{\infty}I_{\mu_j}(\frac 12R^2)\, dR.
		\end{equation}

		In the Dirichlet-Dirichlet case, $\mu_j=j\pi/\alpha$ and \eqref{eq:gencornercontrib} becomes
		\begin{equation}\label{eq:dircornercontrib}
			\textrm{f.p.}_{\epsilon=0}\int_{R=0} ^{1/\epsilon}\frac 12Re^{-\frac 12R^2}\sum_{j=1}^{\infty}I_{j\pi/\alpha}(\frac 12R^2)\, dR.
		\end{equation}
		By \cite{vdbs} and \cite{nrs1},  \eqref{eq:dircornercontrib} is equal to
		\begin{equation}\label{eq:vdbs}
			\frac{\pi^2-\alpha^2}{24\pi\alpha}.
		\end{equation}
		
		In the Neumann-Neumann case, the only difference is that there is a zero eigenvalue. So the difference of the corner contributions in the D-D and N-N cases is
		\begin{equation}\label{eq:dirneumdifference}
			\textrm{f.p.}_{\epsilon=0}\int_0^{1/\epsilon}\frac 12Re^{-\frac 12R^2}I_0(\frac 12R^2)\, dR.
		\end{equation}
		This integral may be evaluated directly.  First make a substitution in the integral setting $u=\frac 12 R^2$, so that it becomes 
		$$\textrm{f.p.}_{\epsilon=0}\int_0^{\frac{1}{2\epsilon^2}}\frac 12 e^{-u} I_0(u)\, du.$$
		In \cite[5.5]{polyakov}, a primitive for the integrand is obtained, 
		$$g(u) := e^{-u} u (I_0(u) + I_1 (u)) \implies g'(u) = e^{-u} I_0 (u).$$
		Since $I_0(0) = I_1(0) = 0$, the integral above is therefore 
		\begin{equation}
			\textrm{f.p.}_{\epsilon=0} \frac 1 2 \left[\frac{1}{2\epsilon^2} e^{-\frac{1}{2\epsilon^2}} \left(I_0 \left(\frac{1}{2\epsilon^2}\right)+I_1\left(\frac{1}{2\epsilon^2}\right) \right) \right].
		\end{equation}
		
		However, both $I_0(z)$ and $I_1(z)$ have expansions of the form
		\[e^z z^{-1/2}(C_0+C_1z^{-1}+C_2z^{-2}+\dots)\]
		as $z\to\infty$. Substituting these expansions for the Bessel functions above, there are only odd powers of $\epsilon$ in the expansion. Therefore the coefficient of the $\epsilon^0$ power is zero, and the finite part is zero. This shows that \eqref{eq:dirneumdifference} equals zero and the Neumann-Neumann corner contribution is \eqref{eq:vdbs}, the same as the Dirichlet-Dirichlet contribution.
		
		In the Dirichlet-Neumann cases, $\mu_j=(j+1/2)\pi/\alpha$, starting at $j=0$, so we get
		\begin{equation}\label{eq:dirneumcornercontrib}
			\textrm{f.p.}_{\epsilon=0}\int_0^{1/\epsilon}\frac 12Re^{-\frac 12R^2}\sum_{j=0}^{\infty}I_{(j+1/2)\pi/\alpha}(\frac 12R^2)\, dR.
		\end{equation}
		Observe that, since renormalized integrals are linear, this equals
		\begin{equation}\label{eq:simplified} \begin{gathered} 
				\textrm{f.p.}_{\epsilon=0}\int_0^{1/\epsilon}\frac 12Re^{-\frac 12R^2}\sum_{j=1}^{\infty}I_{j\pi/(2\alpha)}(\frac 12R^2)\, dR \\ 
				-
				\textrm{f.p.}_{\epsilon=0}\int_0^{1/\epsilon}\frac 12Re^{-\frac 12R^2}\sum_{j=1}^{\infty}I_{j\pi/\alpha}(\frac 12R^2)\, dR.
			\end{gathered} 
		\end{equation}
		We recognize this to be \eqref{eq:dircornercontrib} for angle $2\alpha$ minus \eqref{eq:dircornercontrib} for angle $\alpha$. We  therefore obtain the Dirichlet-Neumann mixed corner contribution is
		\begin{equation}\label{eq:vdbsDN}
			\frac{\pi^2-(2\alpha)^2}{48\pi\alpha}-\frac{\pi^2-\alpha^2}{24\pi\alpha}=\frac{-\pi^2-2\alpha^2}{48\pi\alpha}.
		\end{equation}
		
	\end{appendix}

\end{document}